\documentclass[amssymb,amsfonts, final]{amsart}
\usepackage[T1]{fontenc}
\usepackage[intlimits,sumlimits,noDcommand,narrowiints]{kpfonts}
\usepackage[english]{babel}
\usepackage{mathtools}
\usepackage{slashed,accents}
\usepackage{booktabs, multirow, colortbl}
\usepackage{eqlist}

\usepackage[notref,notcite]{showkeys}

\usepackage[pdfusetitle]{hyperref}

\theoremstyle{plain}
\newtheorem{thm}{Theorem}[section]
\newtheorem{prop}[thm]{Proposition}
\newtheorem{lem}[thm]{Lemma}
\newtheorem{cor}[thm]{Corollary}
\theoremstyle{definition}
\newtheorem{defn}[thm]{Definition}

\newtheorem{notn}[thm]{Notation}
\theoremstyle{remark}
\newtheorem{rmk}[thm]{Remark}

\newcommand*\Real{\mathbb{R}}

\DeclarePairedDelimiterX\set[2]\lbrace\rbrace{#1\;\delimsize\vert\;#2}
\newcommand*\eqdef{\overset{\mbox{\tiny{def}}}{=}}
\DeclarePairedDelimiterXPP\www@innerprod[3]{}\langle\rangle{_{#1}}{#2,#3}
\newcommand*{\innerprod}[3][]{\www@innerprod*{#1}{#2}{#3}}
\newcommand*\ub[1]{\smash{\underline{#1}}}

\newcommand*\supp{supp}
\newcommand*{\dvol}{\mathrm{dvol}}
\newcommand*{\D}{\mathrm{d}}

\newcommand*\bgf{\Upsilon}   
\newcommand*\mink{\mathfrak{m}} 
\newcommand*\secff{\mathit{II}} 
\newcommand*\gmetr{g} 
\newcommand*\gmetrf{\accentset{\circ}{g}} 
\newcommand*\confmetr{\tilde{g}} 
\newcommand*\sigmmetr{\eta} 
\newcommand*\siggmetr{h} 
\newcommand*\seten{\mathcal{Q}} 
\newcommand*\sweight{\mathfrak{W}} 
\newcommand*\algcomm{\mathfrak{A}} 
\newcommand*\algdiff{\mathfrak{B}} 
\newcommand*\lieD{\mathscr{L}} 
\newcommand*\pwWf{\mathscr{P}} 
\newcommand*\swtf{\mathscr{W}} 
\newcommand*\smtf{\mathscr{G}}  
\newcommand*\algdo{\mathscr{B}} 
\newcommand*\Energy{\mathcal{E}}  

\makeatletter
\newcommand{\vast}{\bBigg@{4}}
\newcommand{\Vast}{\bBigg@{5}}
\makeatother

\numberwithin{equation}{section}

\begin{document}
\title{Global nearly-plane-symmetric solutions to the membrane equation}
\author{Leonardo Abbrescia}
\address{Department of Mathematics, Michigan State University, East Lansing, Michigan, USA}
\email{leonardo@math.msu.edu}
\thanks{L Abbrescia is supported by an NSF Graduate Research Fellowship (DGE-1424871).}
\author{Willie Wai Yeung Wong}
\address{Department of Mathematics, Michigan State University, East Lansing, Michigan, USA}
\email{wongwwy@math.msu.edu}
\thanks{WWY Wong is supported by a Collaboration Grant from the Simons Foundation, \#585199.}
\subjclass[2010]{35L72, 35B35; 35C07, 53A10}

\begin{abstract}
	We prove that \emph{any} simple planar travelling wave solution to the membrane equation in spatial dimension $d \geq 3$ with bounded spatial extent is globally nonlinearly stable under sufficiently small compactly-supported perturbations, where the smallness depends on the size of the support of the perturbation as well as on the initial travelling wave profile. The main novelty of the argument is the lack of higher-order peeling in our vector-field based method. In particular, the higher order energies (in fact, all energies at order $2$ or higher) are allowed to grow polynomially (but in a controlled way) in time. This is in contrast with classical global stability arguments where only the ``top'' order energies used in the bootstrap argument exhibit growth, and reflects the fact that the background travelling wave solution has ``infinite energy'' and the coefficients of the perturbation equation are not asymptotically Lorentz invariant. Nonetheless, we can prove that the perturbation converges to zero in $C^2$ by carefully analyzing the nonlinear interactions and exposing a certain ``vestigial'' null structure in the equations. 
\end{abstract}

\maketitle

\tableofcontents

\section{Introduction}

\subsection{Membrane equation}
The starting point of our discussion is the equation
\begin{equation}\label{eq:membrane}
	\frac{\partial}{\partial x^\mu} \left( \frac{\mink^{\mu\nu} \partial_\nu \phi}{\sqrt{1 + \mink(\nabla\phi,\nabla\phi)}} \right) = 0
\end{equation}
on $\Real^{1,d}$, the $(1+d)$ dimensional Minkowski space equipped with the metric $\mink$ which in standard coordinates is given by the diagonal matrix $\mathrm{diag}(-1,1,1,\cdots,1)$. In the equation we used the notation  $\mink(\nabla\phi,\nabla\psi) \eqdef \mink^{\mu\nu} \partial_\mu\phi \partial_\nu\psi$. 
This equation is variously known as the \emph{membrane equation}, the \emph{time-like minimal/maximal surface equation}, or the \emph{Lorentzian vanishing mean curvature flow}. 
This is due to the interpretation that the graph of $\phi$ in $\Real^{1,d} \times \Real \cong \Real^{1,d+1}$ is an embedded time-like hypersurface with zero mean curvature. 

Solutions to \eqref{eq:membrane} model extended test objects (world sheets), in the sense that the case where $d = 0$ reduces to the geodesic equation which models the motion of a test particle. 
(The membrane equation can also be formulated with codimension greater than one; see \cite{AlAnIs2006, Milbre2008}.) 
The membranes can also interact with external forces which manifests as a prescription of the mean curvature; see \cite{AurChr1979, Hoppe2013, Kibble1976, VilShe1994} for some discussion of the physics surrounding such objects, and see \cite{Jerrar2011, Neu1990} for rigorous justifications that membranes represent extended particles. 

Our interest in the membrane equation arose, however, mainly due to it being an exceptional model of a quasilinear wave equation that is highly \emph{non}-resonant.
The exploration of resonant conditions in wave equations proceeded, historically, through two fronts. 
In the case of 1 spatial dimension, it has long been understood that hyperbolic systems with resonance (Lax's ``genuinely nonlinear condition'') lead to shock formation in finite time \cite{Lax1964, Lax1973, John1974}.
For higher spatial dimensions, in the small-data regime, resonance has to compete with the dispersive decay enjoyed by wave equations. 
By now it is well understood that quasilinear wave equations enjoy small-data global existence in dimension $d \geq 4$, and also in dimensions $d = 2,3$ when versions of Klainerman's null condition are satisfied \cite{Klaine1980, Klaine1982, Klaine1984a, Alinha2001, Alinha2001a}. 
More recently the two fronts have met, where small-data shock formation for resonant quasilinear wave equations have been studied in spatial dimensions 2 and 3 \cite{Alinha2001, Alinha2001a, Christ2007a, MR3561670, MR3858399}.
For a recent review of the current understanding of small-data global existence versus shock formation in quasilinear waves, please see \cite{HoKlSW2016}.

In a recent paper, the second author, together with Speck, Holzegel, and Luk, studied the stability of plane-symmetric shock formation for quasilinear wave equations with resonance, under initial data perturbations that breaks the plane-symmetry \cite{SpHoLW2016}. 
More precisely, we start with a background simple-plane-symmetric solution to a quasilinear wave equation that is genuinely nonlinear, such that it forms a shock singularity in finite time. Such background solutions can be extracted from, for example, the late-time evolution of any small compactly supported initial data; we however allow our background solution to be of arbitrary ``size''. 
We were able to show that the shock formation is stable under arbitrary initial data perturbations that breaks the simple-plane-symmetry, provided that the perturbation is small compared to the background solution. 

A natural follow-up question is: \emph{when genuine nonlinearity fails, in particular when there exists simple-plane-symmetric global solutions to the quasilinear wave equation, is the global existence stable under small, symmetry-breaking initial data perturbations?}

Returning to the membrane equation, we note that the equation is highly non-resonant. It satisfies a stronger null condition than is typical of quasilinear waves in 2 or 3 dimensions. 
This was explicitly exploited to show global well-posedness of the small-data problem first by Brendle \cite{Brendl2002} when $d = 3$ and then by Lindblad \cite{Lindbl2004} in $d = 2$ \emph{and} $d = 1$. 
The $d = 1$ case is surprising as, there being no dispersive decay for the one-dimensional wave, any resonance, even arbitrarily high order, can lead to finite-time blow-up. 
The second author explored this case in more detail geometrically \cite{Wong2017} and enlarged the class of initial data for which global existence holds. 

Our focus on the membrane equation in this paper then is due to the fact that (i) as a consequence of \cite{Lindbl2004} and \cite{Wong2017}, there exists robust families of global plane-symmetric solutions to the membrane equation, and (ii) the null geometry of such solutions are well understood by the analyses of \cite{Wong2017}.
We remark that, while not explicitly stated, following the same method of proof of the main theorem in \cite{Wong2017}, one can show that the global simple plane-wave solutions described below in Section \ref{sect:pwbkgd} are automatically stable under plane-symmetric perturbations that are \emph{not necessarily simple}. 
In a future work the authors intend to generalize the results of this paper to more general models of quasilinear wave equations with strong null conditions. 

We state and prove our main result in dimension $d = 3$; as described in the previous paragraph, the result is effectively known in $d = 1$. Our proof also works in all dimensions $d \geq 3$ thanks to the improved dispersive decay of solutions to the linear wave equation in higher spatial dimensions. 
Our proof however doesn't work in $d = 2$ due to certain technical losses of decay (see Remark \ref{rmk:semi:two} below). 
In \cite{LiuZhou2019p} the authors were able to prove a similar result in $d = 2$ with weaker asymptotic control; see Remark \ref{rmk:dimen} for further discussion.

One should note, at this juncture, that the non-resonance of the membrane equation is only effective at preventing a certain type of singularity formation. 
Indeed, far away from the nearly-simple-plane-wave regime that we consider in the present manuscript, singularities are known to arise from regular initial data. 
In the case where $d = 1$ these were analyzed by Nguyen and Tian \cite{NguTia2013} and Jerrard, Novaga, and Orlandi \cite{JeNoOr2015}; while their analyses concentrate on the case with spatially periodic domain, by finite speed of propagation the same singularity formation can be localized and placed in our context. Analogues of \cite{NguTia2013,JeNoOr2015} in higher spatial dimensional backgrounds were studied by the second author \cite{Wong2018Axi}. 
In these cases the singularities are \emph{not} of shock-type, but rather appear due to the degeneration of the principal symbol of the evolution. 

\subsection{Our main result and discussions}
The answer to the question asked in the previous section is in the affirmative: we show that simple-plane-wave solutions to the membrane equation are stable under small initial data perturbations. 
The precise version of our main theorem is Theorem \ref{thm:mainthmQ}; there we state the result as a small-data global existence result for the corresponding perturbation equations, after a nonlinear change of independent variables that corresponds to a gauge choice. 
Here we state a slightly less precise version in terms of the original variables. 

\begin{thm}\label{thm:imprecise}
	Fix the dimension $d = 3$. Let $\bgf$ denote a smooth simple-plane-symmetric solution to \eqref{eq:membrane} with finite extent in its direction of travel. Fix a bounded set $\Omega\subset \Real^3$. There exists some $\epsilon_0 > 0$ depending on the background $\bgf$ and the domain $\Omega$, such that for any $(\psi_0, \psi_1)\in (H^{5}(\Real^3)\cap C^\infty_0(\Omega)) \times (H^4(\Real^3)\cap C^\infty_0(\Omega))$ with $\|(\psi_0, \psi_1)\| < \epsilon_0$, the initial value problem to \eqref{eq:membrane} with initial data 
	\[ \phi(0,x) = \bgf(0,x) + \psi_0(x), \qquad \partial_{t}\phi(0,x) = \partial_{t}\bgf(0,x) + \psi_1(x) \]
	has a global solution that converges in $C^2(\Real^3)$ to $\bgf$ as $t\to \pm\infty$. 
\end{thm}

\begin{rmk}[Finite extent in the direction of travel]
	We ask that $\bgf$ essentially represent a travelling ``pulse''. For example, taking plane-symmetry to mean constant in the $x^2$ and $x^3$ variables, $\bgf$ would be a function of $t = x^0$ and $x^1$ alone. We ask that for any fixed $t$ the function $\bgf$ vanishes for all sufficiently large $x^1$. 

	While we make heavy use of this finite extent property in the course of the proof (see Lemma \ref{lem:fderest}), as we discuss in Remark \ref{rmk:cmpctdecay}, the finite-extent property can be replaced by quantitative decay rates on the background profile $\bgf$ up to $7$ derivatives. We omit this generalization as it makes the arguments more tedious and the main mechanisms less transparent. 
\end{rmk}

\begin{rmk}[Simplicity]
	By a simple plane-wave solution we refer to a solution that is not only constant in the $x^2$ and $x^3$ variables, but one such that the differential $\D\bgf$ is null with respect to the dynamic metric. In other words, a simple plane-wave solution is one that propagates along only one (and not both) of the characteristic directions of the nonlinear wave equation. 

	The assumption of simplicity is only to keep the argument simple (pun intended). In fact, assuming finite extent of the initial data for the plane-symmetric background, automatically by the sharp Huygen's principle for one dimensional waves, after a finite-length of time the background will decouple into two spatially disjoint simple plane-waves travelling in opposite directions. By Cauchy stability of the finite-time initial value problem, we see that the theorem for the simple plane-wave background also implies the theorem for general, globally existing plane-symmetric backgrounds such as those demonstrated to exist in \cite{Lindbl2004, Wong2017}.

	We note here, however, that another feature of simplicity is that simple-plane-wave solutions exist for arbitrary pulse profile (see Section \ref{sect:pwbkgd} below). The same is not the case for non-simple plane-wave solutions: large interacting waves can form finite-time singularities. 
\end{rmk}

\begin{rmk}[Dimensionality] \label{rmk:dimen}
	The theorem above is stated for $d = 3$. The same arguments can be used to prove stability for all dimensions $d \geq 3$ (in fact the arguments can be significantly further simplified when $d \geq 5$). 
	One needs to modify the degree of regularity required. When $d = 3$ the data is taken to be small in $H^k \times H^{k-1}$ with $k = 5$. When $d \geq 4$ is even we will need $k = d+3$, and when $d \geq 5$ is odd we will need $k = d+2$. 
	Compare to the discussion in Section \ref{sect:semilinear} below.

	As mentioned before in this introduction, the $d = 1$ analogue of the result essentially follows from the arguments in \cite{Wong2017}. This leaves the case $d = 2$. 
	After circulating our preprint, we were informed by Profs.\ Jianli Liu (Shanghai University of China) and Yi Zhou (Fudan University) of their work on the $d = 2$ case \cite{LiuZhou2019p}. They were very kind to share with us a draft of their manuscript, which adopted a somewhat different approach to the problem. Aside from minor technical differences in how we approach the energy and pointwise estimates, a difference appears in how we linearize around the background solution. In the present manuscript we used the geometric normal graphical gauge (see below) adapted to the background traveling wave, while in \cite{LiuZhou2019p} they used the gauge adapted to the trivial solution. Our gauge has the advantage that the perturbation equations contain no \emph{linear} potential from the background; the price paid being the appearance of nonlinear contributions of lower order whose null structure are less apparent. In \cite{LiuZhou2019p}, they were so far able to show global existence for the perturbation equations but only $C^0$ (and \emph{not} $C^1$) convergence to the background. The lack of higher-derivative convergence can be attributed, at least in part, to their gauge choice. Based on our own work we have high hopes that in fact $C^2$ convergence can be proven to hold, though at present there are some technical difficulties for even showing global existence using a direct extension of our method; see also Remark \ref{rmk:semi:two}. 
\end{rmk}

Our main theorem is not a straight-forward small-data global existence result for a quasilinear wave equation. 
The equations satisfied by the perturbations around \emph{large} solutions generally include coefficients contributed by background, the effects of which must also be captured.
In our problem, to leading order the perturbation equation looks like
\begin{equation}\label{eqn:modelprob:intro}
\Box_\mink \phi + \phi \bgf'' (\partial_t + \partial_{x^1})^2 \phi + \bgf'' (\partial_t \phi + \partial_{x^1} \phi)^2 = 0.
\end{equation}
Here $\Box_\mink$ is the flat wave operator, and the background pulse is assumed to be travelling in the $+x^1$ direction, so has compact support in the $(t-x^1)$ variable. 

The first thing to notice is that the linearized equation is the linear wave equation on Minkowski space. 
This is a special geometric feature of simple-travelling wave solutions to the membrane equation.
To expose this special linear structure, one needs to make an appropriate gauge choice involving a nonlinear change of variables adapted to the background $\bgf$, which essentially re-writes our perturbation equations as a graph in the normal bundle of $\bgf$, interpreted as a submanifold of $\Real^{1,1+d}$. 
It is well-known that the membrane equation has good structure in such ``normal graphical gauge'': in this formulation the linearized equation can be expressed as the geometric wave operator adapted to the induced Lorentzian metric on the background $\bgf$, plus possibly a potential term. 
This gauge was also used, for example, in \cite{DoKrSW2016}. 

In view of this special geometric feature, we do not need to develop special methods to perform the linear analysis. 
On the other hand, the function $\bgf''$ is \emph{non-decaying} and has support within the ``wave zone''; this significantly complicates the analysis of the nonlinear terms, especially since these nonlinearities are not in the shape of classical null forms. 
This is in contrast with the analyses in \cite{DoKrSW2016} where the stability of another ``large data'' solution to the membrane equation was considered. 
The background solution in that case is the static catenoid solution.
The nontrivial catenoid background introduced a low-frequency correction to the linear evolution (in fact giving an exponentially growing mode). 
But as the background is asymptotically flat, the high-frequency evolution, especially in the wave-zone where it is the most delicate when it comes to the nonlinear interactions, is entirely captured by classical null structures. 
In particular the nonlinearities do not introduce new difficulties beyond the adjustments made for the modified linear evolution. Another difference with our work and \cite{DoKrSW2016} is that they prove that the catenoid is globally stable under \emph{axially symmetric} codimension one initial perturbations, whereas we prove that our plane-wave solution is globally stable under an open set of \emph{symmetry breaking} perturbations. Their symmetry assumptions on the perturbations are there to avoid the issue of trapped geodesics on the catenoid.

In the present manuscript, on the other hand, the focus is entirely on the nonlinearity, with the main difficulty arising precisely from the non-decaying background $\bgf''$. 
At this point it may be worth drawing comparison to another large-data (semi-)global existence result for the membrane equation. 
In \cite{Jinhua}, the authors studied the membrane equation with initial data given as a small perturbation of an out-going ``short-pulse''. 
The (semi-)global existence (note that by their choice of initial data, the result in \cite{Jinhua} is not time-symmetric!) mechanism in this case is essentially still the classical null condition of Klainerman. 
The strong non-resonance condition of the membrane equation means that the ``large'' short-pulse background does not interact with itself; and in fact the pulse itself \emph{decays} like the solution to the linear wave equation. 
Putting this together with the fact that the nonlinearities in \eqref{eq:membrane} are cubic, this means that heuristically we can understand the result of \cite{Jinhua} as very similar to the large data stability result for the wave maps system proven in \cite{Sideri1989}, which also required the ``background geodesic solution'' to be one with finite (weighted) energy, and hence decays like finite energy solutions to the linear wave equation. 
These types of systems can be modeled by the quasilinear system
\begin{gather*}
	\Box_\mink \psi_1 = 0, \\
	\Box_\mink \psi_2 = \mink(\nabla\psi_1, \nabla\mink(\nabla\psi_2, \nabla\psi_2)) + \mink(\nabla\psi_2, \nabla\mink(\nabla\psi_1, \nabla\psi_2)).
\end{gather*}
Even when $\psi_1$ is a ``large'' solution, it contributes enough decay that the nonlinearities for the second equation decay at an integrable rate. Together with the fact that the nonlinearity is quadratic in $\psi_2$, we can upgrade the smallness and close the bootstrap. Note that the decay of $\psi_1$ is crucial, as, in the second term of the nonlinearity we see components like
\[   (\partial_t + \partial_r)^2\psi_1 \cdot (\partial_t \psi_2 - \partial_r \psi_2)^2.\]
This is a resonant interaction in $\psi_2$, whose contribution is significantly ameliorated by the fact that $(\partial_t + \partial_r)^2\psi_1$ should decay like $t^{-3/2}$ (or better) in $\Real^{1,3}$ or $t^{-1}$ (or better) in $\Real^{1,2}$. 
If we were to replace the $\psi_1$ factor by a generic bounded function in $\Real^{1,3}$ (or a function decaying no faster than $1/\sqrt{t}$ in $\Real^{1,2}$) this term will lead to finite-time blow-up. 

Returning to our equation \eqref{eqn:modelprob:intro}, we see that we have precisely this type of resonant interaction with a non-decaying coefficient. 
Instead of coefficient decay, we need to exploit a different aspect of the null structure of the original membrane equation \eqref{eq:membrane}. 
What we will use is the fact that $\bgf''$ has compact support in the $(t - x^1)$ variable. 
The resonant interacting terms $(\partial_t\phi + \partial_{x^1}\phi)$ represent waves traveling in directions transverse\footnote{In $(1+1)$-dimensions, the linear wave equation can be expressed as $(\partial_t - \partial_{x^1})(\partial_t\varphi + \partial_{x^1}\varphi) = 0$, and since $(\partial_t - \partial_{x^1})(t-x^1) = 2$, one sees that $\partial_t \varphi + \partial_{x^1}\varphi$ is a bonafide traveling wave transverse to  the level sets of $t-x^1$.} to the level sets of $t-x^1$. In particular, we expect that the resonant interaction to only take place for a bounded length of time (for each wave packet). 
Our main mechanism would therefore be something similar to that which drives Shatah's space-time resonance arguments \cite{Shatah2010}, but captured in a purely physical space manner. 

Of course, we have to pay a price for this non-decay. 
This manifests in us having to use a polynomially-growing energy hierarchy when using the vector field method.
In fact, our higher order energies, starting with the second (controlling the third derivatives in $L^2$), will grow in time, with each additional order differentiation growing one order faster in time.
One should compare to classical applications of the vector field method where all but the top-order energies are bounded in time, with the top-order typically exhibiting no worse than a log growth. 
The upshot of this energy hierarchy is that we lose strong peeling properties of the solutions. (See Remark \ref{rmk:lostofpeeling}.)

To effectively study this energy hierarchy, it turns out to be convenient to use \emph{hyperboloidal foliations}.
Such foliations were introduced by Klainerman \cite{Klaine1985a} to study the decay properties of Klein-Gordon equations, and further developed and refined in \cite{LeFloch} for use also with wave equations and coupled systems of Klein-Gordon and wave equations.
We will follow the formulation in \cite{Wong2017p} which places emphasis on the use of Lorentz boost commutators; in fact, one interesting technical facet of our argument is that, throughout, we will only commute our equations with Lorentz boost vector fields. 
To help manage the nonlinearities that arise in such arguments in a systematic way, we introduce in this paper a weighted vector field algebra (see Section \ref{sect:wvfalg}). 
The introduced notations help simplify the computation vastly for the higher-order nonlinearities and significantly shorten our arguments. 


\subsection{Outline of the paper}

The remainder of this paper is organized as follows: we first discuss the background plane-wave solutions $\bgf$. 
These solutions are introduced in Section \ref{sect:pwbkgd}. Their basic geometric properties and our gauge choice for studying the perturbations are described in Section \ref{sect:perturb}. 

We next discuss the basic analytic tools used in our arguments; in Section \ref{sect:globsob} we recall the global Sobolev inequalities associated to hyperboloidal foliations, in Section \ref{sect:wvfalg} we develop a weighted vector field algebra to help simplify our analyses of the nonlinear terms using more schematic notations. 

In Section \ref{sect:semilinear} we study the semilinear model problem $\Box_\mink \phi = \bgf'' (\partial_{\ub{u}}\phi)^2$, obtained from dropping the quasilinearity from \eqref{eqn:modelprob:intro}.
This model problem turns out to capture already the majority of the difficulty one faces when analyzing the full problem. 
We prove small-data global wellposedness for the semilinear model in all dimensions $\geq 3$. 
There are certain additional technical difficulties for studying the quasilinear model \eqref{eqn:modelprob:intro} in $d = 2$ due to the fact one expects even the \emph{first order energy} exhibits polynomial growth there, and the loss seems too strong to overcome with the methods described in this paper; therefore we also omit a detailed treatment of the $d=2$ semilinear model.

The remainder of the paper is devoted to studying the quasilinear problem in $d = 3$, and stating and proving a more precise version of 
Theorem \ref{thm:imprecise}. 
In Section \ref{sect:commutedeq} we perform first some preliminary computations casting the equations for the perturbation $\phi$ and its higher order derivatives in schematic form to prepare for analysis. As many of the computations are long and involved, we delegate sketches of the arguments separately to the Appendix. 
At the end of the section we state our Main Theorem \ref{thm:mainthmQ}.
As usual, we will prove our Main Theorem by a bootstrap argument for our energy hierarchy. 
In Section \ref{sect:energyBA} we define our energy quantities, outline our main energy estimate, state our bootstrap assumptions, and derive some immediate consequences that do not involve the equations of motion. 
Section \ref{sect:inhom} is devoted to proving \textit{a priori} estimates for our equations of motion, based on the bootstrap assumptions. 
These are combined in Section \ref{sect:closing} to show that the bootstrap assumptions can be improved, and thereby hold for all time and global existence follows. 

For convenience we include in the Appendix a list of notations that are introduced and references to their definitions.

\section{The background solution}

In this section we first exhibit simple plane-wave solutions to the membrane equation, and describe their geometry.
These solutions are traveling waves and exist for all time; our goal is to analyze their stability under small non-plane-symmetric perturbations. 
To do so we recast the stability problem as a small-data Cauchy problem for the perturbation. 
In the second part of this section we exploit the geometric interpretation of the solutions as minimal submanifolds of higher dimensional Minkowski space to make a convenient choice of gauge, and derive the corresponding perturbation equations.
The gauge choice allows us to simplify the analysis of the linearized dynamics. 
As the membrane equation itself is a quasilinear wave equation, when linearizing around a fixed nontrivial background solution, typically the background contributes to the linearized dynamics (e.g.\ in \cite{DoKrSW2016} where the background contributes a potential term leading to generic instability of the system). 
For the membrane equation in Minkowski space, however, it is known \cite{Choque1976} that the potential term in the linearized dynamics for perturbations parametrized by the normal bundle is given by the double contraction of the extrinsic curvature of the embedding of the background solution. 
For simple plane-waves, this potential term vanishes \cite{Wong2017}. 
The gauge choice below makes this explicit and shows that the perturbed system can be described by a \emph{quasilinear perturbation} of the linear wave equation on Minkowski space, with the background solution only appearing as \emph{coefficients of the nonlinearity}. 

\subsection{Simple plane-wave solutions to the membrane equation}\label{sect:pwbkgd}
Let $\bgf\in C^\infty(\Real;\Real)$ be arbitrary. One easily sees that the function $\mathring \phi : \Real^{1+d} \to \Real$ defined by\footnote{We've made the choice to have our background travelling waves move ``to the left'', i.e.\ as a function of $t+x^1$. Note that for the analyses in \cite{SpHoLW2016} the simple waves move ``to the right''. We beg those readers familiar with the previous work to indulge us and mentally reorient the space-time and relabel the function $u$ as needed.}
\begin{equation}\label{eq:pwbkgd}
	\mathring \phi(t, x^1, x^2, \ldots, x^d) = \bgf(t+x^1)
\end{equation}
solves \eqref{eq:membrane}, seeing as $\D\mathring\phi(t,x^1, \ldots) = \bgf'(t+x^1)~\D(t+x^1)$ and hence $\mink(\D\mathring\phi,\D\mathring\phi) \equiv 0$ and $\mink^{\mu\nu} \partial^2_{\mu\nu} \mathring\phi \equiv 0$. The simple plane-wave background will be interpreted as the graph of $\mathring \phi$ in $\Real^{1,d+1}$, the $(d+2)$-dimensional Minkowski space equipped with the standard Minkowski metric $M$. That is to say, we consider the embedding $\Real^{1+d} \hookrightarrow \Real^{1,d+1}$, given by 
\[ (t, x^1, \ldots, x^d) \mapsto (t, x^1, \ldots, x^d,\mathring \phi(t, x)) \]
with the first component fixed as the timelike one. 
The $\mink(\D\phi,\D\phi) \neq -1$ implies that the induced metric on the graph of $\phi$ is Lorentzian and non-degenerate. By the analysis of \cite{Wong2017} this induced metric is \emph{flat}; this fact can also be seen through the following explicit computations.

Denoting the above embedding by $\Phi$, the induced metric can be in fact given by the line element
\begin{multline*}  \Phi^* M = \D s^2 = (-1 + (\partial_t\mathring\phi)^2) ~\D t^2 + 2 \partial_t\mathring\phi \partial_{x^1}\mathring\phi ~\D t\, \D x^1 + (1 + (\partial_{x^1} \mathring\phi)^2) ~\D (x^1)^2 \\
	+ \D(x^2)^2 + \cdots + \D(x^d)^2.
\end{multline*}
Using that $\partial_t\phi(t,x) = \partial_{x^1} \phi(t,x) = \bgf'(t+x^1)$, we see that if we define 
\begin{equation}\label{eq:changevar}
	\vast\{ \quad \begin{aligned}
	u &\eqdef t+ x^1 , \\
	\ub{u} &\eqdef \frac12 \Big[ t- x^1  - \int_0^{x^1 + t} (\bgf')^2(\tau) ~\D \tau\Big];
\end{aligned}
\end{equation}
that the line element can be alternatively written as the Minkowski metric in standard double-null form
\begin{equation}\label{eq:minkDN}
\mink =  \D s^2 = - 2 ~\D u\, \D\ub{u} + \D(x^2)^2 + \cdots + \D(x^d)^2.
\end{equation}
The functions $u$ and $\ub{u}$ solve the eikonal equation $\mink(\nabla u, \nabla u) = \mink(\nabla \ub{u}, \nabla \ub{u}) = 0$. For the subsequent analyses we will parametrize using the coordinates $\{ u, \ub{u}, x^2, \ldots, x^d\}$.

\begin{rmk}\label{remark:twoMinkMetrics}
Note that there are two Minkowski metrics involved in the construction: (1) The metric on the ambient space $\Real^{1,d+1}$, which is denoted by $M$. (2) The induced Minkowski metric on the plane-wave background given by double-null coordinates $(u,\ub u, \hat x) \in \Real^{1,d}$, denoted by $\mink$.
\end{rmk}

For completeness, we note that the change of variables can be inverted:
\begin{equation}\label{eq:invchgvar}\Vast\{ \quad \begin{aligned}
	t & =  \frac12 u + \ub{u} + \frac12 \int_0^{u} (\bgf')^2(\tau) ~\D\tau,\\
	x^1 &= \frac12 u - \ub{u} - \frac12 \int_0^{u} (\bgf')^2(\tau) ~\D\tau.
\end{aligned}\end{equation}
For convenience, we note that relative to this coordinate system, our simple plane-wave solution is given by the embedding
\begin{equation}\label{eq:pwbkgd2}
	(u, \ub{u}, \hat{x}) \mapsto (t, x^1, \hat{x}, \bgf(u) ) \in \Real^{1, d + 1}
\end{equation}
where $t$ and $x^1$ are given as functions of $u, \ub{u}$ by \eqref{eq:invchgvar}, and for convenience we denote by $\hat{x} = (x^2, \ldots, x^d)$. 

We finish this subsection by computing the extrinsic curvature (second fundamental form) of the embedding \eqref{eq:pwbkgd2}. The change of variables \eqref{eq:invchgvar} implies that the vector fields
\[ \partial_{\ub{u}} = (1,-1,0, \ldots, 0), \qquad \partial_{u} = \big( \frac12 (1 + \bgf'(u)^2) , \frac12( 1 - \bgf'(u)^2), 0, \ldots, 0, \bgf'(u) \big).
\]
Denote by $n: \Real^{1,d} \to \Real^{1,d+1}$ the unit normal vector field (with respect to the Minkowski metric on $\Real^{1,d+1}$) of the embedding \eqref{eq:pwbkgd2} given by 
\begin{equation}\label{eq:unitnormal}
	n(u, \ub{u}, \hat{x}) = (-1)^{d-1}(-\bgf'(u),  \bgf'(u), 0, \ldots, 0, -1).
\end{equation}
The expression \eqref{eq:unitnormal} can be computed from 
\[n^{\alpha} = (M^{-1})^{\alpha\kappa} \epsilonup_{\kappa, \beta,\gamma,\sigma_2,\ldots,\sigma_{d-1}}  (\partial_u)^\beta (\partial_{\ub u})^\gamma(\partial_{x^2})^{\sigma_2} \cdots(\partial_{x^{d-1}})^{\sigma_{d-1}},\]
where $\alpha,\kappa,\beta,\gamma,\sigma_2,\ldots,\sigma_{d-1}\in\{0,1,\dots,d-1\}$ and $\epsilonup_{\kappa, \beta,\gamma,\sigma_2,\ldots,\sigma_{d+1}} $ is the anti-symmetric symbol normalized by $\epsilonup_{0,1,2,\dots,d-1} = 1$. The second fundamental form can then be computed to equal 
\begin{equation}\label{eq:secffbg}
	\secff = (-1)^{d-1} \bgf''(u) ~\D u^2.
\end{equation}
(We use the convention $\secff(\partial_u, \partial_u) =  \langle\partial_u n, \partial_u\rangle_M$.) Notice that $\secff$ is indeed trace-free with respect to the induced metric as a consequence of and additionally the double contraction $\secff:\secff$ with respect to the induced metric also vanishes, both a consequence of the eikonal equation.

\subsection{The gauge choice and the perturbed system}\label{sect:perturb}
Small perturbations of the embedding \eqref{eq:pwbkgd2} reside within a tubular neighborhood of the background. 
We parametrize the perturbations as a graph within the normal bundle, analogously to the analysis in \cite{DoKrSW2016}; that is, we look for embeddings of the form 
\begin{equation}\label{eq:pertemb} 
	(u, \ub{u}, \hat{x}) \mapsto (t, x^1, \hat{x}, \bgf(u) ) + \phi(u,\ub u, \hat x) \cdot n(u,\ub u, \hat x)
\end{equation}
where $\phi: \Real^{1,d} \to \Real$ is the \emph{height} of the graph, and $n$ is the unit normal as defined in \eqref{eq:unitnormal}. 
The induced metric for this perturbation will be denoted by $\gmetr$; it is given by the pull-back of the Minkowski metric $M$ on $\Real^{1,d+1}$ by the embedding \eqref{eq:pertemb}
\begin{equation}\label{eq:dynmet}
	\gmetr = \mink + \D\phi\otimes \D\phi - 2 \phi\bgf''~ \D u \otimes \D u.
\end{equation}
Its corresponding volume element can be computed to be
\[
	\dvol_\gmetr = \sqrt{|\gmetr|} ~\D{u} \D{\ub{u}} \D{\hat{x}}
\]
where
\begin{equation}
	|\gmetr| \eqdef 1 + \mink(\nabla\phi,\nabla\phi) + 2 \phi \bgf'' (\partial_{\ub{u}} \phi)^2.
\end{equation}
We note that $\gmetr$ is a perturbation of the Minkowski metric $\mink$ with terms both quadratic and linear in $\phi$. For later computations it is helpful to also record the perturbations truncated to the linear terms, which we will denote by $\gmetrf$
\begin{equation}\label{eq:dynmet1}
	\gmetrf \eqdef \mink - 2 \phi\bgf''~ \D u \otimes \D u.
\end{equation}
The inverses of $\gmetr$ and $\gmetrf$ can be computed explicitly in the double null (relative to $\mink$) coordinates $(u,\ub{u}, \hat{x})$. 
For $\gmetrf$ one finds
\begin{equation}\label{eq:InvMet1}
	\gmetrf^{-1} = \mink^{-1} + 2 \phi\bgf'' ~ \partial_{\ub{u}} \otimes \partial_{\ub{u}}.
\end{equation}
Note that this implies 
\begin{equation}
	|\gmetr| = 1 + \gmetrf^{-1}(\nabla\phi, \nabla\phi).
\end{equation}
Using that $\gmetr = \gmetrf + \D\phi \otimes \D\phi$, we can apply the Sherman-Morrison formula \cite{MR0035118} to obtain
\begin{equation}\label{eq:InvMet}
	\gmetr^{-1} = \gmetrf^{-1} - \frac{1}{|\gmetr|} \left(\gmetrf^{-1}\cdot \partial\phi\right)\otimes \left(\gmetrf^{-1}\cdot\partial\phi\right).
\end{equation}

\begin{notn}[Index raising and lowering]
	In the computations to follow, one frequently needs to lower or raise indices with respect to any of $\gmetr$ / $\gmetr^{-1}$, $\gmetrf$ / $\gmetrf^{-1}$, or $\mink$ / $\mink^{-1}$.  
	We will adopt the following conventions
	\begin{itemize}
		\item The unadorned musical operators $\flat$ / $\sharp$ are used for lowering and raising indices with respect to the Minkowski metric $\mink$ of the background simple-plane-wave solution. 
		\item Implicitly lowered / raised indices are always with the Minkowski metric $\mink$, so $\partial^j\phi$ refers to $\mink^{jk}\partial_k\phi$.
		\item When it is clear from the context, we will sometimes omit the index ${-1}$ denoting inverses for brevity. For example, we write $\mink(\nabla\phi,\nabla\phi)$ instead of $\mink^{-1}(\nabla\phi,\nabla\phi)$ since $\nabla\phi$ are naturally covariant and so we will need the contravariant metric $\mink^{-1}$. Similarly, if we write $\gmetr^{\mu\nu}\partial_\nu \phi$ it should be interpreted as $(g^{-1})^{\mu\nu}\partial_\nu \phi$. 
		\item Index manipulations with respect to the dynamical metrics $\gmetr$ and $\gmetrf$ will always be adorned. So for example we will write
			\[ \partial^{\gmetrf\sharp}\phi = \gmetrf^{-1}\cdot \partial\phi, \qquad \partial^{\gmetr\sharp}\phi = \gmetr^{-1} \cdot \partial \phi\]
			with corresponding index notation
			\[ \partial^{\gmetrf\sharp j} \phi = \gmetrf^{jk} \partial_k \phi, \qquad \partial^{\gmetr\sharp j} \phi = \gmetr^{jk} \partial_k\phi.\]
	\end{itemize}
\end{notn}

With the notation announced above, we can equivalently write
\[ \gmetr^{-1} = \gmetrf^{-1} - \frac{1}{|\gmetr|} \partial^{\gmetrf\sharp}\phi \otimes \partial^{\gmetrf\sharp}\phi.\]

For the embedding \eqref{eq:pertemb} to have vanishing mean curvature (i.e.\ satisfy the membrane equation), it must be a formal stationary point of the volume functional $\phi \mapsto \int \dvol_\gmetr$.
The perturbation equations satisfied by $\phi$ can be derived as the corresponding Euler-Lagrange equations, as shown below. 

Denoting by $\mathcal{L} = \sqrt{|\gmetr|} =  \sqrt{1 + \mink( \nabla\phi , \nabla\phi ) + 2 \phi \bgf''( \partial_{\ub{u}} \phi)^2}$ the Lagrangian density, the corresponding Euler-Lagrange equation is
\begin{equation}\label{eq:lagrange1} 
	\frac{\delta \mathcal{L}}{\delta \phi} =\frac{\partial}{\partial {u}}\left(  \frac{\delta \mathcal{L}}{\delta \phi_u} \right) + \frac{\partial}{\partial {\ub{u}} }\left( \frac{\delta \mathcal{L}}{\delta \phi_{\ub{u}}}\right) +\frac{\partial}{ \partial \hat{x} }\left( \frac{\delta \mathcal{L}}{\delta \phi_{\hat{x}}}\right)
\end{equation}
where we use the subscript on $\phi$ to denote partial differentiation. 
Expanding $\mink( \nabla\phi , \nabla\phi ) = - 2 \partial_u\phi \partial_{\ub{u}} \phi + (\partial_{\hat{x}} \phi)^2$ we compute 
\begin{align*}
	\frac{\delta \mathcal L}{\delta \phi} & = \mathcal{L}^{-1} \bgf'' (\phi_{\ub{u}})^2  & \frac{\delta\mathcal L}{\delta \phi_{\hat{x}}} & = \mathcal{L}^{-1} \phi_{\hat{x}} \\
	\frac{\delta\mathcal L}{\delta \phi_u} & = \mathcal{L}^{-1} (- \phi_{\ub{u}}) & \frac{\delta\mathcal L}{\delta \phi_{\ub{u}}} & = \mathcal{L}^{-1}(- \phi_u + 2 \bgf'' \phi\phi_{\ub{u}}).
\end{align*}
So the Euler-Lagrange equation reads
\begin{equation}\label{eq:lagrange2}
	\partial_{\mu} \left( \frac{\gmetrf^{\mu\nu} \partial_\nu \phi}{\mathcal{L}}\right) = \mathcal{L}^{-1} \bgf'' (\phi_{\ub{u}})^2.
\end{equation}
Observe that by \eqref{eq:InvMet} we have
\[ \partial^{\gmetr\sharp}\phi = \frac{1}{|\gmetr|} \partial^{\gmetrf\sharp}\phi .\]
This implies that we can rewrite \eqref{eq:lagrange2} as
\begin{equation}\label{eq:geomeqphi}
	\Box_\gmetr \phi = |\gmetr|^{-1}\bgf'' (\phi_{\ub{u}})^2;
\end{equation}
here $\Box_\gmetr$ refers to the Laplace-Beltrami operator of the metric $\gmetr$, given in local coordinates by 
\[ \Box_\gmetr f = \frac{1}{\sqrt{|\gmetr|}} \partial_{\mu}\big( \sqrt{|\gmetr|} \gmetr^{\mu\nu} \partial_\nu f\big).\]

As the metric $\gmetr$ depends on the first jet of the unknown $\phi$, the principal part of the \eqref{eq:geomeqphi} may be different from $\gmetr^{\mu\nu}\partial^2_{\mu\nu}\phi$. 
For our equation, this turns out not to be an issue, as can be seen when we take the first coordinate partial derivatives of \eqref{eq:lagrange2}. 
With the aid of the relation \eqref{eq:InvMet} between $\gmetr^{-1}$ and $\gmetrf^{-1}$ we obtain
\[
	\partial_{\lambda} \partial_{\mu} \left( \frac{\partial^{\gmetrf\sharp\mu}\phi}{\mathcal{L}}\right) = 
		\partial_{\mu} \left( \frac{\partial^{\gmetr\sharp\mu}\partial_\lambda \phi}{\sqrt{|\gmetr|}} \right) 
		+ \partial_{\mu} \left( \frac{\partial_\lambda \gmetrf^{\mu\nu} \partial_\nu \phi}{\sqrt{|\gmetr|}} - \frac12 \frac{\partial^{\gmetrf\sharp\mu}\phi}{|\gmetr|^{3/2}} \partial_\lambda \gmetrf^{\rho\sigma} \partial_\rho\phi \partial_\sigma\phi \right).
\]
Noticing that the derivatives $\partial_\lambda \gmetrf$ depends only on the \emph{first} derivatives of $\phi$, and not the second, we see that the principal term are all captured in the first term on the right in the above identity.

We can simplify the identity further. Notice that 
\[ \partial_\lambda \gmetrf^{-1} = \partial_{\lambda}(2 \phi \bgf'')  ~ \partial_{\ub{u}} \otimes \partial_{\ub{u}} \]
this implies
\begin{multline*}
\partial_{\mu} \left( \frac{\partial_\lambda \gmetrf^{\mu\nu} \partial_\nu \phi}{\sqrt{|\gmetr|}} - \frac12 \frac{\partial^{\gmetrf\sharp\mu}\phi}{|\gmetr|^{3/2}} \partial_\lambda \gmetrf^{\rho\sigma} \partial_\rho\phi \partial_\sigma\phi \right)
= 
2\partial_{\ub{u}} \left( \frac{\partial_\lambda (\phi\bgf'') \partial_{\ub{u}} \phi}{\sqrt{|\gmetr|}}\right) \\
- \underbrace{\partial_\mu \left( \frac{\partial^{\gmetrf\sharp\mu}\phi}{|\gmetr|^{1/2}} \right)}_{|\gmetr|^{-1/2}\bgf'' (\partial_{\ub{u}}\phi)^2} \frac{1}{|\gmetr|}\partial_\lambda(\phi \bgf'') (\partial_{\ub{u}}\phi)^2 
-  \frac{\partial^{\gmetrf\sharp\mu}\phi}{|\gmetr|^{1/2}} \partial_\mu\left( \frac{1}{|\gmetr|} \partial_\lambda(\phi \bgf'') (\partial_{\ub{u}}\phi)^2 \right).
\end{multline*}
So we conclude
\begin{multline}\label{eq:lagrange3}
	|\gmetr|^{-\frac12 - \frac{2}{d-1}} \Box_{\confmetr} \partial_{\lambda} \phi = \partial_{\lambda} \left( |\gmetr|^{-1/2} \bgf'' (\phi_{\ub{u}})^2 \right) \\
    -  2\partial_{\ub{u}} \left( \frac{\partial_\lambda (\phi\bgf'') \partial_{\ub{u}} \phi}{\sqrt{|\gmetr|}}\right) 
    + |\gmetr|^{-3/2} \bgf'' \partial_{\lambda}(\phi \bgf'') (\partial_{\ub{u}}\phi)^4 \\
    +
    |\gmetr|^{-1/2} \partial^{\gmetrf\sharp \mu}\phi \partial_\mu \left( \frac{1}{|\gmetr|} \partial_\lambda(\phi \bgf'') (\partial_{\ub{u}}\phi)^2 \right),
\end{multline}
where we have introduced the conformal metric
\begin{equation}\label{eq:conformalmet}
	\confmetr = |\gmetr|^{-\frac{2}{d-1}} \cdot \gmetr
\end{equation}
where $d$ is, recall, the number of spatial dimensions. The conformal metric $\confmetr$ has its Laplace-Beltrami operator as 
\[
	\Box_{\confmetr} f = |\gmetr|^{\frac12 + \frac{2}{d-1}} \partial_{\mu} \left( \frac{1}{\sqrt{|\gmetr|}} \gmetr^{\mu\nu} \partial_\nu f \right)
\]
which has the same principal part as $\Box_g$. 

\begin{rmk}\label{rmk:resonances}
	Observe that \eqref{eq:geomeqphi} and \eqref{eq:lagrange3} are \emph{geometric} quasilinear wave equations that linearize to the linear wave equation on $\Real^{1,d}$. The \emph{quadratic nonlinearities} include, as can be seen, the \emph{resonant} semilinear interaction $(\partial_{\ub{u}}\phi)^2$ as well as the \emph{weakly resonant} quasilinear interaction $\phi (\partial_{\ub{u}\ub{u}}^2 \phi)$. 

	That we will be able to prove global existence for this equation (and not suffer from shock formation in finite time) is due to the background $\bgf''$ which accompanies the apperance of such resonant terms, and localizes the resonant interactions to the region $t \approx -x^1$; one can think of $\bgf''$ as $\partial^2_{uu} \bgf$, exposing the null condition that was present in the original membrane equation \eqref{eq:membrane}. 
	However, as the background function $\bgf$ has non-compact (in the $\hat{x}$ direction) support, and is non-decaying (in time), the improved decay we obtain due to this space-time localization is weaker than in classical studies of nonlinear waves with null condition. Such issues and their ramifications are discussed in more detail in Section \ref{sect:semilinear} where we examine a semilinear model that captures the main analytical difficulties. 
\end{rmk}

\section{Main analytic tools}
We will approach our analyses of \eqref{eq:geomeqphi} using a variant of the vector field method. 
In particular, we will make use of the refined global Sobolev inequalities adapted to hyperboloidal foliations that was developed by the second author \cite{Wong2017p}. 
This method allows us to conclude our estimates using only the $T$ multiplier (and not the Morawetz $K$ multiplier), as well as using only the Lorentz boosts as commutators.
In the first part of this section we will recall the $L^2$--$L^\infty$ type linear estimates that enables us to limit ourselves to the small selection of vector fields used in the argument.  
For the full nonlinear problem, the background simple plane-wave contributes coefficients in the form of $\bgf''$ in \eqref{eq:geomeqphi}. To efficiently handle these coefficients using \emph{only} the Lorentz boosts as commutators, we will develop in the second part of this section a weighted vector field algebra. 
In the final part of this section we recall elementary energy estimates for wave equations. 
The combination of these techniques will be first illustrated in a model semilinear problem in Section \ref{sect:semilinear}, before we state and prove the main result of this paper. 

\subsection{A global Sobolev inequality}\label{sect:globsob}
We will base our argument on a hyperboloidal foliation of Minkowski space, and make use of a version of the Sobolev inequality for \emph{weighted} spaces on the hyperboloids, as described in \cite{Wong2017p}. 
The specific family of weights are adapted to be used with energy estimates for wave equations on Minkowski space. 
Before recalling the inequalities, we need to introduce some notations. 

Consider Minkowski space as described by our double null coordinate system $(u,\ub{u}, \hat{x})$ with the metric \eqref{eq:minkDN}. 
Consider the set $\mathcal{I}^+ \eqdef \{ u > 0, \ub{u} > 0, 2u \ub{u} - |\hat{x}|^2 > 0\}$. This set corresponds to the \emph{interior of the future light cone} emanating from the origin in Minkowski space. 
On this set, we can define the time function 
\begin{equation}\label{eq:def:tau}
	\tau \eqdef \sqrt{ 2u\ub{u} - |\hat{x}|^2}.
\end{equation} 
\begin{notn}\label{not:def:tau}
	The level set of $\tau$ will be denoted by $\Sigma_\tau$. The Riemannian metric induced on $\Sigma_\tau$ by the Minkowski metric $\mink$ will be denoted $\sigmmetr_\tau$. The geometric metric $\gmetr$ also induces a symmetric bilinear form on $\Sigma_\tau$, we will denote it by $\siggmetr_\tau$. When $\siggmetr_\tau$ can \emph{in principle} be Lorentzian or degenerate, in our application it will turn out to be always Riemannian. 
\end{notn}
These hypersurfaces $\Sigma_\tau$ are \emph{hyperboloids}: the induced metric $\sigmmetr_\tau$ has constant curvature, with scalar curvature $\tau^{-2}$ times that of the standard hyperbolic space $\mathbb{H}^d$. 
We introduce also the function $\rho$ within this forward light-cone $\mathcal{I}^+$ by
\begin{equation}
	\rho \eqdef \cosh^{-1} \left( \frac{u + \ub{u}}{\sqrt{2}\tau} \right).
\end{equation}
We note that relative to the Minkowski metric, the unit normal to $\Sigma_\tau$ is given by (using an abuse of notation)
\begin{equation}\label{eq:minknormsig}
	-(\D\tau)^\sharp = \frac{1}{\tau} \left( u \partial_{u} + \ub{u} \partial_{\ub{u}} + \hat{x} \cdot \partial_{\hat{x}} \right).
\end{equation}
Relative to the perturbed metric $g$, the unit normal to $\Sigma_\tau$ takes the form
\begin{equation}\label{eq:geonormsig}
	- \frac{(\D\tau)^{\gmetr\sharp}}{\sqrt{| \gmetr(\D\tau,\D\tau)|}} = - \frac{(\D\tau)^\sharp + 2(\frac{u}\tau)\phi\bgf'' \partial_{\ub{u}} - \gmetrf(\D\phi,\D\tau) \partial^{g\sharp}\phi }{\sqrt{1 - 2 (\frac{u}\tau)^2 \phi\bgf'' + |\gmetr|^{-1}[ \gmetrf(\D\phi,\D\tau)]^2 }}.
\end{equation}

We define the following vector fields:
\begin{gather}
	\label{eq:def:T} T = \frac{1}{\sqrt{2}} (\partial_u + \partial_{\ub{u}});\\
	\label{eq:def:L1} L^1 =  u \partial_{u} - \ub{u} \partial_{\ub{u}};\\
	\label{eq:def:Li} L^i = \frac{1}{\sqrt{2}}(u + \ub{u}) \partial_{\hat{x}^i}+ \frac{1}{\sqrt{2}}\hat{x}^i (\partial_u + \partial_{\ub{u}}), \mathrlap{\quad i = 2, \ldots, d.}
\end{gather}
They are all Killing with respect to the Minkowski metric; in fact, $T$ is the standard \emph{time-translation} and the $L^i$s ($i = 1, \ldots, d$) are the standard \emph{Lorentz-boosts}. 
Note that the $L^i$ (where $i = 1, \ldots, d$) are also all tangential to $\Sigma_\tau$. 
If $\alpha$ is an $m$-tuple with elements drawn from $\{1, \ldots, d\}$ (namely that $\alpha = (\alpha_1, \ldots, \alpha_m)$ with $\alpha_i \in \{1, \ldots, d\}$), we denote by $L^\alpha$ the differential operator 
\[ f \mapsto L^{\alpha_m} L^{\alpha_{m-1}} \cdots L^{\alpha_2} L^{\alpha_1} f.\]

We introduce also the stress-energy tensors corresponding to a metric 
\begin{equation}\label{eq:def:seten}
	\begin{gathered}
		\seten[\phi;\mink] =  \D\phi \otimes \D\phi - \frac12 \mink(\D\phi,\D\phi) \mink\\
		\seten[\phi;\gmetr] = \D\phi\otimes \D\phi - \frac12 \gmetr(\D\phi,\D\phi) \gmetr.
	\end{gathered}
\end{equation}
Note that $\seten[\phi;\gmetr] = \seten[\phi;\confmetr]$; the (covariant) stress-energy tensors are invariant under conformal changes of metric. 

The main results we will need from \cite{Wong2017p} are the following.

\begin{thm}[Sharp global Sobolev inequality] \label{thm:globsob}
	Let $\ell\in \Real$ be fixed. For any function $f$ defined on $\mathcal{I}^+$, we have, for any $(u,\ub{u}, \hat{x})\in \mathcal{I}^+$, 
	\[ |f(u,\ub{u},\hat{x})|^2 \lesssim_{d,\ell} \tau_0^{-d} \cosh(\rho_0)^{1-d-\ell} \sum \int_{\Sigma_{\tau_0}} \cosh(\rho)^\ell |L^\alpha f|^2~ \dvol_{\sigmmetr_{\tau_0}}.\]
	The sum is taken over all $m$-tuples $\alpha$ with elements drawn from $\{1, \ldots, d\}$ with $m \leq \lfloor \frac{d}{2} \rfloor + 1$. The quantities $\tau_0$ and $\rho_0$ appearing on the right of the inequality are given as
	\[ \tau_0 = \tau(u,\ub{u},\hat{x}) , \qquad \rho_0 = \rho(u,\ub{u},\hat{x}).\]
\end{thm}
\begin{rmk}
	Note that by the definition of the function $\rho$, the coefficient in Theorem \ref{thm:globsob} can be written as
	\[ \tau_0^{-d} \cosh(\rho_0)^{1-d-\ell} = \tau_0^{\ell-1} \left( \frac{u + \ub{u}}{\sqrt{2}} \right)^{1-d-\ell}.\]
\end{rmk}
\begin{prop}[Decomposition of stress-energy] \label{prop:sedecomp}
	At every point in $\mathcal{I}^+$ we have the pointwise identity
	\[ \seten[f;\mink](T, -(\D\tau)^\sharp) = \frac{1}{2 \tau^2 \cosh(\rho)} \sum_{i = 1}^d (L^if)^2 + \frac{1}{2\cosh(\rho)} (Tf)^2.\]
\end{prop}
\begin{lem}[Hardy's inequality, $d\geq 3$] \label{lem:hardy}
	Let $d \geq 3$. For any function $f$ defined on $\Sigma_\tau$, we have
	\[ \int_{\Sigma_\tau} \frac{1}{\cosh(\rho)} |f|^2 ~\dvol_{\sigmmetr_\tau} \leq \frac{4}{(d-2)^2} \int_{\Sigma_{\tau}} \frac{1}{\cosh(\rho)} \sum_{i = 1}^d (L^i f)^2 ~\dvol_{\sigmmetr_\tau}.\]
\end{lem}

\subsection{A weighted vector field algebra}\label{sect:wvfalg}

In classical arguments using the vector field method, one typically commutes the equation with the generators of the Poincar\'e group, which consists of the 
\begin{itemize}
	\item translation vector fields $\partial_t, \partial_{x^i}$;
	\item rotations $x^i \partial_{x^j} - x^j \partial_{x^i}$;
	\item Lorentz boosts $t \partial_{x^i} + x^i \partial_t$.
\end{itemize}
These vector fields form, under the Lie bracket, an $\Real$-algebra. 

In applying the Sobolev inequality of the previous section, we intend to only commute with the Lorentz boosts $L^i$. 
This subset does not form an $\Real$-algebra under the Lie bracket. 
However, they form an algebra with coefficients drawn from a space of weights.
For convenience, we introduce the $y$-coordinates
\begin{equation}\label{eq:ycoord}
		y^0 = \frac{u+\ub{u}}{\sqrt{2}}, \quad y^1 = \frac{u-\ub{u}}{\sqrt{2}}, \quad y^i = \hat{x}^i \;\; (i \geq 2).
\end{equation}
\begin{defn}\label{defn:sweight}
	We denote by $\sweight_*$ the (commutative) ring of polynomial expressions in the $d+1$ variables $\big\{ \frac{y^1}{y^0} , \ldots, \frac{y^d}{y^0}, \frac{1}{y^0} \big\}$, with $\Real$ coefficients. 
	This ring can be graded according to the degree of the $\frac{1}{y^0}$ term in the polynomial expression, we denote by $\sweight_i$ the corresponding set of homogeneous elements. 
\end{defn}

\begin{rmk}
	By way of clarification and for example, we will have that the expression $\frac{y^1}{y^0}\in \sweight_0$, while $(\frac{1}{y^0})^5 (\frac{y^2}{y^0}) (\frac{y^4}{y^0}) \in \sweight_5$. 
\end{rmk}
\begin{rmk}\label{rmk:suppcontrol}
	Notice that within the light cone $\mathcal{I}^+$, we have that the functions (for all $i = 1, \ldots, d$)
	\[ \big| \frac{y^i}{y^0} \big| \leq 1 \]
	are uniformly bounded. 
\end{rmk}
Now, observe that for $i,j \in \{1, \ldots, d\}$,
	\begin{align*}
		T\big(\frac{1}{y^0}\big) &= - \big(\frac{1}{y^0}\big)^2, & T\big(\frac{y^i}{y^0}\big) &= - \frac{1}{y^0}\cdot\frac{y^i}{y^0}, \\
		L^i \big(\frac{1}{y^0}\big) &= - \frac{1}{y^0} \cdot \frac{y^i}{y^0}, & L^i \big(\frac{y^j}{y^0}\big) &= \delta_{ij} - \frac{y^i}{y^0} \cdot \frac{y^j}{y^0}.\\
		\intertext{Furthermore,}
		[L^i, T] &= - \frac{1}{y^0} L^i + \frac{y^i}{y^0} T, &[L^i, L^j] &= \frac{y^i}{y^0} L^j - \frac{y^j}{y^0} L^i.
	\end{align*}
Together these implies that the set of vector fields of the form $c_0 T + \sum c_i L^i$ where the $c_\mu$ are taken from $\sweight_*$ form not only an $\Real$-Lie algebra, but also an algebra over the ring $\sweight_*$, with multiplication being the Lie bracket. 
We will denote this algebra by $\algcomm_*$. 
The following proposition follows immediately from the computations above.
\begin{prop}\label{prop:commalg}
	For $i \in \mathbb{Z}_+$, define
	\[\algcomm_0 \eqdef \left\{ \sum_{j=1}^d c_j L^j\ |\ c_j \in \sweight_0\right\}, \qquad \algcomm_i = \left\{ c_0 T + \sum_{j=1}^d c_j L^j \ | \ c_0 \in \sweight_{i-1}, \ c_j \in \sweight_i\right\}.\] Then $\algcomm_*$ is graded, with $L^i \in \algcomm_0$, and $T\in \algcomm_1$. In particular, given elements $X_a\in \algcomm_a, X_b\in \algcomm_b$ and $f\in \sweight_c$, we have that 
	\[ [X_a, X_b] \in \algcomm_{a+b}, \quad fX_a \in \algcomm_{a+c}.\]
\end{prop}
\begin{rmk}\label{rmk:doubcomm}
	We remark that we also have the following commutator relation
	\[ 
		[L^i, [L^j, T]] = \delta_{ij} T \in \algcomm_{1}
	\]
	as expected. 
\end{rmk}

Using $\algcomm_*$, we can build an algebra of differential operators which we label by $\algdiff_*^{*,*}$. 
Consider terms of the form 
\begin{equation}\label{eq:def:algdiff:term}
	f X^1 X^2 X^3 \ldots X^k
\end{equation}
where $f\in \sweight_*$ and $X^\alpha \in \{L^i, T\}$. 
They are differential operators that act on functions defined on $\mathcal{I}^+$ in the usual way. 
Using the computations above we see that terms of such form are closed under composition of differential operators. 
Hence we define $\algdiff_*^{*,*}$ as the set of finite sums of terms of the form \eqref{eq:def:algdiff:term}, with addition defined normally and composition as multiplication; $\algdiff_*^{*,*}$ is obviously a $\sweight_*$-module. 

In exactly the same way as $\algcomm_*$, the algebra $\algdiff_*^{*,*}$ is graded. We will use its lower index to record this grading. 
\begin{defn}\label{defn:algdiff:ind}
	The \emph{weight} of a term of the form \eqref{eq:def:algdiff:term}, where $f$ is a monomial, is defined by the number of times $T$ appears among the $X^\alpha$, plus the number of times $1/y^0$ appears in the monomial $f$. The \emph{degree} of a term of the form \eqref{eq:def:algdiff:term} is defined as the number $k$. The $T$-\emph{degree} of a term of the form \eqref{eq:def:algdiff:term} is the number of times $T$ appears among the $X^\alpha$. By $\algdiff_w^{k,s}$ we refer to the set of finite sums in $\algdiff_*^{*,*}$ of elements with weight $w$ and degree \emph{at most} $k$, and $T$-degree \emph{at most} $s$. 
\end{defn}

\begin{rmk}
	The set $\algdiff_w^{k,s}$ is well-defined due to Proposition \ref{prop:commalg}. One needs to check that, for example, $f X^1 X^2$ and $f X^2X^1 + f [X^1,X^2]$, which are equal as differential operators, have the same degrees and weight. Proposition \ref{prop:commalg} implies that for $X^\alpha \in \{L^i, T\}$, the terms making up $[X^1,X^2]$ always have the same weight as $X^1 X^2$, and with same or lower $T$-degree. 
\end{rmk}

For example, given any $m$-tuple $\alpha$, the operator $L^\alpha \in \algdiff_0^{m,0}$, while we can identify $\sweight_* = \algdiff_*^{0,0}$.
The set $\algcomm_w$ are the set of degree (exactly) 1 elements in $\algdiff^{*,*}_w$. 
The following proposition follows immediately from the definition and Proposition \ref{prop:commalg}. 
\begin{prop}\label{prop:diffalg}
	If $A\in \algdiff_w^{k,s}$, and $B\in \algdiff_{w'}^{k',s'}$, then 
	\begin{enumerate}
		\item $AB \in \algdiff_{w+w'}^{k+k', s+s'}$;
		\item $[A,B] \in \algdiff_{w+w'}^{k+k' - 1, s+s'}$.
	\end{enumerate}
\end{prop}

We remark finally that if $f = f(u)$ is a function defined within the light cone $\mathcal{I}^+$, then 
\[
	Tf = \frac{1}{\sqrt{2}} f'(u), \quad \text{ and } L^1f = u f'(u).
\]
In particular, if $f$ is smooth and supported within a slab $u\in (a,b)$, then both $Tf$ and $L^1f$ are functions of $u$ alone that are smooth and supported within $u\in (a,b)$. Similarly, we see that for $i \geq 2$
\[ 
	L^i f = \frac{1}{\sqrt{2}} \hat{x}^i f'(u).
\]
To estimate functions of this form, we will use the following lemma. 
\begin{lem}\label{lem:fderest}
	Fix $f = f(u)$ a smooth function supported in $u\in [a,b]$. Then on the set $\mathcal{I}^+$, for any $m$-tuple $\alpha$, 
	we have 
	\[ |L^\alpha f| \lesssim (1 + \ub{u})^{m/2} \cdot \mathbf{1}_{\{u\in [a,b]\}}.\]
	The implicit constant depends on the numbers $a,b$, the degree $m$, the dimension $d$, as well as $\|f\|_{C^m}$. 
\end{lem}
\begin{proof}
	Observe that if $i,j\in \{2, \ldots, d\}$, 
	\[ L^i \hat{x}^j = \frac{1}{\sqrt{2}} \delta^{ij} (u + \ub{u}) \]
	and that
	\[ L^i (u + \ub{u}) = \sqrt{2} \hat{x}^i.\]
	So we have that up to a universal structural constant depending only on the dimension $d$ and the degree $m$,
	\[ | L^\alpha f| \lesssim \big(1 + |u|^m + |u + \ub{u}|^{m/2} + |\hat{x}|^m\big)\cdot \|f\|_{C^m}\cdot \mathbf{1}_{\{u\in [a,b]\}}.\]
	As on the set of interest, $u \in [\max(a,0),b]$, we have that $|u| < b$. Furthermore, on $\mathcal{I}^+$ by definition we have $2 u\ub{u} > |\hat{x}|^2$. The boundedness of $u$ implies that $|\hat{x}| \lesssim \sqrt{\ub{u}}$. The desired bound follows. 
\end{proof}

\begin{rmk}\label{rmk:cmpctdecay}
	In the Lemma above, we consider only the case of $f$ with compact support in $u$ for convenience in notation and clarity of argument. 
	One sees easily that an analogous statement holds for $f = f(u)$ that is Schwartz in $u$ (i.e.\ smooth with rapid decay of all derivatives in $u$). 
	In fact, our main results, which are stated for $\bgf\in C^\infty_0$, will follow if our background function $\bgf$ is smooth and such that $\bgf''$ and its higher derivatives have sufficiently fast polynomial decay; we leave such easy but tedious generalizations to the reader. 
\end{rmk}

\subsection{Basic estimates}\label{sect:basest}
To close this section we summarize the basic energy estimates used to control solutions to wave equations in the vector field method. 
Define the following energy integrals:
\begin{gather}
	\label{eq:def:eng1}\Energy_\tau[\phi]^2 \eqdef \Energy_\tau[\phi;\mink]^2 \eqdef 2 \int_{\Sigma_{\tau}} \seten[\phi;\mink](T, -(\D\tau)^\sharp) ~\dvol_{\sigmmetr_\tau},\\
	\label{eq:def:eng2}\Energy_\tau[\phi;\gmetr]^2 \eqdef 2\int_{\Sigma_\tau} \frac{1}{\sqrt{|g(\D\tau,\D\tau)|}} \seten[\phi;\gmetr](T, (-\D\tau)^{\gmetr\sharp}) ~\dvol_{\siggmetr_\tau},
\end{gather}
Notice that the one corresponding to the Minkowski metric can be written explicitly as
\begin{equation}\label{eq:def:minkener}
	\Energy_\tau[\phi]^2  = \int_{\Sigma_\tau} \frac{1}{\tau^2 \cosh\rho} \sum_{i = 1}^d (L^i\phi)^2 + \frac{1}{\cosh\rho} (T\phi)^2 ~\dvol_{\sigmmetr_\tau}.
\end{equation}

The significance of these integrals stems from the fact that, denoting by $\mathcal{J}[\phi;g]$ the vector field defined by\footnote{Recall that vector fields are uniquely determined through their dual pairing with a basis of one forms.}
\[ \mathcal{J}[\phi;g](\omega) \eqdef \seten[\phi;g](T, \omega^{g\sharp}) \]
where $\omega$ is an arbitrary one-form, we have that its divergence
\[ \mathrm{div}_g \mathcal{J}[\phi;g] = \frac12 \seten[\phi;g] :_g \lieD_T g + \Box_g \phi \cdot T(\phi) \]
where $:_g$ denotes the double contraction of a pair of symmetric two-tensors using the metric $g$. The tensor $\lieD_Tg$ is the Lie-derivative of the metric $g$ by the vector field $T$.
We call $\mathcal{J}$ the $T$-energy current associated to $\phi$ and $g$.
Integrating the divergence between two level sets $\tau_0 <\tau_1$ of $\tau$ one obtains the \emph{energy inequality}
\begin{equation}\label{eq:engineq}
	\Energy_{\tau_1}[\phi;g]^2 \leq \Energy_{\tau_0}[\phi;g]^2 + \iint_{\tau \in [\tau_0,\tau_1]} \big|\seten[\phi;g]:_g \lieD_Tg\big| + 2\big|\Box_g \phi \cdot T(\phi)\big| ~\dvol_{g}.
\end{equation}
Note that the perturbed metric $g$ in \eqref{eq:engineq} can be replaced with $\mink$, and in that case, as a consequence of $\lieD_T \mink = 0$, the inequality reduces to the standard energy estimate on hyperboloids 
\begin{equation}\label{eq:engineqmink}
\Energy_{\tau_1}[ \phi]^2 - \Energy_{\tau_0}[ \phi]^2 \lesssim \iint_{\tau\in [\tau_0, \tau_1]} |\Box \phi|\cdot |T  \phi| ~\dvol_{\mink}.
\end{equation}
The following proposition is how one obtains pointwise control for terms appearing on the right hand side of \eqref{eq:engineq} - \eqref{eq:engineqmink}. 

\begin{prop} \label{prop:peeling}
	For any function $f$ defined on $\mathcal{I}^+$, we have, for any $(u,\ub{u}, \hat{x})\in \mathcal{I}^+$, 
	\begin{gather*}
		|f(u, \ub{u}, \hat{x})| \lesssim_d \tau^{1-\frac{d}2} \cosh(\rho)^{1 - \frac{d}{2}} \sum_{|\alpha| \leq \lfloor\frac{d}{2}\rfloor} \Energy_{\tau}[L^\alpha f], \\
		|L^i f(u, \ub{u}, \hat{x})| \lesssim_d \tau^{1-\frac{d}2} \cosh(\rho)^{1 - \frac{d}{2}} \sum_{|\alpha| \leq \lfloor\frac{d}{2}\rfloor+1} \Energy_{\tau}[L^\alpha f], \\
		|Tf(u, \ub{u}, \hat{x})| \lesssim_d \tau^{-\frac{d}2} \cosh(\rho)^{1 - \frac{d}{2}} \sum_{|\alpha| \leq \lfloor\frac{d}{2}\rfloor+1} \Energy_{\tau}[L^\alpha f].
	\end{gather*}
\end{prop}
\begin{proof}

We provide the proof of the estimates for $f$ and $Tf$, as the estimate for $L^if$ is similar and simpler to deduce. Using Theorem \ref{thm:globsob} with $\ell = -1$, we find 
\begin{align*}
|f(u,\ub{u},\hat x)|^2  & \lesssim_d \tau^{-d} \cosh(\rho)^{2 - d} \sum_{|\alpha| \leq \lfloor \frac{d}{2}\rfloor + 1} \int_{\Sigma_\tau} \frac{1}{\cosh(\rho)} |L^\alpha f|^2 \ \dvol_{\eta_\tau} ,\\
& \lesssim_d \tau^{2-d} \cosh(\rho)^{2-d} \sum_{|\alpha| \leq \lfloor \frac{d}{2}\rfloor + 1} \int_{\Sigma_\tau} \frac{1}{2\tau^2\cosh(\rho)} |L^\alpha f|^2 \ \dvol_{\eta_\tau},
\end{align*}
where we used the fact that $\tau$ is constant on $\Sigma_\tau$. Using Lemma \ref{lem:hardy} to control the $|\alpha| = 0$ case and Proposition \ref{prop:sedecomp} to control the right hand side by $\Energy_\tau$, we see 
\[ |f(u,\ub u, \hat x)|^2 \lesssim_d \tau^{2-d} \cosh(\rho)^{2-d} \sum_{|\alpha| \leq \lfloor \frac{d}{2}\rfloor} \Energy_\tau[L^\alpha f]^2.\] This concludes the proof for $f$.

Using again Theorem \ref{thm:globsob} with $\ell = -1$, we see
\[|Tf(u,\ub u, \hat x )|^2 \lesssim_d  \tau^{-d} \cosh(\rho)^{2 - d} \sum_{|\alpha| \leq \lfloor \frac{d}{2}\rfloor + 1} \int_{\Sigma_\tau} \frac{1}{\cosh(\rho)} |L^\alpha T f|^2 \ \dvol_{\eta_\tau}.\]
By the computations leading up to Proposition \ref{prop:commalg} (specifically the expression for $[L^i,T]$ and Remark \ref{rmk:suppcontrol}), we can estimate 
\[ L^\alpha T f| \lesssim \sum_{|\beta| \le |\alpha|} |T L^\beta f| + \sum_{|\beta| \le |\alpha| -1}\sum_{i=1}^d \frac{1}{\tau\cosh(\rho)} |L^iL^\beta f|.\] 
In the above inequality, $L^i L^\beta f \eqdef L^i f$ when $|\alpha | = 1$. The terms in the second sum can be controlled by 
\[\frac{1}{\tau} |L^iL^\beta f| \] 
because $\cosh(\rho) \ge 1$. Putting these estimates together we find that 
\begin{multline}
\sum_{|\alpha| \leq \lfloor \frac{d}{2}\rfloor + 1} \int_{\Sigma_\tau} \frac{1}{\cosh(\rho)} |L^\alpha T f|^2 \ \dvol_{\eta_\tau} \lesssim \sum_{|\alpha| \leq \lfloor \frac{d}{2}\rfloor + 1} \int_{\Sigma_\tau} \frac{1}{\cosh(\rho)} |TL^\alpha  f|^2 \\
 +  \frac{1}{\tau^2\cosh(\rho)} \sum_{i=1}^d |L^i L^\alpha f|^2 \  \dvol_{\eta_\tau}.
\end{multline}
By Proposition \ref{prop:sedecomp}, the right hand side is bounded by 
\[ \sum_{|\alpha| \leq \lfloor \frac{d}{2}\rfloor + 1} \Energy_\tau[L^\alpha f]^2,\]
as desired.
\end{proof}

\begin{rmk}\label{rmk:peeling}
	A feature of \eqref{eq:def:minkener} is its \emph{anisotropy}. The classical energy estimates of wave equations control integrals of $|\partial_t \phi|^2 + |\nabla \phi|^2$ where all components appear on equal footing. 
	Here, however, the transversal (to $\Sigma_\tau$) derivative $T\phi$ has a different weight compared to the tangential derivatives $L^i\phi$. Noting that by their definitions, $T$ has \emph{unit-sized} coefficients with expressed relative to the standard coordinates of Minkowski space. The coefficients for $L^i$ (within the light cone $\mathcal{I}^+$) have size $\approx t$. 
	Therefore an \emph{isotropic} analogue would be expected to contain integrals of $\frac{1}{t^2} (L^i\phi)^2$ along with integrals of $T\phi$. 
	Noting that $t = \tau\cosh\rho$ this indicates that an isotropic analogue would contain, instead of the integral given in \eqref{eq:def:minkener}, the integral
	\[ \int_{\Sigma_{\tau}} \frac{1}{\tau^2 \cosh(\rho)^3} \sum (L^i \phi)^2 + \frac{1}{\cosh(\rho)} (T\phi)^2 ~\dvol_{\sigmmetr_\tau}.\]
	In other words, the integral for $L^i\phi$ in the energy has a \emph{better} $\rho$ weight than would be expected from an isotropic energy, such as that controlled by the standard energy estimates. 

	This improvement reflects the fact that the energy estimate described in this section captures the \emph{peeling properties} of linear waves within the energy integral itself. It is well-known that derivatives \emph{tangential to an out-going light-cone} decay faster along the light-cone, than derivatives transverse to the light-cone. 
	As asymptotically hyperboloids approximate light-cones, we expect the same peeling property to survive. Indeed, the energy inequality \eqref{eq:engineq} shows that we can capture this in the integral sense. 
\end{rmk}

\section{A semilinear model}
\label{sect:semilinear}

Before stating and proving our main results, we will illustrate both our method of proof and the main difficulties encountered in the simpler setting of a \emph{semilinear} problem. 
Recall that the small-data global existence problem for the membrane equation \eqref{eq:membrane} in dimension $d \geq 3$ follows from a direct application of Klainerman's vector field method, after noting that the equation of motion is a quasilinear perturbation of the linear wave equation with \emph{no} quadratic nonlinearities. 
In particular, Klainerman's null condition plays no role in establishing this result. 
As indicated in Remark \ref{rmk:resonances}, the perturbation problem for simple plane-waves introduces resonant quadratic terms to which Klainerman's null condition does not directly apply. 
On the other hand, as observed in that same remark, there is a hidden null structure from which we can expect to recover some improved decay rates. 

The main difficulty however is that Klainerman's null condition is built upon the \emph{expected decay rates} corresponding to solutions to the linear wave equations with \emph{strongly localized initial data}. 
In particular, the heuristic for the null condition is based on the expectation that, for generic first derivatives of such a solution, $\partial\phi$ decays like $t^{(1-d)/2}$; while for ``tangential'' (to an outgoing null cone foliation) derivatives, the corresponding derivatives decays like $t^{-d/2}$. 

In our setting, however, one of the waves in the interaction is a simple plane-wave which \emph{does not decay at all}. 
This reduces the effectiveness of the null structure in improving decay. 
As will be shown this difficulty manifests already in the model problem to be discussed in this section.
For example, when applying the vector field method one studies the equations of motion satisfied by higher derivatives of the solution.
After commuting the equation with the Lorentz boosts, one sees that when the boost hits on $\bgf''$, we obtain a coefficient that, while still localized to $t \approx -x^1$, is growing in time. 
On an intuitive level one can interpret this as a transfer of energy from the (infinite energy) background simple plane-wave to the perturbation. 
The null structure in our context then serves to cap the rate of this energy transfer, ensuring (in our case) global existence of the perturbed solution. 

The specific semilinear model problem we consider takes \eqref{eq:geomeqphi} and drops from it the quasilinearity. 
That is to say, we consider the small-data problem for the semilinear wave equation 
\begin{equation}\label{eq:modelprob}
	\Box \phi = \bgf''(u) (\phi_{\ub{u}})^2
\end{equation}
on $\Real^{1,d}$, where $\Box$ is the usual wave operator corresponding to the Minkowski metric $\mink$. 
To approach this problem using a vector field method, one commutes \eqref{eq:modelprob} with the Lorentz boosts to derive equations of motions for higher order derivatives. 
The energy estimates for these higher order derivatives are then combined with the global Sobolev inequality to get $L^\infty$ \emph{decay} estimates for the solution. 
The main difficulty one encounters here, however, is when the vector fields hit on the background $\bgf''$. 
We have
\[ \Box L^\alpha \phi = L^\alpha(\bgf'') (\phi_{\ub{u}})^2 + \ldots,\]
where $L^\alpha(\bgf'')$ can have \emph{growing} $L^\infty$ norm. 

This potential growth of the coefficients is the main technical complication in this problem. 
The best uniform estimate we have for $L^\alpha(\bgf'')$, assuming for convenience $\bgf\in C^\infty_0$ and the initial data for $\phi$ is compactly supported, is via Lemma \ref{lem:fderest}, which gives
\[ \Box L^\alpha \phi \approx (1 + \ub{u})^{m/2} \mathbf{1}_{\{u\in [a,b]\}} \cdot (\underbrace{1 + u+\ub{u}}_{\partial_{\ub{u}}\phi})^{-d/2} \partial\phi \]
where we have made the optimistic assumption that $\phi_{\ub{u}}$ decays like $(1 + u+\ub{u})^{-d/2}$, as would be the case for a linear wave. 

At this point, \textbf{two different complications present themselves}. 
First, one may naively hope that the (higher order) energies always stay bounded, in analogy with the linear case. 
This hope is rapidly dashed when we examine the energy estimate for $|\alpha| = d$. After commuting with $d$ derivatives, we see that
\[ \Box L^\alpha \phi \approx \partial\phi \]
with no decay! Even assuming that we can prove the boundedness of the lowest order energy (which controls $\partial\phi$ in $L^2$), the best we can obtain is then that energy for $L^\alpha \phi$ grows linearly in time. 
This first difficulty can be \textbf{overcome with a modified bootstrap scheme} where the expected (polynomial in time) energy growth is incorporated into the assumptions. 
\begin{rmk}
	Several remarks are in order concerning this energy growth:
	\begin{enumerate}
		\item This growth is different from what appears in typical applications of the vector field method to nonlinear wave equations with null-condition satisfying nonlinearities in $d = 3$. In those cases, the equation takes the schematic form
			\[ \Box \phi = L\phi \ub{L}\phi \]
			where $L\phi$ is a ``good'' derivative that is expected to decay like $t^{-d/2}$ and $\ub{L}\phi$ is a ``bad'' derivative that is expected to decay like $t^{(1-d)/2}$. When considering the energy estimates for the top order derivatives, one must face the possibility of needing to control
			\[ \Box \partial^\alpha \phi = \partial^\alpha L\phi \cdot \ub{L}\phi + \ldots.\]
			To close the energy estimate, one must estimate $\partial^\alpha L\phi$ in $L^2$ and thereby bound $\ub{L}\phi$ in $L^\infty$ by $1/t$, whereupon the time integration gives a small energy growth of the top order derivatives.

			This difficulty is already largely avoided in hyperboloidal energy methods, exploiting the anisotropic inclusion of ``good'' versus ``bad'' derivatives in the energy (see Proposition \ref{prop:sedecomp} and Remark \ref{rmk:peeling}), and is \emph{not} the cause of the energy growth in our argument. 
		\item That the two energy growths are distinct can be seen in the fact that for the classical applications of the vector field method, the energy growth occurs only for the highest order derivatives used in the argument. The more derivatives one uses in the bootstrap, the more levels of energy that remain bounded. In our case, the energy growth starts appearing at a fixed (depending on the dimension $d$) level of derivatives, regardless of how many derivatives is used in the bootstrap argument. The reason for this is because there are terms in the equation that do not enjoy the boost symmetries, and every time you differentiate them with a boost, one gets another growth of $\ub{u}$. (see Lemma \ref{lem:fderest}).
		\item Similarly, this energy growth is also different from the $\mu$-degeneracy of the highest orders of energies (and the associated ``descent scheme'') that appears in the study of formation of shocks \cite{Christ2007a} (see also discussion in \cite{HoKlSW2016}). 
	\end{enumerate}
\end{rmk}

The second difficulty is more sinister. 
To close the energy estimate, and estimate $\phi_{\ub{u}}$ in $L^\infty$, we need to commute with at least $d/2$ derivatives in order to make use of Sobolev, implying that $m > d/2$. 
But then the coefficients on the right hand side are of size $(1 + \ub{u})^{-d/4+\epsilon}$, which is not integrable when $d = 2, 3, 4$.
This seemingly prevents us from even closing \emph{any} bound for $|\partial \phi|$. Take for example the case $d = 3$. 
\begin{itemize}
	\item Assuming $L^\infty$ control on $|\partial\phi|$ of the type $(1 + u + \ub{u})^\lambda$, the coefficients in the equation for $LL\phi$ grows like $(1 + \ub{u})^{1- \lambda}$, This implies that, even assuming the lowest-order energy remains bounded, the energy for $LL \phi$ grows like $(1 + u + \ub{u})^{2 - \lambda}$. 
	\item In the \emph{best} case, we expect that the growth of the $LL \phi$ energy means a weakened $L^\infty$ control on $|\partial\phi|$, to the tune of $(1 + u + \ub{u})^{2 - \lambda - 3/2}$, with the power $(-3/2)$ coming from the global Sobolev inequalities. 
	\item Thus, we see that \emph{at every iteration} one would increase the growth rate of $|\partial\phi|$ by $(1 + u + \ub{u})^{1/2}$.
\end{itemize}
To handle this difficulty, \textbf{we will make use of the hyperboloidal foliation and its associated sharp global Sobolev inequalities}. 
In particular, the anisotropy discussed in Remark \ref{rmk:peeling} allows us to exploit an additional vestige of the null structure of the membrane equation to gain, effectively, an additional $(1 + u + \ub{u})^{-1}$ decay in the most difficult terms and close the argument also in $d = 3$ and $4$. 
This is accomplished by essentially ``borrowing'' a weight from the $|\partial_{\ub{u}}\phi|$ term when we put it in $L^2$, using the fact that the term we are trying to control is also a ``good derivative'' and benefits from the anisotropic energy.  
The vestigial null structure is explained in Remark \ref{rmk:ncvestige} below. 

\begin{rmk}
	This improvement is not sufficient for the $d = 2$ case, even at the heuristic level, due to logarithmic divergences when integrating $(1 + s)^{-1}$. As the stability of plane-waves is trivial in $d = 1$ (using either the integrability of the membrane equation in this case, or via an easy modification of the arguments in \cite{Wong2017}), we have reasons to expect that the stability result also holds for $d = 2$. This turns out to be indeed the case, if we factor in the additional improvements we used in the more detailed analyses for the quasilinear problem in Section \ref{sect:commutedeq}. See also Remark \ref{rmk:semi:two}.
\end{rmk}

Note that these difficulties are essentially due to the fact that the background function $\bgf(u)$, while being a solution to the linear wave equation $\Box \bgf(u) = 0$, is not one that is associated to localized initial data. 
Hence its derivative with respect to $L^\alpha$ has worse decay rates. (In fact, it grows in time.)

Concerning this semilinear model, we will study the initial value problem for \eqref{eq:modelprob} with initial data prescribed on the hypersurface $\{y^0 = 2\}$ (here $y$ is defined as in \eqref{eq:ycoord}),
\[ \phi|_{y^0 = 2} = \phi_0, \quad \partial_{y^0} \phi |_{y^0 = 2} = \phi_1.\]
The remainder of this section is devoted to proving the following theorem. 

\begin{thm}
	Let $d \geq 3$ and assume $\bgf''(u)$ is smooth and has compact support in $u$. 
	Consider the initial value problem for \eqref{eq:modelprob} where $\phi_0$ and $\phi_1$ are smooth compactly supported functions on $B(0,1) \subset \Real^d$. 
	Let $s = d$ if $d$ is odd, and $s = d+1$ if $d$ is even. 
	Then provided $\|\phi_0\|_{H^{s+1}} + \|\phi_1\|_{H^s}$ is sufficiently small, the initial value problem has a global-in-time solution. 
\end{thm}

\subsection{Preliminaries}
Using the standard local existence theorem with finite-speed of propagation we can assume the solution exists up to at least $\Sigma_2$.
Furthermore, by finite speed of propagation, the solution must vanish when 
\[ \sqrt{\sum_{i = 1}^d |y^i|^2} > |y^0 - 2| + 1.\]
In particular, this implies
\begin{equation}\label{eq:suppcond}
	\sqrt{2}(u+\ub{u}) \leq \tau^2 + 1
\end{equation}
on the support of $\phi$. 

By the blow-up criterion for wave equations, it suffices to show $\|\phi\|_{W^{1,\infty}(\Sigma_\tau)} < \infty$ for every $\tau\in (2,\infty)$. 
The general approach, which we will take also for studying the quasilinear problem, is that of a bootstrap argument. 
\begin{enumerate}
	\item We will assume that, up to time $\tau_{\text{max}} > 2$, that the energy $\Energy_\tau$ of the solution $\phi$ and its derivatives $L^\alpha\phi$ verify certain bounds. 
	\item Using Proposition \ref{prop:peeling}, this gives $L^\infty$ bounds on $\phi$, and its derivatives of the form $L^\alpha \phi$ and $TL^\alpha\phi$. 
	\item We can then estimate the nonlinearity using these $L^\infty$ estimates, which we then feed back into the energy inequality \eqref{eq:engineq} to get an \emph{updated} control on $\Energy_\tau$ for all $\tau \in [2,\tau_{\text{max}}]$. 
	\item Finally, show for sufficiently small initial data sizes, the updated control \emph{improves} the original control, whereupon by the method of continuity the original bounds on $\Energy_\tau$ must hold for all $\tau \geq 2$, implying the desired global existence. 
\end{enumerate}
Before implementing the bootstrap in the following two sections (one each for the cases $d$ being odd or even), we record first basic pointwise bounds on the nonlinearity. 
For estimating the nonlinearity, we observe that 
\begin{equation}
	\partial_{\ub{u}} = \frac{\sqrt{2}u}{u+\ub{u}} T - \frac{1}{u+\ub{u}} L^1 = \frac{1}{u+\ub{u}} (\sqrt{2} u T - L^1).
\end{equation}
This allows us to decompose 
\[ \bgf''(u) (\phi_{\ub{u}})^2 = \frac{1}{(u+\ub{u})^2} \left[ A(u) (L^1 \phi)^2 + B(u) L^1 \phi \cdot T\phi + C(u) (T\phi)^2 \right]\]
where $A, B, C$ are all compactly supported smooth functions of $u$. 
By Proposition \ref{prop:commalg} we can rewrite $L^\alpha T\phi$ as
\[ L^\alpha T\phi = \sum_{|\beta| \leq |\alpha|} \frac{1}{u+\ub{u}} c_\beta L^\beta \phi + \sum_{|\gamma| \leq |\alpha|} c'_\gamma T L^\gamma \phi \]
where $c_\beta, c'_\gamma\in \sweight_0$ and hence are bounded. 
Additionally on the region $\tau \geq 2$ that we are interested in, $u+\ub{u}$ is bounded from below. 
So finally using Lemma \ref{lem:fderest} on the coefficients $A, B, C$ above, we obtain the following uniform pointwise bound on the region $\{\tau \geq 2\}$
\begin{multline}\label{eq:semiLbd}
	\Big| L^\alpha \big[ \bgf''(u) (\phi_{\ub{u}})^2 \big]\Big| \lesssim \sum_{k + \ell_1 + \ell_2 \leq |\alpha|} (1 + \ub{u})^{\frac{k}2 - 2}\cdot \mathbf{1}_{\{u\in \mathrm{supp}\; \bgf''\}} \cdot \\
	  \left| (L^{\leq \ell_1 + 1} \phi) (L^{\leq\ell_2 + 1}\phi) + (T L^{\leq \ell_1}\phi)(T L^{\leq \ell_2}\phi) + (L^{\leq \ell_1 + 1}\phi) (T L^{\leq \ell_2} \phi) \right|.
\end{multline}
\begin{notn}
	Here we denote schematically by $L^{\leq \ell}\phi$ terms of the form $L^\beta\phi$ with $\beta$ an $m$-tuple with $m \leq \ell$. 
\end{notn}

By Proposition \ref{prop:peeling} we can replace the term with the smallest of $\ell_1, \ell_2$ using an energy integral:
\begin{align*}
	\left| (L^{\leq \ell_1 + 1} \phi) (L^{\leq\ell_2 + 1}\phi) \right| &\lesssim \tau^{1-\frac{d}2} \cosh(\rho)^{1-\frac{d}{2}} \Energy_{\tau}[L^{\leq \ell_1 + \lfloor \frac{d}2\rfloor + 1}\phi] \cdot \left| L^{\leq\ell_2 + 1}\phi\right|, \\
	\left| (L^{\leq \ell_1 + 1} \phi) (T L^{\leq\ell_2}\phi) \right| &\lesssim \tau^{1-\frac{d}2} \cosh(\rho)^{1-\frac{d}{2}} \Energy_{\tau}[L^{\leq \ell_1 + \lfloor \frac{d}2\rfloor + 1}\phi] \cdot \left| TL^{\leq\ell_2}\phi\right|, \\
	\left| (TL^{\leq \ell_1} \phi) (L^{\leq\ell_2 + 1}\phi) \right| &\lesssim \tau^{-\frac{d}2} \cosh(\rho)^{1-\frac{d}{2}} \Energy_{\tau}[L^{\leq \ell_1 + \lfloor \frac{d}2\rfloor + 1}\phi] \cdot \left| L^{\leq\ell_2 + 1}\phi\right|, \\
	\left| (TL^{\leq \ell_1} \phi) (TL^{\leq\ell_2}\phi) \right| &\lesssim \tau^{-\frac{d}2} \cosh(\rho)^{1-\frac{d}{2}} \Energy_{\tau}[L^{\leq \ell_1 + \lfloor \frac{d}2\rfloor + 1}\phi] \cdot \left| TL^{\leq\ell_2 }\phi\right|. 
\end{align*}
This allows us to condense \eqref{eq:semiLbd} as
\begin{multline}\label{eq:semiLbd2}
	\Big| L^\alpha \big[ \bgf''(u) (\phi_{\ub{u}})^2 \big]\Big| \lesssim \sum_{\substack{k + \ell_1 + \ell_2 \leq |\alpha|\\ \ell_1 \leq \ell_2}} (1 + \ub{u})^{\frac{k}2 - 2}\cdot \mathbf{1}_{\{u\in \mathrm{supp}\; \bgf''\}} \cdot \\
	  (\underbrace{1 + \ub{u}}_{\mathclap{\approx \tau\cosh(\rho)}})^{2 - \frac{d}{2}} \cdot \Energy_{\tau}[L^{\leq \ell_1 + \lfloor \frac{d}2 \rfloor + 1}\phi]\left[ \frac{1}{\tau\cosh(\rho)} |L^{\leq \ell_2 + 1} \phi| + \frac{1}{\cosh(\rho)}|T L^{\leq \ell_2}\phi|\right].
\end{multline}
Here we used that $\tau\cosh(\rho) = \frac{1}{\sqrt{2}}(u + \ub{u}) \approx (1 + \ub{u})$ using the support properties of $\bgf''$. 
Next we note that $2u\ub{u} \geq \tau^2$ in $\mathcal{I}^+$. On the support of $\bgf''$ this means $\ub{u} \gtrsim \tau^2$. 
On the other hand, from \eqref{eq:suppcond} we also get $\ub{u} \lesssim 1 + \tau^2$. 
This allows us to replace $(1 + \ub{u})$ by $(1 + \tau)^2$ in \eqref{eq:semiLbd2}.

Observe next that since $L^i$ is Killing with respect to $\mink$, we have that firstly $\lieD_T\mink = 0$ and secondly
\[ \Big|\Box L^\alpha \phi\Big| \leq 	\Big| L^\alpha \big[ \bgf''(u) (\phi_{\ub{u}})^2 \big]\Big|.\]
So from the energy identity \eqref{eq:engineq} we get
\[
	\Energy_{\tau_1}[L^\alpha \phi]^2 - \Energy_{\tau_0}[L^\alpha \phi]^2 \lesssim \iint_{\tau\in [\tau_0, \tau_1]} |\Box L^\alpha \phi|\cdot |T L^\alpha \phi| ~\dvol_{\mink}.
\]
Applying \eqref{eq:semiLbd2} we finally arrive at our fundamental \textit{a priori} estimate
\begin{multline*}
	\Energy_{\tau_1}[L^\alpha \phi]^2 - \Energy_{\tau_0}[L^\alpha \phi]^2 \lesssim\\
	\int_{\tau_0}^{\tau_1} \sum_{\substack{k + \ell_1 + \ell_2 \leq |\alpha|\\ \ell_1 \leq \ell_2}}(1 + \tau)^{k-d}\cdot \Energy_{\tau}[L^\alpha \phi]\cdot \Energy_{\tau}[L^{\leq \ell_2} \phi]\cdot \Energy_{\tau}[L^{\leq \ell_1 + \lfloor \frac{d}2\rfloor + 1}\phi] ~\D\tau.
\end{multline*}

To simplify notation, let us write
\begin{equation}
	\mathfrak{E}_k(\tau) = \sup_{\sigma\in [2,\tau]} \Energy_\sigma[L^{\leq k}\phi].
\end{equation}
Our \textit{a priori} estimate reads
\begin{equation}\label{eq:apestsemi}
	\mathfrak{E}_k(\tau)^2 - \mathfrak{E}_k(2)^2 \lesssim \sum_{\substack{\ell_0 + \ell_1 + \ell_2 = k \\ \ell_1 \leq \ell_2}} \int_2^\tau s^{\ell_0 - d} \mathfrak{E}_k \mathfrak{E}_{\ell_1 + \lfloor \frac{d}{2}\rfloor + 1} \mathfrak{E}_{\ell_2} \D{s}.
\end{equation}
In the remainder of this section we will discuss the bootstrap scheme that allows us to control $\mathfrak{E}_k$, for all $k\leq d+1$ when $d$ is even and all $k \leq d$ when $d$ is odd, for all time $\tau \geq 2$. 
Note that the implicit constant in \eqref{eq:apestsemi} depends only on the dimension $d$, the order $k$ of differentiation, and properties of the background function $\bgf$, and is in particular independent of $\phi$.  

\subsection{Bootstrap for \texorpdfstring{$d \geq 6$}{d >= 6} even}
When $d \geq 6$ is even, we will denote by $m$ the value $d/2$. Note that $m \geq 3$. 
We will assume a uniform bound on the initial data
\begin{equation}
	\mathfrak{E}_k(2) \leq \epsilon, \quad k \leq d + 1.
\end{equation}
Our bootstrap assumption is that for some $\delta  > \epsilon$ and for all $2 \leq \tau \leq \tilde{\tau}$, 
\begin{equation}\label{eq:Semi6PEbootstrap}
	\mathfrak{E}_{k}(\tau) \leq 
	\begin{cases}
		\delta & k \leq d - 2 \\
		\delta \tau^{k - (d - 1)} \ln(\tau) & d-1 \leq k \leq d+1
	\end{cases}.
\end{equation}
We note that under \eqref{eq:apestsemi} this system is closed: if $\ell_0 + \ell_1 + \ell_2 \leq d+1$ and $\ell_1 \leq \ell_2$, then $\ell_1 \leq m$. This means that $\ell_1 + \lfloor d/2\rfloor + 1 \leq 2m+1 = d+1$. 
Our goal is to show that the bootstrap assumptions \eqref{eq:Semi6PEbootstrap} can be used to prove improved versions of themselves, under a smallness assumption on $\delta$ and $\epsilon$. 

Under our bootstrap assumptions, we can expression every term of the form 
\[ 
	s^{\ell_0 - d} \mathfrak{E}_k(s) \mathfrak{E}_{\ell_1 + m + 1}(s) \mathfrak{E}_{\ell_2}(s) = w_{\ell_0, \ell_1, \ell_2}(s) \delta^3,
\]
noting that $\ell_1 \leq \ell_2$ by assumption and $\ell_0 + \ell_1 + \ell_2 = k \leq d+1$. 
Observing that at most \emph{one} of $\ell_0$, $\ell_1 + m + 1$, and $\ell_2$ can be $\geq d$ under these conditions, we tabulate upper bounds for the weight functions $w_{\ell_0, \ell_1, \ell_2}(s)$ in Table \ref{tbl:Semi6Pweights}.
\begin{table}[b]
	\caption{\label{tbl:Semi6Pweights} ($d \geq 6$, even) List of admissible $\ell_0, \ell_1, \ell_2$ values as well as the corresponding upper bounds for $w_{\ell_0, \ell_1, \ell_2}$. The value of ``---'' means any value compatible with the prescribed columns. The shaded rows are those with non-integrable upper bounds for $w_{\ell_0, \ell_1, \ell_2}$.}
	\begin{tabular}{cccccl}
		\toprule
		$k$ & $\ell_0$ & $\ell_1$ & $\ell_2$ & $w_{\ell_0, \ell_1, \ell_2}(s) \leq$ & Comment\\
		\midrule
		$\leq d-2$ &  --- & $< m-2$ & --- & $s^{-2}$ & $\implies\ell_0 \leq k$.\\
		$\leq d-2$ & --- & $m-2$ & --- & $s^{2-d} \ln(s)$ & $\implies \ell_0 \leq 2$. \\
		$\leq d-2$ & --- & $m-1$ & --- & $s^{-d} \ln(s)$ & $\implies \ell_0 =0$.\\
		\addlinespace
		$d-1$ & $\leq d-2$ & $\leq m-3$ & $\leq d-2$ & $s^{-2}\ln(s)$ \\
		$d-1$ & --- & $m-2$ & --- & $s^{3-d}\ln(s)^2$ & $\implies \ell_2 \leq m+1, \ell_0 \leq 3$\\
		$d-1$ & --- & $m-1$ & --- & $s^{1-d}\ln(s)^2$ & $\implies \ell_0 \leq 1$\\
		$d-1$ & --- & --- & $d-1$ & $s^{-d}\ln(s)^2$ \\
		\rowcolor[gray]{0.8} $d-1$ & $d-1$ & --- & --- & $s^{-1}\ln(s)$ &\\
		\addlinespace
		\rowcolor[gray]{0.8} $d$ & $\leq d-2$ & $\leq m-3$ & $\leq d-2$ & $s^{-1}\ln(s)$ & \\
		$d$ & --- & $m-2$ & --- & $s^{4-d}\ln(s)^3$ & $\implies \ell_2 \leq m+2, \ell_0 \leq 4$\\
		$d$ & --- & $m,m-1$ & --- & $s^{2-d}\ln(s)^2$ & $\implies \ell_0 \leq 2$\\
		$d$ & --- & --- & $d,d-1$ & $s^{2-d}\ln(s)^2$ \\
		\rowcolor[gray]{0.8} $d$ & $d,d-1$ &---& ---& $s \ln(s)$ & \\
		\addlinespace
		\rowcolor[gray]{0.8} $d+1$ & $\leq d-2$ & $\leq m-3$ & $\leq d-2$ & $\ln(s)$ & \\
		\rowcolor[gray]{0.8} $d+1$ & --- & $m-2$ & --- & $s^{5-d}\ln(s)^3$ & $\implies \ell_2 \leq m+3, \ell_0 \leq 5$\\
		$d+1$ & --- & $m,m-1$ & --- & $s^{3-d}\ln(s)^3$ & $\implies \ell_2 \leq m+2, \ell_0 \leq 3$\\
		$d+1$ & --- & --- & $d,d\pm1$ & $s^{4-d}\ln(s)^2$ \\
		\rowcolor[gray]{0.8} $d$ & $d,d\pm1$ &---& ---& $s^3 \ln(s)$ & \\
		\bottomrule
	\end{tabular}
\end{table}
From this table, we see immediately that $\mathfrak{E}_{k}(\tau)^2 - \mathfrak{E}_{k}(2)^2 \lesssim \delta^3$ whenever $k \leq d-2$. 
Furthermore, using that for $p > -1$
\[ \int s^p \ln(s) ~\D{s} = \frac{1}{p+1} s^{p+1} \ln(s) - \frac{1}{(p+1)^2} s^{p+1} \lesssim s^{p+1} \ln(s)^2,\]
and
\[ \int s^{-1} \ln(s) ~\D{s} = \frac12 \ln(s)^2, \]
we conclude that for $k = d-1, d, d+1$
\[ \mathfrak{E}_k(\tau)^2 - \mathfrak{E}_k(2)^2 \lesssim \delta^3 \tau^{2(k-(d-1))} \ln(\tau)^2.\]
Thus for $\delta$ sufficiently small (depending on the implicit constants in the inequalities above), we have that our bootstrap assumptions \eqref{eq:Semi6PEbootstrap} together with initial data assumptions implies
\begin{equation}
	\mathfrak{E}_{k}(\tau) \leq 
	\begin{cases}
		\sqrt{\epsilon^2 + \frac12\delta^2} & k \leq d - 2 \\
		\sqrt{\epsilon^2 + \frac12 \delta^2 \tau^{2(k - (d - 1))} \ln(\tau)^2} & d-1 \leq k \leq d+1
	\end{cases}.
\end{equation}
By choosing $\epsilon$ sufficiently small relative to $\delta$, we can guarantee
\begin{equation}
	\mathfrak{E}_{k}(\tau) \leq 
	\begin{cases}
		\sqrt{\frac34} \delta & k \leq d - 2 \\
		\sqrt{\frac34} \delta \tau^{k - (d - 1)} \ln(\tau) & d-1 \leq k \leq d+1
	\end{cases},
\end{equation}
thereby closing the bootstrap and proving global existence. 

\subsection{Bootstrap for \texorpdfstring{$d \geq 5$}{d >= 5} odd}
When $d \geq 5$ odd, we will take our bootstrap assumption to be
\begin{equation}
	\mathfrak{E}_{k}(\tau) \leq 
	\begin{cases}
		\delta & k \leq d - 2 \\
		\delta \tau^{k - (d - 1)} \ln(\tau) & k = d-1, d
	\end{cases}.
\end{equation}
That we can close with one fewer derivative is due to $\lfloor d/2\rfloor < d/2$ in this case. 
Define $m = \lfloor d/2 \rfloor$ for convenience; note that $2m = d-1$. 
By our assumption then $\ell_1 \leq \ell_2 \implies \ell_1 \leq m$, and hence $\ell_1 + m + 1 \leq d$, allowing the system to close. 
\begin{table}[b]
	\caption{\label{tbl:Semi5Pweights} ($d \geq 5$, odd) List of admissible $\ell_0, \ell_1, \ell_2$ values as well as the corresponding upper bounds for $w_{\ell_0, \ell_1, \ell_2}$. The value of ``---'' means any value compatible with the prescribed columns. The shaded rows are those with non-integrable upper bounds for $w_{\ell_0, \ell_1, \ell_2}$.}
	\begin{tabular}{cccccl}
		\toprule
		$k$ & $\ell_0$ & $\ell_1$ & $\ell_2$ & $w_{\ell_0, \ell_1, \ell_2}(s) \leq$ & Comment\\
		\midrule
		$\leq d-2$ &  --- & $\leq m-2$ & --- & $s^{-2}$ & $\implies\ell_0 \leq k$.\\
		$\leq d-2$ & --- & $m-1$ & --- & $s^{1-d} \ln(s)$ & $\implies \ell_0 \leq 1$.\\
		\addlinespace
		$d-1$ & $\leq d-2$ & $\leq m-2$ & $\leq d-2$ & $s^{-2}\ln(s)$ \\
		$d-1$ & --- & $m-1$ & --- & $s^{2-d}\ln(s)^2$ & $\implies \ell_0 \leq 2$\\
		$d-1$ & --- & --- & $d-1$ & $s^{-d}\ln(s)^2$ \\
		\rowcolor[gray]{0.8} $d-1$ & $d-1$ & --- & --- & $s^{-1}\ln(s)$ &\\
		\addlinespace
		\rowcolor[gray]{0.8} $d$ & $\leq d-2$ & $\leq m-2$ & $\leq d-2$ & $s^{-1}\ln(s)$ & \\
		$d$ & --- & $m,m-1$ & --- & $s^{4-d}\ln(s)^3$ & $\implies \ell_0 \leq 3$\\
		$d$ & --- & --- & $d,d-1$ & $s^{2-d}\ln(s)^2$ \\
		\rowcolor[gray]{0.8} $d$ & $d,d-1$ &---& ---& $s \ln(s)$ & \\
		\bottomrule
	\end{tabular}
\end{table}
The bootstrap argument here is largely similar to the case $d \geq 6$ even. 
In Table \ref{tbl:Semi5Pweights} we record upper bounds for the weight functions $w_{\ell_0, \ell_1, \ell_2}(s)$, and omit the straightforward remainder of arguments.

\subsection{Bootstrap for \texorpdfstring{$d = 4$}{d = 4}} 
We will assume a uniform bound on the initial data
\begin{equation}
	\mathfrak{E}_{k}(2) \leq \epsilon, \quad k \leq 5,
\end{equation}
with an additional bootstrap assumption for some $\delta > \epsilon$
\begin{equation}\label{eq:semidevenbootstrap}
	\mathfrak{E}_{k}(\tau) \leq 
	\begin{cases}
		\delta & k \leq 2 \\
		\delta \tau^{k - 3  + \gamma} & 3 \leq k \leq 5
	\end{cases}.
\end{equation}
The number $\gamma$ is assumed to be $\ll 1$ and arbitrary; in particular we will throughout take $\gamma < \frac13$.
The smallness of $\gamma$ will impact the smallness of the initial data allowed: the smaller the $\gamma$ the smaller the initial data needs to be. We consider $\gamma$ as fixed once and for all. 

We argue similarly to the case when $d \geq 6$, and record in Table \ref{tbl:Semi4weights} the corresponding weight functions $w_{\ell_0, \ell_1, \ell_2}$. 
Note, however, an additional complication arises since $d/2 + 1 = 3 = d -1$ in this setting (which is why instead of a logarithmic growth of energy $\mathfrak{E}_{d-1}$, we see a small polynomial growth). 

\begin{table}[b]
	\caption{\label{tbl:Semi4weights} ($d = 4$) List of admissible $\ell_0, \ell_1, \ell_2$ values as well as the corresponding upper bounds for $w_{\ell_0, \ell_1, \ell_2}$. The value of ``---'' means any value compatible with the prescribed columns. The shaded rows are those with non-integrable upper bounds for $w_{\ell_0, \ell_1, \ell_2}$. (Recall that $3\gamma < 1$ by fiat.)}
	\begin{tabular}{ccccc}
		\toprule
		$k$ & $\ell_0$ & $\ell_1$ & $\ell_2$ & $w_{\ell_0, \ell_1, \ell_2}(s) \leq$\\
		\midrule
		$\leq 2$ &  --- & $0$ & --- & $s^{\gamma-2}$ \\
		$2$ & 0 & $1$ & 1 & $s^{\gamma -3} $ \\
		\addlinespace
		$3$ & $\leq 2$ & $0$ & $\leq 2$ & $s^{2\gamma-2}$  \\
		$3$ & --- & $1$ & --- & $s^{2\gamma-2}$ \\
		$3$ & 0 & 0 & $3$ & $s^{3\gamma-4}$ \\
		\rowcolor[gray]{0.8} $3$ & $3$ & 0 & 0 & $s^{2\gamma-1}$ \\
		\addlinespace
		\rowcolor[gray]{0.8} $4$ & $4$ & $0$ & $0$ & $s^{2\gamma+1}$  \\
		\rowcolor[gray]{0.8} $4$ & --- & --- & 1 & $s^{2\gamma}$ \\
		\rowcolor[gray]{0.8} $4$ & --- & --- & 2 & $s^{2\gamma-1}$ \\
		$4$ & --- & --- & 3, 4 & $s^{3\gamma-2}$\\
		\addlinespace
		\rowcolor[gray]{0.8} $5$ & $5$ & 0 & 0 & $s^{2\gamma+3}$  \\
		\rowcolor[gray]{0.8} $5$ & --- & --- & 1 & $s^{2\gamma+2}$ \\
		\rowcolor[gray]{0.8} $5$ & --- & --- & 2 & $s^{2\gamma+1}$\\
		\rowcolor[gray]{0.8} 5 & --- & --- & 3,4,5 & $s^{3\gamma}$ \\
		\bottomrule
	\end{tabular}
\end{table}

Based on the weights derived in the table, we see clearly that, by \eqref{eq:apestsemi} we have
\[ \mathfrak{E}_k(\tau)^2 - \mathfrak{E}_k(2)^2 \lesssim \begin{cases}
	\delta^3 & k \leq 2\\
	\delta^3 \tau^{2\gamma + 2k - 6} & k = 3, 4, 5\end{cases} \]
and hence taking $\delta$ sufficiently small and $\epsilon$ even sufficiently smaller will allow us to close the bootstrap and obtain global existence. 

\begin{rmk}
	Applying Proposition \ref{prop:peeling} we see that the corresponding solution has the following decay rates:
	\begin{align*}
		|\phi| &\lesssim (y^0)^{-1}, \\
		|L^i \phi| &\lesssim (y^0)^{\gamma - 1}, &|L^i L^j \phi| &\lesssim (y_0)^{\gamma},  \\
		|T \phi| &\lesssim (y^0)^{\gamma - 1} \tau^{-1}, &|T L^i \phi| &\lesssim (y^0)^{\gamma} \tau^{-1}. 
	\end{align*}
	The difference between the decay rate for $|\phi|$ and the expected $(y^0)^{-3/2}$ is due to our not using the Morawetz $K$ multiplier (see \cite{Wong2017p}) and purely technical.
	The remaining modifications are due to the equation.
	We see at the first derivative level the decay rates shown are modified from the standard linear rate by $(y^0)^\gamma$, while at the second derivative level the decay rates are worse by a factor of $(y^0)^{\gamma + 1}$. (For linear waves in $d = 4$,  $|L^i L^j \phi|$ should decay like $(y^0)^{-1}$.) This worsened decay is a consequence of the background $\bgf''$ that appears in the equation. 
\end{rmk}

\begin{rmk}
	Notice that we do not make use of fractional Sobolev spaces. In the integer setting, to close the $L^2$--$L^\infty$ Sobolev estimate, in 4 dimensions we need to take 3 derivatives. Returning to the schematics described in the introduction of this section, we expect the equation satisfied by $L^{\leq 3}\phi$ to have a right hand side growing like $(1 + \ub{u})^{-1/2}$. 
	Our bootstrap assumptions (as well as was shown in Table \ref{tbl:Semi4weights}) indicate, on the other hand, that the inhomogeneity can take a coefficient growing like $(1 + \ub{u})^{-1 + \epsilon}$ (remember that $\gamma < \frac13$ is fixed and arbitrary). This gain of effectively a power of $1/2$ is due to our use of an anisotropic energy (see Remark \ref{rmk:peeling}) and that on the support of $\bgf''$ the derivative $\partial_{\ub{u}}$ is well-approximated by a ``tangential derivative''. 
\end{rmk}

\subsection{Bootstrap for \texorpdfstring{$d = 3$}{d = 3}}
We close this section by recording the bootstrap argument for $d = 3$. Here the bootstrap assumptions will be taken to be
\begin{equation}
	\mathfrak{E}_{k}(\tau) \leq 
	\begin{cases}
		\delta & k = 0,1 \\
		\delta \tau^{k - 2 + \gamma} & k = 2,3
	\end{cases}.
\end{equation}
Here again $\gamma \ll 1$ is fixed to be $<\frac13$. The weight bounds are shown in Table \ref{tbl:Semi3weights}. 
\begin{table}[b]
	\caption{\label{tbl:Semi3weights} ($d = 3$) List of admissible $\ell_0, \ell_1, \ell_2$ values as well as the corresponding upper bounds for $w_{\ell_0, \ell_1, \ell_2}$. The value of ``---'' means any value compatible with the prescribed columns. The shaded rows are those with non-integrable upper bounds for $w_{\ell_0, \ell_1, \ell_2}$. (Recall that $3\gamma < 1$ by fiat.)}
	\begin{tabular}{ccccc}
		\toprule
		$k$ & $\ell_0$ & $\ell_1$ & $\ell_2$ & $w_{\ell_0, \ell_1, \ell_2}(s) \leq$\\
		\midrule
		$\leq 1$ &  --- & --- & --- & $s^{\gamma-2}$ \\
		\addlinespace
		\rowcolor[gray]{0.8} $2$ & 2 & 0 & 0 & $s^{2\gamma - 1}$  \\
		$2$ & --- & --- & 1 & $s^{2\gamma-2}$ \\
		$2$ & 0 & 0 & $2$ & $s^{3\gamma-3}$ \\
		\addlinespace
		\rowcolor[gray]{0.8} $3$ & $3$ & $0$ & $0$ & $s^{2\gamma+1}$  \\
		\rowcolor[gray]{0.8} $3$ & --- & --- & 1 & $s^{2\gamma}$ \\
		\rowcolor[gray]{0.8} $3$ & --- & --- & 2,3 & $s^{3\gamma-1}$ \\
		\bottomrule
	\end{tabular}
\end{table}
Arguing similarly to the case $d = 4$ we see that the bootstrap assumptions imply
\[ \mathfrak{E}_k(\tau)^2 - \mathfrak{E}_k(2)^2 \lesssim \begin{cases}
	\delta^3 & k \leq 1\\
	\delta^3 \tau^{2\gamma + 2k - 4} & k = 2,3\end{cases} \]
and hence for sufficiently small $\delta$ and $\epsilon$, the bootstrap argument closes and we have global existence. 

For convenience we record here the corresponding $L^\infty$ decay rates relative to the $y$ coordinates. 
These can be obtained by applying Proposition \ref{prop:peeling} to the bootstrap assumptions above. 
	\begin{align*}
		|\phi| &\lesssim (y^0)^{-1/2}, \\
		|L^i \phi| &\lesssim (y^0)^{\gamma - 1/2}, &|L^i L^j \phi| &\lesssim (y_0)^{\gamma+1/2}, \\
		|T \phi| &\lesssim (y^0)^{\gamma - 1/2} \tau^{-1},  &|T L^i \phi| &\lesssim (y^0)^{\gamma+1/2} \tau^{-1}.
	\end{align*}

\begin{rmk}
	An examination of Tables \ref{tbl:Semi6Pweights}, \ref{tbl:Semi5Pweights}, \ref{tbl:Semi4weights}, and \ref{tbl:Semi3weights} shows that, exactly as discussed in the introduction to this section, the nonlinear terms that cause the main difficulty are those where the commutator vector fields hit \emph{entirely} on the background plane-wave $\bgf''$. 
	This shows that even if we start by considering initial data with higher degree of regularity, the loss of decay will always appear in the energy $\mathfrak{E}_{k}$ starting from $k = d-1$. 
\end{rmk}

\begin{rmk}
	In the arguments given above, when $d$ is odd we only commuted with up to $d$ vector fields, and when $d$ is even we used $d+1$ vector fields. 
	It is fairly straightforward to check, in fact, that for initial data of higher regularity, the higher regularity is preserved in the solution. However, for each additional derivative the energy growth speeds up by another factor of $\tau$. 
	So for example, in dimension $d = 3$ the higher energy $\mathfrak{E}_{11}(\tau)$ will have controlled growth like $\tau^{9+\gamma}$ in our bootstrap scheme. 
\end{rmk}

\section{Commuted equations}
\label{sect:commutedeq}
We now return to the membrane equation. 
As discussed in Section \ref{sect:perturb}, to handle the quasilinearity it is convenient to consider not just \eqref{eq:geomeqphi} but also the prolonged system \eqref{eq:lagrange3} for its first derivatives. 
As seen in Section \ref{sect:semilinear} previously, we will prefer to work with the weighted vector field derivatives $L^i\phi$ instead of the coordinate partials $\partial_\lambda \phi$. 
In this section we will first write down the corresponding propagation equations for $L^i\phi$. 

While the arguments in Section \ref{sect:semilinear} sums up neatly our approach toward the semilinear inhomogeneity in the equation, the quasilinear nature of \eqref{eq:geomeqphi} introduces additional complications. 
Whereas in the semilinear case the commutation relations $[L^i, \Box_\mink] = 0$ hold, in the quasilinear case $[L^i, \Box_{\gmetr}]$ or $[L^i, \Box_{\confmetr}]$ are generally non-vanishing second order differential operators, whose coefficients depend on the unknown $\phi$ itself. 
In the second part of this section we perform these basic commutation computations and systematically record the additional terms that arise which would also need to be controlled. 

In the final part of this section, we give a statement of our main stability theorem for simple plane-wave solutions to the membrane equation. 
We will state and prove our theorem in the most critical case $d = 3$. 
Returning to the results of Section \ref{sect:semilinear}, we see that when $d \geq 5$ the solution $\phi$ to the semilinear equation is such that $\phi$ and its first order weighted derivatives $L^i\phi, T\phi$ all enjoy pointwise decay at rates identical to the solution to the linear wave equation. 
For the corresponding quasilinear problem the dynamical metric $\gmetr$ also has fast decay toward $\mink$, and the quasilinearity poses almost no additional complications compared to the semilinear case. 

As already discussed in the introduction to Section \ref{sect:semilinear}, in lower spatial dimensions even the semilinearity causes additional difficulties compared to $d \geq 5$; this requires, in particular, that the decay rates of even the lowest order derivatives $L^i\phi$ and $T\phi$ be modified from their expected linear rates. 
In the quasilinear setting, this causes \emph{additional complications}. 
In three dimensions, in particular, the appearance of terms of the form 
\[  \bgf''(u) \phi \partial^2_{\ub{u}\ub{u}}\phi \]
in \eqref{eq:lagrange3} is potentially troublesome. Based on purely the \emph{linear} peeling estimates, which follows from applying Proposition \ref{prop:peeling} to a solution of the linear wave equation, and which would give (on the support of $\bgf''$)
\[    |\phi| \lesssim (y^0)^{-1/2}, \quad |\partial_{\ub{u}} \phi| \lesssim (y^0)^{-3/2}, \quad |\partial^2_{\ub{u}\ub{u}} \phi | \lesssim (y^0)^{-5/2}, \]
one may naively expect that $\bgf''(u) \phi \partial^2_{\ub{u}\ub{u}}\phi$ has similar decay properties as the semilinear nonlinearity $\bgf''(u) (\phi_{\ub{u}})^2$ that we already treated. 
However, if we instead examine the decay rates proven in Section \ref{sect:semilinear} (which we should not expect to be better), we have
\[  |\phi| \lesssim (y^0)^{-1/2}, \quad |\partial_{\ub{u}} \phi| \lesssim (y^0)^{-3/2+\gamma}, \quad |\partial^2_{\ub{u}\ub{u}} \phi | \lesssim (y^0)^{-3/2+\gamma}, \]
making the decay for $\bgf''(u)\phi \partial^2_{\ub{u}\ub{u}}\phi$ \emph{slower} by a factor of $y^0$ compared to the semilinear term. 

This potential difficulty is significantly ameliorated in $d \geq 5$; doing a similar analysis using the proven decay rates in Section \ref{sect:semilinear} shows that the difference between the quasilinear $\bgf''(u)\phi \partial^2_{\ub{u}\ub{u}}\phi$ term and its semilinear counterpart, when $d = 5,6$ is merely a factor of $\ln(\tau)$ which does not impact the bootstrap argument; and when $d \geq 7$ no difference is present. 
Hence, both for brevity of presentation and clarity of argument, we shall concentrate the remainder of this paper on the most difficult case $d = 3$.
The higher dimensional cases can all be handled similarly; with the difference being mainly one of bookkeeping. 

The overcoming of this potential difficulty with the $\bgf'' \phi \partial^2_{\ub{u}\ub{u}}\phi$ terms in dimension $d = 3$ relies, unsurprisingly, on the ``null structure'' of the equation. 
In Section \ref{sect:semilinear} for brevity of argument the derivatives $L^1\phi$ and $L^i\phi$ for $i = 2,\ldots, d$ are estimated isotropically. 
However, the equations that they satisfy are not the same: recalling that the worst term of the inhomogeneity arises from when the weighted vector fields hit the background $\bgf''$, we expect
\[ \Box L^1 \phi \approx (L^1 \bgf'') (\phi_{\ub{u}})^2 \]
in the semilinear argument. However, a direct computation shows that 
\[ L^1 \bgf'' = u \bgf''' \]
is again a smooth function with compact support in $u$. In particular, while for $i = 2, \ldots, d$ we have the \emph{growing weights} as described in Lemma \ref{lem:fderest}, this loss \emph{is not seen by pure $L^1$ derivatives}. 
Therefore we expect $L^1\phi$ to actually enjoy better decay compared to $L^i\phi$ for $i \neq 1$. 
Finally, returning to the difficult term $\phi_{\ub{u}\ub{u}}$, we see that the $\partial_{\ub{u}}$ derivative lies in the span of $T$ and $L^1$ (see also \eqref{eq:hiddennull} and Remark \ref{rmk:ncvestige}); hence we will expect that $\partial^2_{\ub{u}\ub{u}}\phi$ to decay \emph{faster} than the generic tangential second derivative, allowing us to eventually close our estimates. 

\begin{rmk}\label{rmk:semi:two}
	In $d = 2$ this observation is in fact enough to allow us to close the energy estimate for the \emph{semilinear} model. However, additional difficulties come up in the analysis of the full \emph{quasilinear} problem that cannot be treated using only this method, hence we omit its discussion below. For the semilinear problem \eqref{eq:modelprob}, let us denote by $\mathfrak{E}_k$ the $k$-th order energies for $\phi$, and $\mathfrak{F}_k$ the $k$-th order energies for $L^1\phi$ (analogously to how we proceed in Section \ref{sect:energyBA} below for the quasilinear problem in $d = 3$). This way of treating the equations for $\phi$ and $L^1\phi$ separately allows us to close the global-existence bootstrap in a manner similar to that described in Section \ref{sect:semilinear} with the energy bounds
	\begin{align*}
		\mathfrak{E}_0, \mathfrak{F}_0 & \lesssim \delta, \\
		\mathfrak{E}_1, \mathfrak{F}_1 & \lesssim \delta \tau^{\gamma},\\
		\mathfrak{E}_2, \mathfrak{F}_2 & \lesssim \delta \tau^{1+\gamma},\\
		\mathfrak{E}_3 & \lesssim \delta \tau^{2+\gamma}.
	\end{align*}
	For the quasilinear problem, this scheme breaks down when dealing with the $TT\phi$ derivatives that crop up. 
\end{rmk}

In dimension $d \geq 3$, the cubic and higher nonlinearities are essentially harmless, even with the slightly reduced decay rates. (In the linear case the terms placed in $L^\infty$ combine to decay at least as fast $(y^0)^{-3/2}$; a loss of $\gamma < \frac13$ can be easily absorbed.)
This fact allows us to essentially ignore all ``null structure'' when handling the cubic and higher order terms, which allows us to significantly simplify the bookkeeping involved. 

\subsection{The perturbed system, restated}\label{sect:perturbedsystem}
Our goal this section is to derive the evolution equations for $L^i\phi$. 
Some of the computations are lengthy and not entirely transparent: they are recorded in Appendix \ref{sect:app:ps}.
We start with \eqref{eq:geomeqphi} which we re-write as
\[ \sqrt{|\gmetr|} \partial_{\mu}  \frac{\gmetrf^{\mu\nu}\partial_\nu \phi}{\sqrt{|\gmetr|}} = \bgf'' (\phi_{\ub{u}})^2.\]
We expand the left hand side as
\begin{multline*} \Box_\mink \phi + 2 \partial_{\ub{u}} \big( \phi \bgf'' \partial_{\ub{u}}\phi \big) + \sqrt{|\gmetr|} \gmetrf( \D\phi, \D |\gmetr|^{-1/2}) \\
	= \Box_{\mink} \phi + 2 \partial_{\ub{u}}  \big( \phi \bgf'' \partial_{\ub{u}}\phi \big) - \frac1{2|\gmetr|} \gmetrf\big( \D\phi, \D\big( \gmetrf(\D\phi, \D\phi) \big)\big).
\end{multline*}
Notice, on the other hand, that 
\begin{multline}\label{eq:boxgexpr} \Box_\gmetr \psi = \Box_{\mink}\psi + 2 \partial_{\ub{u}} \big( \phi \bgf'' \partial_{\ub{u}} \psi \big) - \frac{1}{|\gmetr|} \gmetrf\big(\D\phi, \D\big(\gmetrf(\D\phi, \D\psi)\big)\big)\\
	- \frac{1}{|\gmetr|} \bgf'' (\phi_{\ub{u}})^2\cdot \gmetrf(\D\phi,\D\psi) + \frac1{2|\gmetr|} \gmetrf(\D|\gmetr|, \D\psi).
\end{multline}
Together this implies that, if $X$ is a Killing vector field of the Minkowski metric $\mink$, that
\begin{multline}\label{eq:geomeqphiX}
	\Box_{\gmetr} X\phi = \boxed{X \big( \bgf'' (\phi_{\ub{u}})^2 \big)} - \frac{1}{|\gmetr|} \bgf'' (\phi_{\ub{u}})^2 \gmetrf(\D\phi, \D(X\phi)) + \frac{1}{|\gmetr|} \gmetrf(\D|\gmetr|, \D(X\phi)) \\
	- \boxed{2 [X, \partial_{\ub{u}}](\phi \bgf'' \phi_{\ub{u}})} - \boxed{2 \partial_{\ub{u}}( X(\phi \bgf'') \phi_{\ub{u}})} - \boxed{2 \partial_{\ub{u}}(\phi \bgf'' [X,\partial_{\ub{u}}] \phi)} \\
	- \frac{1}{2|\gmetr|^2} X(|\gmetr|) \gmetrf(\D\phi, \D|\gmetr|) + \frac{1}{2|\gmetr|}\lieD_X(\gmetrf^{-1})(\D\phi, \D|\gmetr|)\\
	+ \frac{1}{2|\gmetr|} \gmetrf\big(\D\phi, \D\big(\lieD_X(\gmetrf^{-1})(\D\phi, \D\phi)\big)\big).
\end{multline}
Here, $\lieD_X\gmetrf^{-1}$ is the Lie derivative of the inverse metric $\gmetrf^{-1}$ by the vector field $X$. It can be given as
\begin{equation}
	\lieD_X \gmetrf^{-1} = 2 X(\phi\bgf'') \partial_{\ub{u}} \otimes \partial_{\ub{u}} + 
		 2 \phi\bgf'' [X,\partial_{\ub{u}}] \otimes \partial_{\ub{u}} +
		  2\phi\bgf''\partial_{\ub{u}} \otimes [X,\partial_{\ub{u}}].
\end{equation}
The boxed terms in \eqref{eq:geomeqphiX} are those with \emph{quadratic nonlinearity} and are the ones for which the null structure play an important role. The remaining terms on the right hand side all have cubic or higher nonlinearities, and will be treated more roughly in the estimates. 

Later on we will take $X$ to be one of $L^i$; we can compute the commutators (see \eqref{eq:def:L1} and \eqref{eq:def:Li} for definitions)
\begin{align}
	[L^1, \partial_{\ub{u}}] &= - \partial_{\ub{u}};\\
	[L^i, \partial_{\ub{u}}] & = - \frac{1}{\sqrt{2}} \frac{1}{y^0}(L^i - y^i T), && i\in \{2, \ldots, d\}.
\end{align}

For convenience, we will introduce the following \emph{schematic notations}. 
\begin{notn}\label{notn:PfWf}
First, in view of Lemma \ref{lem:fderest}, we will denote by $\pwWf_m$ any finite sum of terms of the form
\begin{equation}
	\Big(\text{Polynomial in }\{\ub{u}, \hat{x}\}\Big) \cdot \Big(\text{Compactly supported smooth function of $u$}\Big)
\end{equation}
such that on $\mathcal{I}^+$ it is bounded by $(1 + \ub{u})^{m/2}$. 
Our assumptions imply $\bgf'' = \pwWf_0$. 
The computations surrounding the proof of Lemma \ref{lem:fderest} imply that 
\begin{equation}
	T \pwWf_m = \pwWf_m\;, \qquad L^1 \pwWf_m = \pwWf_m\;, \qquad L^i \pwWf_m = \pwWf_{m+1}\; \text{ for }i\in \{2, \ldots, d\}.
\end{equation}
We will denote by $\swtf_m$ any element of $\sweight_m$. 
\end{notn}
With these notations, we can rewrite schematically 
\begin{equation}\label{eq:hiddennull}
	\pwWf_m \partial_{\ub{u}} = \swtf_1  \pwWf_m \big( L^1 +  T\big).
\end{equation}
\begin{rmk}[Vestige of null condition] \label{rmk:ncvestige}
	As discussed in Remark \ref{rmk:resonances}, the presence of the $\bgf''$ factor in $\bgf'' (\phi_{\ub{u}})^2$ helps to ameliorate the resonant interaction. This improvement is a vestige of the null condition of the original membrane equation. 
	In our reformulation here, this improvement is captured in \eqref{eq:hiddennull} above. 
	Observe that a generic coordinate derivative $\partial_{u}$, $\partial_{\ub{u}}$, or $\partial_{\hat{x}}$ can be written only as an element of the commutator algebra $\algcomm_1$, which means that the \emph{transversal} factor $T$ is not accompanied by a decaying weight. 
	From this one can see that quadratic terms of the form $(T\phi)^2$ will serve as a severe obstacle to global existence. 
	In our setting, however, the $\pwWf_0$ weight $\bgf''$ provides a spatial localization and gives an \emph{anomalous weighting}: the term $\swtf_1 T \in \algcomm_2$ and has \emph{improved decay} and this improvement is, fundamentally, what allows our argument to close in this paper. 
\end{rmk}
\begin{notn}\label{notn:algdo}
We will frequently denote by $\algdo^{k,s}_{w}$ an element of $\algdiff^{k,s}_{w'}$ with $w' \geq w$. 
When $|\gmetr|$ appears in a higher order term, if is often sufficient to control it as
\begin{equation}
	|\gmetr| = 1 + (\algdo^{1,1}_1\phi)^2(1 + \pwWf_0 \phi),
\end{equation}
and similarly we can write
\begin{equation}
	\gmetrf(\D\varphi, \D\psi) =  (\algdo^{1,1}_1\varphi)(\algdo^{1,1}_1\psi)(1 + \pwWf_0 \phi).
\end{equation}
\end{notn}

\begin{rmk}\label{rmk:TTphi}
	Observe that in \eqref{eq:geomeqphiX}, the inhomogeneity depends on \emph{up to second order derivatives of $\phi$}. If we decompose nonlinearities, the second order derivatives that appear are \emph{generic}, in the sense that derivatives with principal parts $TT\phi$, $TL^i\phi$, and $L^iL^j\phi$ all appear. (Note that $\{T, L^i\}$ span the tangent space $\Real^{1,d}$.) To control $TL^i\phi$ and $L^iL^j\phi$ in $L^\infty$, by Proposition \ref{prop:peeling} it suffices to control the energies of $L^\alpha \phi$. 
	The term $TT\phi$, however, is \emph{not} controlled by these energies.  
	There are two approaches to address this. 
	First is to enlarge the set of commutators required; instead of only commuting with the boosts $L^\alpha$, one can commute with also the $T$ vector field. 
	Checking the commutator relations, to close this argument one would have to commute with all differential operators of the form $L^\alpha T^k$ where $|\alpha| + k$ is bounded by some $k_0$. 
	For our problem, it appears slightly simpler computationally to take the second (essentially equivalent) alternative. 
	By decomposing $\Box_g$ we can solve \eqref{eq:geomeqphi} for $TT\phi$ in terms of $TL^i\phi$ and $L^iL^j\phi$ and lower order derivatives. 
	This implies $TTL^\beta \phi$ can be estimated in terms of $T L^\gamma \phi$ and $L^\alpha \phi$ where $|\gamma| \leq |\beta| + 1$ and $|\alpha | \leq |\beta| + 2$. 
	See Appendix \ref{sect:app:ps} for the details of this computation.
\end{rmk}

\begin{notn}\label{notn:smtf}
	We will denote by $\smtf = \smtf(\phi\pwWf_0, \algdo^{1,0}_0 \phi, \algdo^{1,1}_1\phi)$ an arbitrary smooth function of its arguments. In particular, $|\gmetr| = \smtf$ in this notation, as well as $|\gmetr|^{-1} = \smtf$ when the $\phi, \algdo^{1,0}_0 \phi$, and $\algdo^{1,1}_1\phi$ are all sufficiently small. 
\end{notn}
It is convenient to simplify \eqref{eq:geomeqphiX} a bit more. 

With the aid of these schematic notations, we find that $L^1\phi$ satisfies
\begin{multline}\label{eq:geomeqL1phi}
\Box_\gmetr L^1\phi = \pwWf_0 \swtf_2 \cdot \Bigl[ (L^1\phi + T\phi)^2 + T\phi(L^1L^1\phi + TL^1\phi) \\
	+ \phi (L^1\phi + T\phi) + (\phi + L^1\phi) (L^1 L^1 \phi + T L^1\phi + T T\phi) \Bigr] \\
+ \smtf (\algdo_1^{1,1}\phi)(\algdo_1^{1,1}L^1\phi) (\algdo_2^{2,2}\phi) + \smtf \pwWf_0 (\algdo_0^{1,0}\phi)(\algdo_1^{1,1}\phi)^2 (\algdo_2^{2,2}\phi)\\
	+ \smtf \pwWf_0 (\algdo_1^{1,1}\phi)^3(\algdo_1^{2,1}\phi) + \smtf\pwWf_0 (\algdo_1^{1,1}\phi)^5(\algdo_2^{2,2}\phi) + \smtf \pwWf_0 (\algdo_0^{1,0}\phi)(\algdo_1^{1,1}\phi)^4\\
	+ \smtf \phi \swtf_1\pwWf_1 (\algdo_1^{1,1}\phi)^2(\algdo_1^{2,1}\phi) +   \smtf \phi^2 \swtf_2\pwWf_2 (\algdo_1^{1,1}\phi)^4(\algdo_1^{2,1}\phi) \\
	+ \smtf \phi \swtf_1\pwWf_1 (\algdo_1^{1,1}\phi)^5(\algdo_2^{2,2}\phi) + \smtf \swtf_1 \pwWf_1 (\algdo_0^{1,0}\phi)(\algdo_1^{1,1}\phi)^3
\end{multline}
The first brackets capturing all the quadratic nonlinearities and the cubic and higher nonlinearities are described schematically after. 
For $i \neq 1$, the term $L^i\phi$ satisfies the equation
\begin{multline}\label{eq:geomeqLiphi}
\Box_\gmetr L^i \phi= \pwWf_0 \swtf_1(\phi + L^1\phi + T \phi)(\algdo_1^{1,1} L^1\phi + \algdo_2^{2,2}\phi + \algdo_1^{1,1}\phi)\\
	+ \swtf_2 (\pwWf_0 L^i\phi + \pwWf_1 \phi)(L^1 L^1\phi + TL^1\phi + TT\phi + L^1\phi + T\phi) \\
	+ \pwWf_1 \swtf_2 (\phi + L^1\phi + T\phi)(L^1\phi + T\phi) \\
+  \smtf (\algdo_1^{1,1}\phi)(\algdo_1^{2,1}\phi) (\algdo_2^{2,2}\phi) + \smtf \pwWf_0 (\algdo_0^{1,0}\phi)(\algdo_1^{1,1}\phi)^2 (\algdo_2^{2,2}\phi)\\
	+ \smtf \pwWf_0 (\algdo_1^{1,1}\phi)^3(\algdo_1^{2,1}\phi) + \smtf\pwWf_0 (\algdo_1^{1,1}\phi)^5(\algdo_2^{2,2}\phi) + \smtf \pwWf_0 (\algdo_0^{1,0}\phi)(\algdo_1^{1,1}\phi)^4\\
	+ \smtf \phi \swtf_1\pwWf_1 (\algdo_1^{1,1}\phi)^2(\algdo_1^{2,1}\phi) +   \smtf \phi^2 \swtf_2\pwWf_2 (\algdo_1^{1,1}\phi)^4(\algdo_1^{2,1}\phi) \\
	+ \smtf \phi \swtf_1\pwWf_1 (\algdo_1^{1,1}\phi)^5(\algdo_2^{2,2}\phi) + \smtf \swtf_1 \pwWf_1 (\algdo_0^{1,0}\phi)(\algdo_1^{1,1}\phi)^3 \\
+ \smtf \phi \pwWf_1 (\algdo_1^{1,1}\phi)^2 (\algdo_2^{2,2}\phi) 
		+ \smtf \phi \swtf_1 \pwWf_2 (\algdo_1^{1,1}\phi)^3 + \smtf \pwWf_1 (\algdo_1^{1,1}\phi)^4.
\end{multline}
Note that the cubic and higher-order terms are schematically represented largely in the same way, with the main differences coming in the quadratic terms. 
The key observation, as already mentioned in the introduction to this section, is that the quadratic terms in the equation for $L^1\phi$ do not see the growing weight term, and therefore behaves like $\phi$ instead of a generic $L\phi$ term. 
This improvement then also propagates into the analysis of the quadratic terms of equation \eqref{eq:geomeqLiphi} of the general $L$ derivatives.

For convenience, we record \eqref{eq:geomeqphi} here in the schematic notation. 
\begin{equation}\label{eq:geomeqphiS}
	\Box_\gmetr \phi = \smtf \pwWf_0\swtf_2 (L^1 \phi + T\phi)^2.
\end{equation}

\subsection{Commutator relations}\label{sect:commutatorrels}

To use the vector field method, we will be commuting our equations with the $L^i$ derivatives. More precisely, we study the wave equations satisfied by $\algdo_0^{k,0}(L^1\phi, L^i\phi)$ by writing
\[ \Box_\gmetr (\algdo_0^{k,0} L\phi) = \algdo_0^{k,0} (\Box_\gmetr L\phi) + [\algdo_0^{k,0},\Box_g] (L\phi).\]
Note that after applying \eqref{eq:geomeqL1phi} and \eqref{eq:geomeqLiphi} the right-side does not contain principal terms. 
Differentiation of the schematic relations in \eqref{eq:geomeqL1phi}, \eqref{eq:geomeqLiphi}, and \eqref{eq:geomeqphiS} are straightforward. 
To implement our strategy, we need to compute the commutators $[X, \Box_g]$ acting on a smooth scalar $\psi$, where $X = L^1$ or $L^i$. 
We merely record the results here, and defer the actual computation to \ref{sect:app:cr}.
\begin{multline}\label{eq:commtermlist}
	[X,\Box_g]\psi = \pwWf_0 \swtf_1 (\algdo_1^{1,1}\phi)(L^1\psi + T\psi) + \pwWf_0 \swtf_1 (\phi + L^1\phi + T\phi) (\algdo_1^{1,1}\psi)\\
	+ \pwWf_0 \swtf_1 (\algdo_0^{1,0}\phi) (\algdo_1^{1,1} L^1\psi + \algdo_1^{1,1} T\psi) + \pwWf_0 \swtf_2 (\algdo_0^{1,0} L^1 \phi + \algdo_0^{1,0} T\phi)(L^1\psi + T\psi) \\
	+\pwWf_1 \swtf_2 (\phi + L^1\phi + T\phi) (L^1\psi + T\psi) + \pwWf_1 \swtf_2 \phi (L^1 L^1\psi + T L^1 \psi + T T\psi) \\
	+(X\smtf)\cdot \bigl[ (\algdo_1^{1,1}\phi)(\algdo_2^{2,2}\phi)(\algdo_1^{1,1}\psi) + (\algdo_1^{1,1}\phi)^2 (\algdo_2^{2,2}\psi) \\
	+ \pwWf_0 (\algdo_1^{1,1}\phi)^3 (\algdo_1^{1,1}\psi) + \pwWf_1 \swtf_1\phi (\algdo_1^{1,1}\phi)^2 (\algdo_1^{1,1}\psi)\bigr]\\
	+ \smtf \cdot \bigl[ (\algdo_1^{2,1}\phi)(\algdo_2^{2,2}\phi)(\algdo_1^{1,1}\psi) + (\algdo_1^{1,1}\phi)(\algdo_2^{3,2}\phi)(\algdo_1^{1,1}\psi) \\
		+ (\algdo_1^{1,1}\phi)(\algdo_1^{2,1}\phi)(\algdo_2^{2,2}\psi) + \pwWf_0 (\algdo_1^{1,1}\phi)^2(\algdo_1^{2,1}\phi) (\algdo_1^{1,1}\psi) \\
		+ \pwWf_1 (\algdo_1^{1,1}\phi)^3 (\algdo_1^{1,1}\psi) + \pwWf_1 \swtf_1\phi (\algdo_1^{1,1}\phi_)(\algdo_1^{2,1}\phi)(\algdo_1^{1,1}\psi) \\
	+ \pwWf_1 \swtf_1 (\algdo_0^{1,0}\phi)(\algdo_1^{1,1}\phi)^2 (\algdo_1^{1,1}\psi) + \pwWf_2 \swtf_1 \phi (\algdo_1^{1,1}\phi)^2 (\algdo_1^{1,1}\phi)\bigr].
\end{multline}
Notice that the quadratic terms (linear in both $\phi$ and $\psi$) are listed explicitly, as we expect to need to use the null structure to extract sufficient decay. The cubic and higher terms (which are at least quadratic in the background $\phi$), are listed purely schematically. 

\begin{rmk}\label{rmk:commterm:simpl}
	Now and in the sequel, $\mathrm{HO}_1$ constitutes the cubic and higher order terms that arise in the right hand side of \eqref{eq:geomeqL1phi}, see also Appendix \ref{appen:HO1}. Similarly, $\mathrm{HO}_i$ for equation \eqref{eq:geomeqLiphi}, see also Appendix \ref{appen:HOi}. A key thing to note about the commutator relation \eqref{eq:commtermlist} is that, with $\psi = L^\alpha L^1\phi$ for some multi-index $\alpha$, every \emph{cubic and higher} term that appears in the schematic decomposition above can be obtained, schematically, as a term that appears in an $L^{\leq |\alpha|+1}$ derivative of $\mathrm{HO}_1$. And similarly with $\psi = L^\alpha \phi$ every cubic and higher term in the schematic decomposition is a term that appears in an $L^{\leq|\alpha|}$ derivative of $\mathrm{HO}_i$. (The only difference being our schematic treatment of the purely cubic term; see Remark \ref{app:rmk:hoDN}.). Thus we will not separately treat the cubic and higher terms that arise from the commutator in our analyses later, and absorb it as part of the general discussion of higher order terms. 

	Similarly, with $\psi = L^\alpha\phi$ all the quadratic terms that appear in \eqref{eq:commtermlist} can be obtained from $L^{\leq |\alpha|}$ derivatives hitting on $\mathrm{QN}_i$, which are defined as the quadratic inhomogeneity of \eqref{eq:geomeqLiphi}. 
	However as we can see in the case $\psi = L^\alpha L^1$, the final quadratic commutator term of the form $\pwWf_1 \swtf_2 \phi (L^1 L^1 \psi+ TL^1\psi + TT\psi)$ cannot be obtained as an $L^{\leq |\alpha| + 1}$ derivative of $\mathrm{QN}_1$, which are defined as the quadratic homogeneity of \eqref{eq:geomeqL1phi} (notice the differing weights $\pwWf_0$ and $\pwWf_1$). These turn out to be the most delicate terms in the analysis, and in Section \ref{ssect:qno} will be the main terms to saturate the polynomial growth in the energy estimates.   
\end{rmk}

\subsection{Statement of the main theorem}

Our main theorem asserts that when the initial plane-wave $\bgf$ has bounded width, then this travelling wave solution is stable under small compactly supported perturbations. 
By rescaling and translating we can assume the perturbation is supported in the unit ball $B(0,1)\subset \Real^3$ on the spatial slice $\{y^0 = 2\}$. 

\begin{thm}\label{thm:mainthmQ}
	Let $d = 3$ and assume $\bgf(u)$ is such that $\bgf''$ has compact support in $u$. 
	Consider the initial value problem for \eqref{eq:geomeqphi}, where the dynamical metric is given by \eqref{eq:dynmet}. 
	We assume the initial data is prescribed on the spatial slice $\{y^0 = 2\}$ by 
	\[ \phi |_{y^0 = 2} = \phi_0, \quad \partial_{y^0} \phi |_{y^0 = 2} = \phi_1, \]
	where $\phi_0, \phi_1 \in C^\infty_0(B(0,1))$. 
	Then for any $\gamma > 0$ there exists some $\epsilon_0 > 0$ (which we allow to depend on $\bgf$ and on $\gamma$) such that whenever
	\[ \|\phi_0\|_{H^{5}} + \|\phi_1\|_{H^4} \leq \epsilon_0 \]
	the solution $\phi$ exists for all time $y^0 \geq 2$. Furthermore, we have the following uniform bounds on the solution and its derivatives:
	\begin{gather*}
		|\phi| + |L^1\phi| \lesssim (y^0)^{-1/2}, \\
		|T\phi| + |T L^1\phi| \lesssim \tau^{\gamma-1}(y^0)^{-1/2} , \\
		|\algdo_0^{1,0} \phi| + |\algdo_0^{1,0} L^1\phi| + |T\algdo_0^{1,0}\phi| + |TT\phi| \lesssim \tau^\gamma (y^0)^{-1/2}, \\
		|\algdo_0^{2,0} \phi| + |\algdo_1^{3,1}\phi| \lesssim  \tau^{1+\gamma}(y^0)^{-1/2}. \\
	\end{gather*}
\end{thm}

\begin{rmk}\label{rmk:lostofpeeling}
Observe that in particular, the coordinate derivatives (with respect to $y$) up to second order all decay uniformly as $y^0 \nearrow \infty$.
As will be clear from the proof, if the initial data has higher regularity the regularity persists for the solution. 
This can be extended to show that (the details of the proof we omit here) that arbitrary order coordinate derivatives of the solution decay uniformly like $(y^0)^{-1/2 + \gamma}$. 
Peeling, however, doesn't hold to arbitrary orders, unlike the case of the linear wave. If we denote by $\bar{\partial}$ a derivative that is tangential to out-going Minkowski light-cones, our results are only compatible with these outgoing tangential derivatives  $\bar{\partial}^\beta \phi$ being uniformly bounded by $(y^0)^{-3/2 + \gamma}$ \emph{for all orders} $|\beta| \geq 2$. 
\end{rmk}

\section{Energy quantities and bootstrap assumptions}
\label{sect:energyBA}

The remainder of this paper is devoted to proving Theorem \ref{thm:mainthmQ}. 
In view of the robust local existence theory for nonlinear wave equations, the strategy we will take is that of a standard bootstrap argument. 
In this section we will set the notations for the basic energy quantities and perform some preliminary analyses on them, having also introduced the main bootstrap assumptions.

\subsection{The energy quantities defined; bootstrap assumptions}
Recall from \eqref{eq:def:eng2} the energy quantity
\[ \Energy_\tau[\psi;\gmetr]^2 = 2\int_{\Sigma_\tau} \frac{1}{\sqrt{|g(\D\tau,\D\tau)|}} \seten[\psi;\gmetr](T, (-\D\tau)^{\gmetr\sharp}) ~\dvol_{\siggmetr_\tau}\]
which satisfies the basic energy inequality \eqref{eq:engineq} for $\tau_0 < \tau_1$
\begin{equation}\label{eq:engineq:restate}
	\Energy_{\tau_1}[\psi;g]^2 \leq \Energy_{\tau_0}[\psi;g]^2 + \iint_{\tau \in [\tau_0,\tau_1]} \big|\seten[\psi;g]:_g \lieD_Tg\big| + 2\big|\Box_g \psi \cdot T(\psi)\big| ~\dvol_{g}.
\end{equation}
Here $\psi$ will stand for some higher $L$ derivative of the solution $\phi$. 
One difference between our quasilinear setting and the semilinear model treated in Section \ref{sect:semilinear} is the presence of the first integrand in the energy inequality. In the semilinear case $\lieD_T\mink = 0$.
The analysis of the second integrand will occupy Section \ref{sect:inhom}, using the equations \eqref{eq:geomeqphiS}, \eqref{eq:geomeqL1phi}, \eqref{eq:geomeqLiphi}; we treat the first integrand here. 

The integrand can be expanded as
\[ \seten[\psi;\gmetr] :_\gmetr \lieD_T \gmetr = (\lieD_T\gmetr^{-1})(\D\psi,\D\psi) - \frac12 \gmetr^{-1}(\D\psi,\D\psi)\cdot  \gmetr :_\gmetr \lieD_T \gmetr.\]
We primarily care about terms that are linear in $\phi$: the terms with higher order dependence on $\phi$ we expect to behave better and will estimate very roughly. 
With that, and \eqref{eq:InvMet} in mind, schematically 
\begin{multline*}
	(\lieD_T\gmetr^{-1})(\D\psi,\D\psi) = (\phi + T\phi) \pwWf_0 \swtf_2 (L^1\psi + T\psi)^2 \\
	+ \smtf \left[ (\algdo_1^{1,1}\phi)(\algdo_2^{2,2}\phi) + (\algdo_1^{1,1}\phi)^2(1 + T\phi \pwWf_0) \right] (\algdo_1^{1,1}\psi)^2.
\end{multline*}
And we also have schematically, by \eqref{eq:dynmet}, that
\begin{align*}
	\gmetr :_\gmetr \lieD_T \gmetr &= \gmetr^{-1}(\D\phi, \D T\phi) + \gmetr^{-1}(\D u, \D u) (\phi + T\phi)\pwWf_0\\
				       &= \smtf (\algdo_1^{1,1}\phi)(\algdo^{2,2}_2\phi) + \smtf \pwWf_0 (\algdo_1^{1,1}\phi)^2(\phi + T\phi).
\end{align*}
Therefore we can conclude that schematically
\begin{multline}\label{eq:schema:deftenterm}
	\seten[\psi;\gmetr] :_\gmetr \lieD_T \gmetr = (\phi + T\phi) \pwWf_0 \swtf_2 (L^1\psi + T\psi)^2 \\
	+ \smtf \left[ (\algdo_1^{1,1}\phi)(\algdo_2^{2,2}\phi) + (\algdo_1^{1,1}\phi)^2(1 + T\phi \pwWf_0) \right] (\algdo_1^{1,1}\psi)^2.
\end{multline}
We will return to estimating this term in Section \ref{sect:deftenTerm} 

For convenience, for $\tau \geq 2$ and $k$ a non-negative integer, we will denote by 
\begin{align}
	\mathfrak{E}_k(\tau) &\eqdef \sup_{\sigma\in [2,\tau]} \Energy_\sigma[L^{\leq k}\phi; \gmetr], \\
	\mathfrak{F}_k(\tau) &\eqdef \sup_{\sigma\in [2,\tau]} \Energy_\sigma[L^{\leq k} L^1\phi; \gmetr].
\end{align}
We will make the following \emph{initial data assumption}:
\begin{equation}\label{eq:ass:id}
	\mathfrak{E}_4(2) + \mathfrak{F}_3(2) \leq \epsilon \tag{A\textsubscript{ID}}
\end{equation}
for some $\epsilon \geq \epsilon_0$.
We can make this assumption as by the standard local-existence argument for nonlinear wave equations, with the assumptions in Theorem \ref{thm:mainthmQ}, for sufficiently small $\epsilon_0$ the solution necessarily exists up to $\Sigma_2$. 
The continuity of the energy norms on initial data implies that as $\epsilon_0 \to 0$ the quantity $\mathfrak{E}_4(2) + \mathfrak{F}_3(2)\to 0$ also. 

As is typical of bootstrap arguments, we will assume there is some $T > 2$ such that for every $\tau \in [2,T]$ the following \emph{bootstrap assumptions} hold. 
We need three parameters: $\delta_0 > 0$ whose size will be fixed in Section \ref{sect:geoconBA} and considered constant afterwards; $\delta \in (0, \delta_0)$ which is a smallness parameter we will adjust to close the bootstrap. 
Without loss of generality we will assume $\gamma \in (0,1/4)$ is fixed throughout the argument. 
Our goal, as usual, is to demonstrate that the bootstrap assumptions below leads to improved versions of themselves, when $\delta$ and $\epsilon$ are taken to be sufficiently small. 
This then implies by standard continuity argument that the assumptions in fact hold for all times $\tau > 2$ and we obtain global existence. 

Our bootstrap assumptions are:
First, along $\Sigma_\tau$ we have the uniform bounds  
\begin{gather}\label{eq:ass:bainf}
	\left\{ \quad
\begin{aligned}
	|\phi| &\leq \delta_0 (y^0)^{-1/2}; \\
	|L^1\phi| & \leq \delta_0 (y^0)^{-1/2};\\
	|L^i\phi| & \leq \delta_0 (y^0)^{-1/2} \tau^\gamma;\\
	|T\phi| & \leq \delta_0 (y^0)^{-1/2} \tau^{\gamma - 1}.
\end{aligned}\right. \tag{BA\textsubscript{$\infty$}}
\end{gather}
Second, we assume that  
\begin{gather}\label{eq:ass:ba2}
	\left\{\quad
\begin{aligned}
	\mathfrak{E}_1(\tau) + \mathfrak{F}_1(\tau) & \leq \delta;\\
	\mathfrak{E}_2(\tau) + \mathfrak{F}_2(\tau) & \leq \delta \tau^\gamma;\\
	\mathfrak{E}_3(\tau) + \mathfrak{F}_3(\tau) & \leq \delta \tau^{1+\gamma};\\
	\mathfrak{E}_4(\tau) & \leq \delta \tau^{2 + \gamma}.\\
\end{aligned} \right. \tag{BA\textsubscript{2}}
\end{gather}

\subsection{Inequalities on that we use frequently} In the subsequent analysis, we will freely use the control of $y^0,\ \cosh(\rho),$ and $\ub{u}$ afforded by Lemma \ref{lem:decayonsupport}. As we will see, these estimates will be an important tool to obtain coercive control (with respect to $\mathfrak{E}_k, \ \mathfrak{F}_k$) of terms that arise in the energy estimates. They also have important consequences when used concurrently with the bootstrap assumptions, see, for instance, Proposition \ref{prop:pointwiseconBAinf}.

\begin{lem}\label{lem:decayonsupport}
The following estimates hold on $\mathcal{I}^+\cap \{\supp \phi \}\cap\{ \supp \pwWf_0\}$
\begin{align}
\ub u & \approx y^0 \label{eq:u=y0} \\ 
y^0 & \approx \tau^2  \label{eq:y0=tau^2} \\
\cosh(\rho) & \approx \tau. \label{eq:cosh=tau}
\end{align}
\end{lem}
\begin{proof}
Using $y^0 = (u + \ub u)/\sqrt{2}$, \eqref{eq:u=y0} follows because $\pwWf_0$ has compact support in $u$. Under the assumptions of the initial data in Theorem \ref{thm:mainthmQ}, finite speed of propagation implies that 
\[ \sqrt{|y^1|^2 + |y^2|^2 + |y^3|^2 } \le |y^0 - 2| + 1= y^0 - 1\] 
on the support of $\phi$. Since $\tau^2 = 2 u \ub u - |\hat x|^2 = (y^0)^2 - (y^1)^2 - (y^2)^2 - (y^3)^2$,  the previous inequality reads $2y_0 \le \tau^2 + 1$ and hence $y^0 \lesssim \tau^2$ because $\tau \ge 2$. Since $2u\ub u \ge \tau^2$ on $\mathcal{I}^+$ (see Section \ref{sect:globsob}),
\[ \tau^2 \le 2 u\ub u \lesssim \ub u\approx y^0\]
by appealing to the support of $\pwWf_0$. We have then proved \eqref{eq:y0=tau^2}. Finally, \eqref{eq:cosh=tau} follows from the identity $\tau \cosh(\rho) = y^0$ and \eqref{eq:y0=tau^2}.
\end{proof}

\subsection{Some first consequences of \texorpdfstring{(\ref*{eq:ass:bainf})}{(BA-infty)}}\label{sect:geoconBA}
The assumptions \eqref{eq:ass:bainf} are not strictly speaking necessary; its presence however helps jump-start basic geometric comparisons that simplifies especially the energy comparisons to be taken in the next subsection. 

\begin{prop}\label{prop:pointwiseconBAinf}
	The assumptions \eqref{eq:ass:bainf} imply
	\begin{gather*}
		|\pwWf_0 \phi| \lesssim \delta_0 \tau^{-1}; \\
		|\algdo_0^{1,0}\phi| \lesssim \delta_0 (y^0)^{-1/2} \tau^\gamma;\\
		|\algdo_1^{1,1}\phi| \lesssim \delta_0 (y^0)^{-1/2} \tau^{\gamma -1}.
	\end{gather*}
	And hence
	\[
		|\smtf| \lesssim 1.
	\]
\end{prop}
\begin{proof}
	The estimates on $|\algdo_0^{1,0}\phi|$ and $|\algdo_1^{1,1}\phi|$ are trivial using the assumptions, together with the fact that $y^0 \geq \tau$ by definition. The estimate on $|\pwWf_0\phi|$ follows from the bootstrap assumption and the estimate \eqref{eq:y0=tau^2} in Lemma \ref{lem:decayonsupport}.  
	Finally, as $\gamma < 1/2$ by assumption, we see that the three $\algdo_0^{1,0}$, $\algdo_1^{1,1}$, and $\pwWf_0 \phi$ all have global uniform bounds, therefore we must also have global uniform bounds on the arbitrary smooth functions $\smtf$. 
\end{proof}

\begin{prop}[Geometric consequences] \label{prop:geoconBAinf}
	The assumptions \eqref{eq:ass:bainf} implies, when $\delta_0$ is sufficiently small, that
	\begin{enumerate}
		\item The Jacobian determinant $\frac12 \leq |\gmetr| \leq 2$. 
		\item The hyperboloids $\Sigma_\tau$ are space-like relative to $\gmetr$; in fact $\gmetr^{-1}(\D\tau, \D\tau) = -1 + O(\delta_0 \tau^{-5/2})$.
		\item The volume forms $\dvol_{\sigmmetr_\tau}$ and $\dvol_{\siggmetr_\tau}$ are uniformly comparable.
		\item The quantity $\mathfrak{c}_{TT}$ from \eqref{eq:def:ctt} is comparable to $\tau^2 / (y^0)^2$. 
	\end{enumerate}
\end{prop}
\begin{proof}
	The first claim follows from the fact that 
	\[ |\gmetr| = 1 + \gmetrf^{-1}(\D\phi,\D\phi) = 1 + \smtf (\algdo_1^{1,1}\phi)^2.\]
	For the second claim it suffices to prove bounds on $\gmetr^{-1}(\D\tau,\D\tau)$. 
	From \eqref{eq:InvMet} we have that 
	\[ \gmetr^{-1}(\D\tau, \D\tau) = \underbrace{\mink^{-1}(\D\tau,\D\tau)}_{=-1} + 2 \phi \bgf'' (\partial_{\ub{u}}\tau)^2 - \frac{1}{|\gmetr|} (\gmetrf^{-1}(\D\tau, \D\phi))^2.\]
	By definition 
	\[ \partial_{\ub{u}}\tau = \frac{u}{\tau} \]
	and since $\bgf''$ has compact support in $u$ the middle term $\lesssim \delta_0 \tau^{-3}$. For the final term we have schematically
	\begin{align*}
		\gmetrf^{-1}(\D\tau, \D\phi) &= \mink^{-1}(\D\tau, \D\phi) + \phi \bgf'' \partial_{\ub{u}} \tau \partial_{\ub{u}}\phi\\
					     &= \frac{\tau}{y^0} T\phi + \sum_i \frac{y^i}{y^0\tau} L^i\phi + \phi \bgf'' \frac{u}{\tau} \frac{1}{y^0} (L^1 \phi + T\phi)
	\end{align*}
	and so we see $|\gmetrf^{-1}(\D\tau, \D\phi)| \lesssim \delta_0 \tau^{\gamma} (y^0)^{-3/2}$. This implies that the final term decays at least as fast as $(\delta_0)^2\tau^{2\gamma}(y^0)^{-3}$ and hence for sufficiently small $\delta_0$ we have the desired bounds. 

	For the third claim we first examine \eqref{eq:geonormsig}, as the induced volume form on $\Sigma_\tau$ is given by the interior product of the space-time volume form with the unit normal. By the explicit form of $g$ and the pointwise bounds of Proposition \ref{prop:geoconBAinf}, it suffices that $(\D\tau)^\sharp - (\D\tau)^{\gmetr\sharp} / \sqrt{|\gmetr(\D\tau,\D\tau)|}$ is bounded when measured by $\mink$. Due to the above bound on $\gmetr(\D\tau, \D\tau)$, it suffices to control 
	\[ (\D\tau)^{\gmetr\sharp} - (\D\tau)^\sharp = 2\frac{u}{\tau} \phi \bgf'' \partial_{\ub{u}} - \gmetrf^{-1}(\D\phi,\D\tau) \partial^{\gmetr\sharp}\phi. \]
	In terms of the coordinate basis $\partial_{y^\mu}$, the coefficients of the right hand side can be worked out to be bounded by 
	\[ \delta_0\tau^{-2} + (\delta_0)^2 \tau^{2\gamma -1} (y^0)^{-2}.\] 
	This implies the desired conclusion. 

	The fourth and final claim follows immediately from the definition of \eqref{eq:def:ctt}.  
\end{proof}

For conducting the estimates, we will frequently need to swap between the quantities $\Energy_\tau[\psi;\mink]^2$, $\Energy_\tau[\psi;\gmetr]^2$, and 
\[ \int_{\Sigma_\tau} \frac{1}{\tau^2\cosh(\rho)} \sum |L^i\psi|^2 + \frac{1}{\cosh(\rho)} |T\psi|^2 ~\dvol_{\siggmetr_{\tau}}.\]
These three quantities turns out to be comparable if we assume \eqref{eq:ass:bainf} holds with $\delta_0$ sufficiently small. 

\begin{prop}[Energy comparison] \label{prop:enercomp}
	Assuming \eqref{eq:ass:bainf} holds with $\delta_0$, the three energy-type quantities above are compatible. 
\end{prop}
\begin{proof}
	By Proposition \ref{prop:geoconBAinf}, it suffices to compare the term $\seten[\psi;\gmetr](T, (-\D\tau)^{\gmetr\sharp})$ with $\seten[\psi;\mink](T,(-\D\tau)^{\mink\sharp})$. We note first that their difference is given by 
	\[  T\psi [ (\D\tau)^{\gmetr\sharp} - (\D\tau)^{\mink\sharp}]\psi - \frac12 T(\tau) (\gmetr - \mink)^{-1}(\D\psi,\D\psi).\]
	We can expand this to be schematically 
	\begin{multline*}
		T\psi \left[ \phi \pwWf_0 \frac{u}{\tau} \swtf_1 (L^1\psi + T\psi) + \smtf \gmetrf^{-1}(\D\phi,\D\psi) \gmetrf^{-1}(\D\phi, \D\tau) \right] \\
		+ \frac{y^0}{\tau} \left[\phi \pwWf_0 \swtf_2 (L^1 \psi + T\psi)^2 + \smtf (\gmetrf^{-1}(\D\phi,\D\psi))^2 \right].
	\end{multline*}
	Hence we can bound the expression by, using \eqref{eq:ass:bainf} and Proposition \ref{prop:geoconBAinf}, 
	\begin{multline}\label{eq:pf:encom1}
		\lesssim \frac{\delta_0}{\tau^2 y^0} \pwWf_0 | T\psi (L^1\psi + T\psi)| + \frac{\delta_0 \tau^\gamma}{(y^0)^{3/2}}|T\psi \gmetrf^{-1}(\D\phi, \D\psi)| \\
		+ \frac{\delta_0 y^0}{\tau^2} \pwWf_0 \swtf_2 (L^1 \psi + T\psi)^2 + \frac{y^0}{\tau} (\gmetrf^{-1}(\D\phi,\D\psi))^2.
	\end{multline}
	The first term in \eqref{eq:pf:encom1} can be bounded by 
	\[ \lesssim \frac{\delta_0}{\tau^2} \frac{1}{\cosh(\rho)} |T\psi|^2 + \frac{\delta_0}{\tau^2} \frac{1}{\tau^2 \cosh(\rho)} |L^1\psi|^2\]
	and the third term by 
	\[ \lesssim \frac{\delta_0}{\tau} \frac{1}{\tau^2 \cosh(\rho)} |L^1\psi|^2 + \frac{\delta_0}{\tau^3} \frac{1}{\cosh(\rho)} |T\psi|^2.\]
	Both are bounded obviously by a small multiple of $\seten[\psi;\mink](T,(-\D\tau)^{\mink\sharp})$. 
	We can evaluate
	\begin{multline*}
		\gmetrf^{-1}(\D\phi,\D\psi) = -\frac{\tau^2}{(y^0)^2} T\phi T\psi + \algdo_1^{1,0}\phi T\psi + \algdo_1^{1,0}\psi T\phi \\+ \algdo_1^{1,0}\phi \algdo_1^{1,0}\psi + \phi \pwWf_0 \swtf_2 (L^1\phi + T\phi)(L^1\psi + T\psi).
	\end{multline*}
	This implies
	\[ |\gmetrf^{-1}(\D\phi,\D\psi)| \lesssim \frac{\delta_0 \tau^{\gamma}}{(y^0)^{3/2}} |T\psi| + \frac{\delta_0 \tau^\gamma}{(y^0)^{3/2} \tau} |\algdo_0^{1,0}\psi|. \]
	Thus the second term in \eqref{eq:pf:encom1} can be bounded by 
\[ \lesssim \frac{(\delta_0)^2}{\tau^{3-2\gamma}\cosh(\rho)^2} \left( \frac{1}{\cosh(\rho)} |T\psi|^2 + \frac{1}{\tau^2\cosh(\rho)} \sum |L^i \psi|^2\right),\]
	and the fourth term by 
	\[ \lesssim \frac{(\delta_0)^2}{\tau^{3 - 2\gamma}\cosh(\rho)} \left( \frac{1}{\cosh(\rho)} |T\psi|^2 + \frac{1}{\tau^2 \cosh(\rho)} \sum |L^i \psi|^2 \right).\]
	Both terms are similarly controlled by a small multiple of $\seten[\psi;\mink](T,(-\D\tau)^{\mink\sharp})$. This implies our proposition. 
\end{proof}

In Section \ref{sect:inhom} below where we treat the inhomogeneous terms, we frequently need to estimate weighted $L^2$ integrals along $\Sigma_\tau$. We can compare such integrals to the energies by the following Corollary, which follows after noting $y^0 = \tau\cosh(\rho)$.
\begin{cor}\label{cor:basic:l2}
	We have the following bounds for $L^2$ integrals of derivatives of $\phi$:
	\begin{gather*}
		\|(y^0\tau)^{-1/2}\algdo_0^{k,0}\phi\|_{L^2(\Sigma_\tau)} \lesssim \mathfrak{E}_{k-1}(\tau), \\
		\|(y^0\tau)^{-1/2}\algdo_0^{k,0} L^1\phi\|_{L^2(\Sigma_\tau)} \lesssim \mathfrak{F}_{k}(\tau), \\
		\|(y_0)^{-1/2}\tau^{1/2} \algdo_1^{k+1,1} \phi \|_{L^2(\Sigma_\tau)} \lesssim \mathfrak{E}_{k}(\tau), \\
		\|(y^0)^{-1/2}\tau^{1/2} \algdo_1^{k+1,1} L^1\phi \|_{L^2(\Sigma_\tau)} \lesssim \mathfrak{F}_{k+1}(\tau). 
	\end{gather*}
\end{cor}

\subsection{Improved \texorpdfstring{$L^\infty$}{L-infty} bounds from \texorpdfstring{(\ref*{eq:ass:ba2})}{(BA2)}}

As a consequence of the energy comparison Proposition \ref{prop:enercomp}, we can apply Proposition \ref{prop:peeling} with $d = 3$ to \eqref{eq:ass:ba2} and derive the following $L^\infty$ estimates of $\phi$ and its derivatives. 
\begin{gather}\label{eq:linfty:est}
	\left\{\quad
	\begin{gathered}
		|\phi| + |L^1\phi| \lesssim \frac{\delta}{(y^0)^{1/2}}, \\
		|\algdo_0^{1,0} \phi| + |\algdo_0^{1,0} L^1\phi| \lesssim \frac{\delta \tau^\gamma}{(y^0)^{1/2}}, \\
		|T\phi| + |T L^1\phi| \lesssim \frac{\delta \tau^{\gamma}}{(y^0)^{1/2} \tau}, \\
		|\algdo_0^{2,0} \phi| \lesssim \frac{\delta \tau^{1+\gamma}}{(y^0)^{1/2}}, \\
		|\algdo_1^{2,1} \phi| \lesssim \frac{\delta \tau^{\gamma}}{(y^0)^{1/2}}, \\
		|\algdo_1^{3,1} \phi| \lesssim \frac{\delta \tau^{1+\gamma}}{(y^0)^{1/2}}.
	\end{gathered}
	\right.
\end{gather}

With the aid of \eqref{eq:TTcomp}, we can also estimate 
\begin{gather}\label{eq:linfty:est2}
	\left\{\quad
	\begin{gathered}
		|TT\phi | \lesssim \delta \tau^{\gamma-1}, \\
		|\algdo_0^{1,0} TT\phi| \lesssim \delta \tau^{\gamma}.
	\end{gathered}
	\right.
\end{gather}
Here we also made use of Lemma \ref{lem:decayonsupport} freely. 

\begin{rmk}
	Note that we have estimated \eqref{eq:linfty:est2} by directly estimating the right hand side of \eqref{eq:TTcomp} using \eqref{eq:linfty:est}. In particular these were \emph{not} derived from applying Proposition \ref{prop:peeling} to appropriate energy integrals: in fact we have not yet proven any $L^2$ estimates for $TT\phi$ and its higher derivatives. It turns out the necessary $L^2$ estimates require a little bit of work, and we defer their proofs to Lemma \ref{lem:TT:energest}.
\end{rmk}

\begin{rmk}
	Notice that \eqref{eq:linfty:est} and \eqref{eq:linfty:est} controls up to two derivatives of $\phi$ in all directions, and in particular controls the first derivative of the dynamical metric $\gmetr$. 
	Thus we can apply the blow-up criterion for quasilinear wave equations and assert that the \textit{a priori} estimates guaranteed by our bootstrap argument suffices to prove global existence of the solution. 
\end{rmk}

\begin{rmk}
	In the bootstrap argument we will be studying energies to the top order $\mathfrak{E}_4$ and $\mathfrak{F}_3$, which corresponds to 3 additional derivatives applied to the equations \eqref{eq:geomeqLiphi} and \eqref{eq:geomeqL1phi} respectively. Examining the terms that show up in the nonlinearities, which depend only on up-to-two derivatives of $\phi$, this means that when performing energy estimates the highest derivative that we will put into $L^\infty$ would be three; and as we will only be commuting with $\algdo_0^{1,0}$ derivatives, there will be no $TTT\phi$ terms to worry about. Hence between \eqref{eq:linfty:est} and \eqref{eq:linfty:est2} all possible $L^\infty$ terms are captured. 
\end{rmk}

\subsection{Controlling the deformation tensor term}\label{sect:deftenTerm}

Now let us return to studying the first integrand in \eqref{eq:engineq:restate} as promised. 
First, using Proposition \ref{prop:geoconBAinf}, the space-time integral with regards to $\dvol_\gmetr$ can be replaced by the integral with regards to $\dvol_\mink$ to which we can apply the co-area formula and decompose as $\dvol_{\sigmmetr_\tau}~\D\tau$.
The same proposition also implies we can replace the hypersurface volume element and have the integral conducted with respect to $\dvol_{\siggmetr_\tau}~\D\tau$. 

For the integration along $\Sigma_\tau$, we will put $\psi$, which is automatically \emph{top order}, in the appropriate weighted $L^2$ space; by Proposition \ref{prop:enercomp} these $L^2$ integrals can be bounded by the quasilinear energies. 
We therefore obtain the following bound
\begin{multline}
	\iint_{\tau\in [\tau_0,\tau_1]} |\seten[\psi;\gmetr] :_\gmetr \lieD_T\gmetr | ~\dvol_g \leq \int_{\tau_0}^{\tau_1} \Bigl\| \frac{1}{\cosh(\rho)} \pwWf_0 (\phi + T\phi) \Bigr\|_{L^\infty(\Sigma_\tau)} \Energy_\tau[\psi;\gmetr]^2 \\
	 + \Bigl\| \cosh(\rho) \smtf \bigl[ (\algdo_1^{1,1}\phi)(\algdo_2^{2,2}\phi) + (\algdo_1^{1,1}\phi)^2(1 + \pwWf_0 T\phi) \bigr]\Bigr\|_{L^\infty(\Sigma_\tau)} \Energy_\tau[\psi;\gmetr]^2 ~\D\tau.
\end{multline}
The terms in $L^\infty$ can be estimated with the help of \eqref{eq:linfty:est} and \eqref{eq:linfty:est2}. 
First we have
\[ \bigl|\frac{1}{\cosh(\rho)} \pwWf_0 (\phi + T\phi) \bigr| \lesssim \frac{1}{\tau} \bigl( \frac{\delta}{\tau} + \frac{\delta \tau^\gamma}{\tau^2}\bigr) \leq \delta \tau^{-2}; \]
we used here that  $\frac{1}{\cosh(\rho)} \approx \tau^{-1}$ by Lemma \ref{lem:decayonsupport}. 
Next we have
\[ |\cosh(\rho) \smtf \algdo_1^{1,1}\phi \algdo_2^{2,2}\phi| \lesssim \frac{y^0}{\tau} \frac{\delta \tau^\gamma}{\sqrt{y^0}\tau} \frac{\delta \tau^{\gamma}}{\tau} \leq \delta^2 \tau^{2\gamma - 2} \]
after observing Lemma \ref{lem:decayonsupport} again. 
Finally the last term
\[ |\cosh(\rho) \smtf (\algdo_1^{1,1}\phi)^2 (1 + \pwWf_0 T\phi) | \lesssim \frac{y^0}{\tau} \frac{\delta^2 \tau^{2\gamma}}{y^0 \tau^2} \bigl( 1 +  \frac{\delta \tau^\gamma}{\sqrt{y^0}\tau}\bigr) \lesssim \delta^2 \tau^{2\gamma - 3}. \]
Hence, with our assumption that $\gamma < 1/4$ we have that 
\begin{equation}\label{eq:deftenTermEst}
	\iint_{\tau\in [\tau_0,\tau_1]} |\seten[\psi;\gmetr] :_\gmetr \lieD_T\gmetr | ~\dvol_g \lesssim \int_{\tau_0}^{\tau_1} \delta \tau^{-3/2} \Energy_\tau[\psi;\gmetr]^2 ~\D\tau. 
\end{equation}
Note the integrable power in $\tau$: the deformation tensor term does not cause any difficulty in the analysis.

\section{Controlling the inhomogeneity}
\label{sect:inhom}

In this section we focus our attention on estimating the second term in the energy estimate \eqref{eq:engineq:restate}, given by the integral
\[ 
	\iint_{\tau\in [\tau_0, \tau_1]} \left| \Box_\gmetr \psi\cdot T\psi\right| \dvol_\gmetr.
\]
By virtue of the geometric comparison Proposition \ref{prop:geoconBAinf} and the energy comparison Proposition \ref{prop:enercomp}, we can bound this by
\[ 
	\int_{\tau_0}^{\tau_1} \Bigl\| \sqrt{\frac{y^0}{\tau}} \Box_\gmetr \psi\Bigr\|_{L^2(\Sigma_\tau)} \Energy[\psi;\gmetr]~\D\tau.
\]
We will take $\psi$ here one of $\{ \phi, L^1\phi, L^\alpha L^1\phi, L^i\phi, L^\alpha L^i\phi\}$, where $\alpha$ is some multi-index with length no more than 3, and $i \in \{2,3\}$. 

To streamline our control for the higher derivative terms, we observe the following principle: 
\begin{gather}\label{eq:schemeprincip}
	\Bigl\| \sqrt{\frac{y^0}{\tau}} (\mathrm{expr}) \Bigr\|_{L^2(\Sigma_\tau)} \lesssim \tau^{\nu} 
	\implies \Bigl\| \sqrt{\frac{y^0}{\tau}} \algdo_0^{1,0} (\mathrm{expr}) \Bigr\|_{L^2(\Sigma_\tau)} \lesssim \tau^{\nu+1}. \tag{SP}
\end{gather}
Here, $(\mathrm{expr})$ means some polynomial expressions in $\smtf, \pwWf_*, \swtf_*$, and $\algdo_*^{*,1}\phi$. 
We emphasize that \eqref{eq:schemeprincip} is a principle meta to our proof, where we will bound each term in the polynomial expression either in some weighted $L^2$ space on $\Sigma_\tau$ or in $L^\infty$, using the bootstrap assumptions \eqref{eq:ass:ba2} and their consequences \eqref{eq:linfty:est} and \eqref{eq:linfty:est2}. 
The symbol ``$\lesssim$'' in \eqref{eq:schemeprincip} should be understood to mean ``can be proven as the result of our bootstrap argument to be bounded by'', and not a factual assertion of a possibly better bound. 

Understood this way, \eqref{eq:schemeprincip} follows simply from the facts that:
\begin{itemize}
	\item For $\algdo_*^{*,1}\phi$ terms, in \eqref{eq:ass:ba2}, each higher derivative brings at most an additional loss of $\tau$.
	\item The terms $\swtf_*$ are invariant under action by $L$-derivatives. 
	\item As discussed after Notation \ref{notn:PfWf}, $\algdo_0^{1,0}\pwWf_m = \pwWf_{m+1}$, which allows it to grow with an additional factor of $\ub{u}^{1/2}$. By Lemma \ref{lem:decayonsupport} this can be bounded by $\tau$. 
	\item Finally, observe that 
		\[ \algdo_0^{1,0}\smtf = \smtf \cdot \bigl[ \algdo_0^{1,0}(\phi \pwWf_0) + \algdo_0^{2,0}\phi + \algdo_1^{2,1}\phi\bigr]
		\]
		by the chain rule. The first and third terms are in fact decaying by \eqref{eq:linfty:est}, and the middle term is bounded by $\delta \tau^{1/2 + \gamma}$, which, since $\gamma < 1/4$, is less than a full order of $\tau$ increase in growth. 
\end{itemize}

Occasionally $\algdo_*^{*,2}\phi$ terms also occur: these are the terms with two $T$ derivatives. Their $L^\infty$ estimates are already captured in \eqref{eq:linfty:est2} and they can be seen to also obey the schematic principle \eqref{eq:schemeprincip} where higher derivatives lose factors of $\tau$. 

We complement the estimates with the following $L^2$ version:
\begin{lem}\label{lem:TT:energest}
	For $0 \leq k \leq 3$, we have 
	\[ \| \tau^{5/2} (y^0)^{-3/2} \algdo_0^{k,0} TT\phi \|_{L^2(\Sigma_\tau)} \lesssim \delta \tau^{\max(k-1,0) + \gamma}.\]
\end{lem}
\begin{proof}[Sketch of proof]
	The proof of this estimate itself is an application of the principle \eqref{eq:schemeprincip} and the bootstrap assumptions. Observe first that by \eqref{eq:TTcomp} that $TT\phi$ can be expanded as $1/\mathfrak{c}_{TT}$ times a polynomial expression in $\smtf, \pwWf_*, \swtf_*$, and $\algdo_*^{*,1}\phi$ to which \eqref{eq:schemeprincip} can apply. For convenience call this polynomial expression $\mathcal{O}$. 
	From Corollary \ref{cor:basic:l2} combined with \eqref{eq:linfty:est} we have that
	\begin{align*}
		\| (\tau y^0)^{1/2} \algdo_2^{2,1}\phi \|_{L^2} &\lesssim \mathfrak{E}_1 \\
		\| \tau^{1/2} (y^0)^{-1/2} \algdo_1^{1,1}\phi \|_{L^2} &\lesssim \mathfrak{E}_0 \\
		| (y^0) \pwWf_0 \phi \swtf_1 | & \lesssim \delta \tau^{-1} \\
		| (y^0) \pwWf_0 (\algdo_1^{1,1}\phi)^3 | & \lesssim \delta^3 \tau^{3\gamma - 4}\\
		| (y^0) (\algdo_1^{1,1}\phi)^2 | & \lesssim \delta^2 \tau^{2\gamma - 2} \\
		| (y^0) (1+\pwWf_1 \phi) \swtf_1 (\algdo_1^{1,1}\phi)^2 | & \lesssim \delta^3 \tau^{2\gamma - 4}.
	\end{align*}
	Additionally, we pay attention to the quadratic term
	\[ \| (\tau y^0)^{1/2} \pwWf_0 \swtf_2 (L^1 \phi + T\phi)^2 \|_{L^2} \lesssim \| \pwWf_0 (L^1\phi + T\phi) \|_{L^\infty} \mathfrak{E}_0 \lesssim \delta^2 \tau^{-1}.\]  
	Together with the estimate $\mathfrak{c}_{TT} \approx \tau^2 / (y^0)^2$ implies the Lemma when $k = 0$. Specifically, we have that
	\[ \| (\tau y^0)^{1/2} \mathcal{O} \|_{L^2} \lesssim \mathfrak{E}_1 + \mathfrak{E}_0 \delta \tau^{-1}.\]
	Similar arguments show that 
	\[ \| (\tau y^0)^{1/2} \algdo_0^{1,0}\mathcal{O}\|_{L^2} \lesssim \mathfrak{E}_2 + \delta \mathfrak{E}_1.\]
	For this we crucially need Remark \ref{rmk:nogrowth:TT} which shows that there is no growth arising from first derivatives of $\smtf$ terms in \eqref{eq:TTcomp}.
	(Note that this step requires explicit argument and \emph{not} an appeal to the principle \eqref{eq:schemeprincip}.)
	For higher derivatives we can appeal to \eqref{eq:schemeprincip}. 

	For higher $k$, one also needs to estimate derivatives of $\mathfrak{c}_{TT}$. We observe the following schematic computation
	\begin{multline*}
		\algdo_0^{k,0} \mathfrak{c}_{TT} = \tau^2 \swtf_2 \bigl[1 + \algdo_{0}^{\leq k,0} \bigl( (\algdo_1^{1,1}\phi)^2+ \frac{1}{\tau^2} (\algdo_0^{1,0}\phi)^2 \\ 
		+ \frac{1}{\tau^2} (\algdo_0^{0,0}\phi) \pwWf_0 (1 + (\algdo_1^{1,1}\phi)(L^1\phi + T\phi))\bigr)\bigr] 
	\end{multline*}
	The inner term, operated on by $\algdo_0^{\leq k, 0}$ can be bounded by 
	\[ \delta^2 \tau^{2\gamma - 2} (y^0)^{-1} + \delta \tau^{-3} (1 + \delta^2 \tau^{\gamma-4}) \]
	through \eqref{eq:linfty:est}. By the schematic principle we have that for $k \leq 2$, $\algdo_0^{k,0} \mathfrak{c}_{TT}$ is bounded by $\tau^2 \swtf_2$. And this shows the Lemma up to $k \leq 2$. 

	For $k = 3$, we need to consider the case where all derivatives hits on $\mathfrak{c}_{TT}$, since all other terms follow from the principle \eqref{eq:schemeprincip}. In this case we need to essentially estimate something that is schematically the same as 
	\begin{multline*}
		\Bigl\| \tau^{5/2} (y^0)^{-3/2} TT\phi \cdot \bigl[ (\algdo_1^{1,1}\phi)(\algdo_1^{4,1}\phi) + \\
		\tau^{-2} (\algdo_0^{1,0}\phi) (\algdo_0^{4,0}\phi) + \tau^{-2} \phi \pwWf_0 ( L^1 \phi + \algdo_1^{1,1}\phi) \algdo_1^{4,1}\phi \bigr] \Bigr\|_{L^2}.
	\end{multline*}
	Here we group $\tau^{1/2} (y^0)^{-1/2}$ with the $\algdo_1^{4,1}\phi$ terms, and $\tau^{-1/2} (y^0)^{-1/2}$ with $\algdo_0^{4,0}\phi$, to bound in $L^2$ by $\mathfrak{E}_3$. The remaining parts to be controlled in $L^\infty$ boils down to
	\[  \tau^2 (y^0)^{-1} TT\phi [ \algdo_1^{1,1}\phi + \tau^{-1} (\algdo_0^{1,0}\phi) + \tau^{-2} \phi \pwWf_0 (L^1\phi + \algdo_1^{1,1}\phi)]\]
	which can be bounded by 
	\[ \delta^2 \tau^{1 + \gamma} (y^0)^{-1} [ \tau^{\gamma - 1}(y^0)^{-1/2} + \delta \tau^{-5} ]\lesssim \delta^2 \tau^{2\gamma - 3/2} \] 
	and so we see contributes to a lower order term, and the Lemma holds also for $k = 3$. 
\end{proof}

This previous Lemma implies we can also extend \eqref{eq:schemeprincip} to handle also $\algdo_*^{*,2}\phi$ terms in the expression. 

\begin{rmk}
	Note that \eqref{eq:schemeprincip} gives the \emph{worst case scenario} bound on the higher derivatives of an expression. As one already sees in the proof of the Lemma above, sometimes this worst case bound is not realized. For example, first derivatives of $\smtf$ do no lose a whole factor of $\tau$ even in the worst case, and as seen in the proof of the Lemma above, sometimes derivatives of $\smtf$ do not lose decay at all. Similarly, going from $\phi$ to $\algdo_0^{1,0}\phi$ in $L^\infty$ only entails a $\tau^{\gamma}$ loss. 

	However, overall, the schematic principle \eqref{eq:schemeprincip} cannot be generally improved. This is due to the possible presence of the $\pwWf_*$ terms. Each time a $\algdo_0^{1,0}$ derivative hits $\pwWf_*$ we necessarily incur a penalty of one factor of $\tau$. This entirely agrees with our semilinear analysis in Section \ref{sect:semilinear} where the highest growth rates always accompanies the terms when $\ell_0$ is largest (where most derivatives hit on $\bgf''$).  
\end{rmk}

\subsection{Higher order nonlinear terms}

\begin{prop}\label{prop:HO:bnd}
	The following bounds hold:
	\begin{gather} \| (y^0)^{1/2} \tau^{-1/2} \mathrm{HO}_1 \|_{L^2(\Sigma_\tau)} \lesssim \delta^3 \tau^{3\gamma -4}\\
	\| (y^0)^{1/2} \tau^{-1/2} \mathrm{HO}_i\|_{L^2(\Sigma_\tau)} \lesssim \delta^3 \tau^{2\gamma - 3}.\end{gather}
\end{prop}
\begin{proof}
	We focus first on $\smtf (\algdo_1^{1,1}\phi)(\algdo_1^{1,1}L^1\phi)(\algdo_1^{2,2}\phi)$, the sole cubic term in $\mathrm{HO}_1$. Outside the $TT\phi$ term this can be bounded by
	\begin{multline*}
		\| (y^0)^{-1/2} \tau^{-1/2} \smtf(\algdo_1^{1,1}\phi) (\algdo_1^{1,1}L^1\phi)(\algdo_1^{2,1}\phi) \|_{L^2(\Sigma_\tau)} \\
		\lesssim \| \tau^{-1} (\algdo_1^{1,1}\phi)(\algdo_1^{1,1}L^1\phi) \|_{L^\infty} \mathfrak{E}_1(\tau)  \lesssim \delta^3 \tau^{2\gamma - 4}.
	\end{multline*}
	For the $TT\phi$ term we need to add a factor of $\frac{(y^0)}{\tau^2}$: this is because by Proposition \ref{prop:geoconBAinf} we have
	\[
		|\mathfrak{c}_{TT}^{-1} \algdo_2^{2,1} \phi| \lesssim \frac{(y^0)^2}{\tau^2} \frac{1}{y^0} \algdo_1^{2,1}\phi.
	\]
	This gives the bound by
	\[
		\| \tau^{-3} (y^0) (\algdo_1^{1,1}\phi)(\algdo_1^{1,1}L^1\phi) \|_{L^\infty} \delta \lesssim \delta^3 \tau^{2\gamma - 5}
	\]
	where we again used Lemma \ref{lem:decayonsupport}. (The differing decay rates of the two terms stems from the fact that $\algdo_0^{1,0}\phi$ has additional decay along $\Sigma_\tau$ compared to $T\phi$, but this decay is not seen when taking $L^\infty$ on $\Sigma_\tau$.)

	The quartic and higher order terms can be treated similarly, the details of which we omit here, the general idea being to put the highest order derivative terms in $L^2$ and lower ones in $L^\infty$. This shows that the quartic and higher order terms in $\mathrm{HO}_1$ can be bounded by $\lesssim \delta^4 \tau^{3\gamma - 4}$ uniformly. (We remark here that as all the remaining terms are multiplied by a $\pwWf_*$ weight, for their estimates we can consider $(y^0) \approx \tau^2$. This means that the anisotropy between $\algdo_{2}^{2,1}\phi$ terms and $TT\phi$ terms that showed up in the cubic term estimates can be avoided.) 

	For $\mathrm{HO}_i$, the additional cubic term now is a generic $\algdo_1^{2,1}\phi$ instead of $\algdo_1^{1,1}L^1\phi$, which means it decays slower by a factor of $\tau$. The additional quartic terms can all be bounded by $\lesssim \delta^4 \tau^{2\gamma - 3}$, and our claims follow. 
\end{proof}

\begin{rmk}
	To estimate $\algdo_*^{*,1}\phi$ terms using either the energy (and then by the bootstrap \eqref{eq:ass:ba2}) or using a straight-up $L^\infty$ estimate using the peeling estimates in Proposition \ref{prop:peeling}, we would need any factors of $T$ derivative to be the outermost one. Luckily, commutation reduces the order of derivatives and leaves the weight unchanged (see Proposition \ref{prop:diffalg}), which has the advantage of guaranteeing that the commutator terms have faster decay (by $\tau^{-1}$). 
\end{rmk}

\begin{rmk}
	The quartic term bounds for $\mathrm{HO}_1$ can be improved from $\delta^4 \tau^{3\gamma - 4}$ to $\delta^4 \tau^{4\gamma-5}$, thereby upgrading the overall bound on $\mathrm{HO}_1$ to $\delta^3 \tau^{2\gamma - 4}$.
	This improvement comes from noting that the term $\smtf \pwWf_0 (\algdo_0^{1,0}\phi)(\algdo_1^{1,1}\phi)^2 (\algdo_2^{2,2}\phi)$ in the definition of $\mathrm{HO}_1$ is actually $\smtf \pwWf_0 (L^1 \phi) (\algdo_1^{1,1}\phi)^2 (\algdo_2^{2,2}\phi)$. As for our purposes these types of improvements are not essential, and does not effect the closing of the bootstrap, we shall not pursue this and myriad other improvements in the higher order terms. 

	One should however note that for studying the missing case $d = 2$, the above indicates that careful treatment of all quadratic, cubic, \emph{and} quartic nonlinearities will be likely necessary. 
\end{rmk}

\subsection{Quadratic terms}

Now we consider the quadratic nonlinearities. These terms are a bit more delicate and we will include more details of the arguments.  

\subsubsection{Zeroth order case}\label{ssect:qn0}

Looking at \eqref{eq:geomeqphiS}, we need to estimate 
\[ \| (y^0)^{-3/2} \tau^{-1/2} \pwWf_0 (L^1\phi + T\phi)^2\|_{L^2(\Sigma_\tau)}. \]
We observe that
\[ \| (y^0)^{-3/2} \tau^{-1/2} \pwWf_0 (L^1\phi)^2\|_{L^2(\Sigma_\tau)} \lesssim \| (y^0)^{-1} \pwWf_0 L^1\phi \|_{L^\infty} \mathfrak{E}_0(\tau) \lesssim \frac{\delta^2}{\tau^{3}}. \]
Here we used that by \eqref{eq:linfty:est}, our bootstrap assumptions imply $|L^1\phi| \lesssim \delta (y^0)^{-1/2}$. Additionally recall that $y^0 \approx \tau^2$ in the presence of $\pwWf_*$. Next, we have
\[ \| (y^0)^{-3/2} \tau^{-1/2} (T\phi)^2\|_{L^2(\Sigma_\tau)} \lesssim \| (y^0)^{-1} \tau^{-1} T\phi\|_{L^\infty} \mathfrak{E}_0(\tau) \lesssim \frac{\delta^2\tau^{\gamma}}{\tau^{5}}. \]
We note here for this term the $\swtf_2$ term in the nonlinearity is crucial: without it the denominator would only have $\tau^{-1}$ which would not have enabled us to close our estimates.

\subsubsection{First order, $\psi = L^1\phi$}\label{ssect:qn1:1}

Let us now consider $\mathrm{QN}_1$ (see \eqref{eq:def:qn1}). The terms with $(L^1\phi + T\phi)^2$ and $\phi (L^1\phi + T\phi)$ are controlled exactly as the zeroth order term case, by $\delta^2 \tau^{-3}$. For the remaining terms, we see first that
\[
	\| (y^0)^{-3/2} \tau^{-1/2} \pwWf_0 T\phi (L^1 L^1\phi + T L^1\phi) \|_{L^2(\Sigma_\tau)}
	\lesssim \| (y^0)^{-1} T\phi \|_{L^\infty} \mathfrak{F}_0(\tau) \lesssim \frac{\delta^2 \tau^{\gamma}}{\tau^{4}}.
\]
Similarly 
\[ \| (y^0)^{-3/2} \tau^{-1/2} \pwWf_0 (\phi + L^1\phi) (L^1 L^1\phi + T L^1\phi) \|_{L^2(\Sigma_\tau)} \lesssim \frac{\delta^2}{\tau^3}.\]
The final term involves $TT\phi$, for which we can bound
\[ \| (y^0)^{-3/2} \tau^{-1/2} \pwWf_0 (\phi + L^1\phi) TT\phi \|_{L^2(\Sigma_\tau)} \lesssim \| \tau^{-3}(\phi + L^1\phi)\|_{L^\infty} \mathfrak{E}_1(\tau) \lesssim \frac{\delta^2}{\tau^4}.\]

\subsubsection{First order, $\psi = L^i \phi$}\label{ssect:qn1:i}

We next consider $\mathrm{QN}_i$ (see \eqref{eq:def:qni}). There is a loss compared to the $\mathrm{QN}_1$ terms, which we expect. 
First, 
\begin{multline*}
	\| (y^0)^{-1/2}\tau^{-1/2} \pwWf_0 (\phi + L^1\phi + T\phi)(\algdo_1^{1,1}L^1\phi + \algdo_2^{2,2}\phi + \algdo_1^{1,1}\phi) \|_{L^2(\Sigma_\tau)}\\
	\lesssim \| \tau^{-1} \pwWf_0 (\phi + L^1\phi + T\phi)\|_{L^\infty} \cdot [ \mathfrak{F}_0(\tau) + \mathfrak{E}_1(\tau)] \lesssim \frac{\delta^2}{\tau^2}.
\end{multline*}
Next
\begin{multline*}
	\| (y^0)^{-3/2} \tau^{-1/2} (\pwWf_0 L^i\phi + \pwWf_1 \phi)(L^1 L^1 \phi + T L^1\phi + TT\phi + L^1\phi + T\phi)\|_{L^2(\Sigma_\tau)} \\
	\lesssim \|(y^0)^{-1} (\pwWf_0 \algdo_0^{1,0}\phi + \pwWf_1 \phi)\|_{L^\infty} \cdot [ \mathfrak{F}_0(\tau) + \mathfrak{E}_1(\tau) ] \lesssim \frac{\delta^2}{\tau^2}. 
\end{multline*}
Finally, 
\begin{multline*}
	\| (y^0)^{-3/2} \tau^{-1/2} \pwWf_1 (\phi + L^1\phi + T\phi)(L^1\phi + T\phi)\|_{L^2(\Sigma_\tau)} \\
	\lesssim \| (y^0)^{-1} \pwWf_1 (\phi + L^1\phi + T\phi)\|_{L^\infty} \cdot \mathfrak{E}_0(\tau) \lesssim \frac{\delta^2}{\tau^2}
\end{multline*}

\subsubsection{Higher order cases}\label{ssect:qno}

By Remark \ref{rmk:commterm:simpl}, the higher order derivatives $L^\alpha L^i\phi$ where $i = 2,3$ can be treated using \eqref{eq:schemeprincip}. 
It suffices to consider the higher derivatives of $L^1\phi$. 
Observe that the principle \eqref{eq:schemeprincip} can also be applied to the commutator terms: that control of $\algdo_0^{1,0}[\algdo_0^{1,0},\Box_\gmetr]\psi$ also gives control of $[\algdo_0^{1,0}, \Box_\gmetr]\algdo_0^{1,0}\psi$, since the terms of the latter is schematically a subset of those terms that appears in the former. Hence it suffices to consider the estimates for $[\algdo_0^{1,0}, \Box_\gmetr]L^1\phi$. 

We treat each of the six quadratic terms in $[\algdo_0^{1,0}, \Box_\gmetr]L^1\phi$ listed in the schematic decomposition \eqref{eq:commtermlist} below. First, we can estimate
\[
	\| (y^0)^{-1/2}\tau^{-1/2} \pwWf_0 (\algdo_1^{1,1}\phi)(L^1\psi + T\psi)\|_{L^2(\Sigma_\tau)}
	\lesssim \| \pwWf_0 (\algdo_1^{1,1}\phi)\|_{L^\infty} \cdot \mathfrak{F}_0(\tau) \lesssim \frac{\delta^2 \tau^\gamma}{\tau^2}.
\]
Next, we have
\[ 
	\| (y^0)^{-1/2} \tau^{-1/2} \pwWf_0 (\phi + L^1\phi + T\phi) (\algdo_1^{1,1}\psi) \|_{L^2(\Sigma_\tau)} \lesssim
	\| \tau^{-1} \pwWf_0 (\phi + L^1\phi + T\phi) \|_{L^\infty} \cdot \mathfrak{F}_0(\tau) \lesssim \frac{\delta^2}{\tau^2}.
\]
The third term we estimate by 
\begin{multline*}
	\| (y^0)^{-1/2}\tau^{-1/2} \pwWf_0 (\algdo_1^{1,1}\phi) (\algdo_1^{2,1} \psi + TT\psi) \|_{L^2(\Sigma_\tau)} \\
	\lesssim \|\tau^{-1} \pwWf_0 (\algdo_1^{1,1}\phi) \|_{L^\infty} \cdot [ \mathfrak{F}_1(\tau) + \mathfrak{E}_2(\tau)] \lesssim \frac{\delta^2 \tau^{2\gamma}}{\tau^3}.
\end{multline*}
Next we have
\begin{multline*}
	\| (y^0)^{-3/2} \tau^{-1/2} \pwWf_0 (\algdo_0^{1,0} L^1 \phi + \algdo_0^{1,0}T\phi)(L^1\psi + T\psi) \|_{L^2(\Sigma_\tau)} \\
	\lesssim \| (y^0)^{-1} \pwWf_0 (\algdo_0^{1,0} L^1\phi + \algdo_0^{1,0}T\phi)\|_{L^\infty} \cdot \mathfrak{F}_1(\tau) \lesssim \frac{\delta^2 \tau^\gamma}{\tau^3}.
\end{multline*}
The fifth term we estimate by
\begin{multline*}
	\| (y^0)^{-3/2} \tau^{-1/2} \pwWf_1 (\phi + L^1\phi + T\phi)(L^1\psi + T\psi)\|_{L^2(\Sigma_\tau)} \\
	\lesssim \| (y^0)^{-1} \pwWf_1 (\phi + L^1\phi + T\phi)\|_{L^\infty} \cdot \mathfrak{F}_1(\tau) \lesssim \frac{\delta^2}{\tau^2}.
\end{multline*}
And the final term is estimated by
\begin{multline*}
	\| (y^0)^{-3/2} \tau^{-1/2} \pwWf_1 \phi (L^1 L^1\psi + T L^1\psi + TT\psi)\|_{L^2(\Sigma_\tau)} \\
	\lesssim \| (y^0)^{-1} \pwWf_1 \phi\|_{L^\infty} \cdot [ \mathfrak{F}_1(\tau) + \mathfrak{E}_2(\tau)] \lesssim \frac{\delta^2 \tau^\gamma}{\tau^2}.
\end{multline*}

\section{Closing the bootstrap}
\label{sect:closing}

We conclude our proof of Theorem \ref{thm:mainthmQ} by putting together the estimates in the previous sections using \eqref{eq:engineq:restate}. 
By our control of the deformation tensor \eqref{eq:deftenTermEst}, we have that 
\[ \Energy_{\tau_1}[\psi;\gmetr]^2 - \Energy_{\tau_0}[\psi;\gmetr] \lesssim \int_{\tau_0}^{\tau_1} \frac{\delta}{\tau^{3/2}} \Energy_\tau[\psi;\gmetr]^2 + \| (y^0)^{1/2} \tau^{-1/2} \Box_\gmetr \psi\|_{L^2(\Sigma_\tau)} \Energy_\tau[\psi;\gmetr] ~\D\tau.\] 
Now let $\sigma \in (2,T)$. 

From our bootstrap assumptions \eqref{eq:ass:ba2} and the computations of Section \ref{ssect:qn0}, we get 
\begin{equation}
	\mathfrak{E}_0(\sigma)^2 - \mathfrak{E}_0(2)^2 \lesssim \int_{2}^{\sigma} \frac{\delta^3}{\tau^{3/2}} + \frac{\delta^3}{\tau^3} ~\D\tau \lesssim \delta^3.
\end{equation}
From Section \ref{ssect:qn1:1}, we get
\begin{equation}
	\mathfrak{F}_0(\sigma)^2 - \mathfrak{F}_0(2)^2 \lesssim \int_2^{\sigma} \frac{\delta^3}{\tau^{3/2}} + \frac{\delta^3}{\tau^3} ~\D\tau \lesssim \delta^3.
\end{equation}
From Section \ref{ssect:qn1:i}, we get
\begin{equation}
	\mathfrak{E}_1(\sigma)^2 - \mathfrak{E}_1(2)^2 \lesssim \int_2^\sigma \frac{\delta^3}{\tau^{3/2}} + \frac{\delta^3}{\tau^2} ~\D\tau \lesssim \delta^3.
\end{equation}
By applying the principle \eqref{eq:schemeprincip} and factoring in Remark \ref{rmk:commterm:simpl} this implies for $k \geq 2$, 
\begin{multline}
	\mathfrak{E}_k(\sigma)^2 - \mathfrak{E}_k(2)^2 \lesssim \int_2^\sigma \frac{\delta^3 \tau^{2\gamma + 2(k-2)}}{\tau^{3/2}} + \frac{\delta^3\tau^{\gamma + k-2}}{\tau^2} \cdot \tau^{k-1}~ \D\tau \\
	\lesssim \delta^3 \sigma^{2\gamma + 2(k-2) - 1/2} + \delta^3 \sigma^{\gamma + 2(k-2)} \leq \delta^3 \sigma^{2\gamma + 2(k-2)}.
\end{multline}

Finally, from Section \ref{ssect:qno} and principle \eqref{eq:schemeprincip} we get
\begin{equation}
	\mathfrak{F}_1(\sigma)^2 - \mathfrak{F}_1(2)^2 \lesssim \int_2^\sigma \frac{\delta^3}{\tau^{3/2}} + \frac{\delta^3}{\tau^2} + \frac{\delta^3 \tau^\gamma}{\tau^2} ~\D\tau \lesssim \delta^3.
\end{equation}
Further applications of the principle \eqref{eq:schemeprincip} gives us the higher order estimates for $k \geq 2$
\begin{multline}
	\mathfrak{F}_k(\sigma)^2 - \mathfrak{F}_k(2)^2 \lesssim \int_2^{\sigma} \frac{\delta^3 \tau^{2\gamma + 2(k-2)}}{\tau^{3/2}} + \frac{\delta^3\tau^{\gamma + k-2}}{\tau^2} \cdot \tau^{k-1} + \frac{\delta^3 \tau^{2\gamma + k-2}}{\tau^2} \cdot \tau^{k-1} ~\D\tau\\
	\lesssim \delta^3 \sigma^{2\gamma + 2(k-2) - 1/2} + \delta^3 \sigma^{\gamma + 2(k-2)} + \delta^3 \sigma^{2\gamma + 2(k-2)} \leq \delta^3 \sigma^{2\gamma + 2(k-2)}. 
\end{multline}

With these estimates, the bootstrap closes provided $\delta, \epsilon$ are taken sufficiently small. 

We close our discussion with a couple of remarks. 
\begin{rmk}
	One interesting aspect of our argument is that the semilinear nonlinearities seem to allow closing the bootstrap using only a $\log(\tau)$ loss instead of $\tau^\gamma$. This is seemingly in contradiction to the discussion in Section \ref{sect:semilinear}, where $\tau^{\gamma}$ losses seems to be necessary when the dimension $d = 3,4$. The explanation for this is that in our semilinear analyses we did \emph{not} separate out the privileged direction $L^1\phi$ as having better decay properties. Had we also isolated the direction $L^1\phi$ and run the argument with separate energies for generic derivatives and derivatives with at least one $L^1$ vector field, we would find also that it is possible to close the argument with merely a $\log(\tau)$ loss at energy level $d-1$, with a further $\tau$ loss with each additional derivative, analogously to the cases where $d \geq 5$. 

	As discussed at the start of Section \ref{sect:commutedeq}, one would see additional losses for the full quasilinear problem were one not to separate out the better direction $L^1$. 
	This is reflected in the fact that the part where we \emph{required} the $\tau^{\gamma}$ loss in place of a mere log-loss occurs in Section \ref{ssect:qno}, where we considered the effects of the commutator term $[X,\Box_g]\psi$; note of course that the commutator term vanishes for our semilinear model problem. 
\end{rmk}

\begin{rmk}
	One may ask whether the higher energy growth is associated to the \emph{blow-up at infinity} described by Alinhac \cite{Alinha2003}, and which seems generic for wave equations with weak-null quasilinearities \cite{Lindbl2008, LinRod2005,DenPus2018}. This seems \emph{not} to be the case for several reasons. 
	First among the reasons is that we observe the same higher energy growth even for the semilinear model considered in Section \ref{sect:semilinear}. Additionally, our energy growth is not very severe; when translated back to $L^\infty$ estimates of the coordinate derivatives, we still observe decay (though at a reduced rate compared to what is available for the linear wave equation). Finally, examining the leading order correction of the quasilinear metric is given with the coefficients $\phi \bgf'' \D{u} \otimes \D{u}$. The localization by the $\bgf''$ term means that the slowly decaying coefficients are supported away from future time-like infinity. 
	This appears in contrast to the known manifestations of blow-up at infinity where the null structure of the dynamical metric is significantly different from the Minkowskian one near future time-like infinity. 
\end{rmk}

\appendix
\section{Various computations}%
\subsection{Computations supporting Section \ref*{sect:perturbedsystem}}\label{sect:app:ps}

\subsubsection{Verification of (\ref*{eq:boxgexpr})}
\begin{align*}
	\Box_\gmetr \psi &= \frac{1}{\sqrt{|\gmetr|}} \partial_{\mu} \big( \sqrt{|\gmetr|} \gmetr^{\mu\nu} \partial_\nu \psi \big)\\
	& = \frac{1}{\sqrt{|\gmetr|}} \partial_{\mu} \big( \sqrt{|\gmetr|} \gmetrf^{\mu\nu} \partial_\nu \psi \big) - \frac{1}{\sqrt{|\gmetr|}} \partial_{\mu} \big( \frac{1}{\sqrt{|\gmetr|}} \gmetrf^{\mu\nu} \partial_\nu \phi \cdot \gmetrf(\D\phi,\D\psi) \big)\\
	&= \frac1{2|\gmetr|} \gmetrf( \D|\gmetr|, \D\psi) + \Box_\mink \psi + 2 \partial_{\ub{u}} ( \phi \bgf'' \partial_{\ub{u}} \psi) \\
	& \qquad - \frac{1}{|\gmetr|} \cdot \underbrace{\sqrt{|\gmetr|} \partial_{\mu} \big( \frac{1}{\sqrt{|\gmetr|}} \gmetrf^{\mu\nu} \partial_\nu \phi \big)}_{= \bgf''(\phi_{\ub{u}})^2} \cdot \gmetrf(\D\phi,\D\psi) - \frac{1}{|\gmetr|} \gmetrf\big(\D\phi, \D\big( \gmetrf(\D\phi, \D\psi)\big) \big)
\end{align*}

\subsubsection{Verification of (\ref*{eq:geomeqphiX})} We start with 
\[ \lieD_X \Big( \sqrt{|g|} \partial_\mu \frac{\gmetrf^{\mu\nu}\partial_\nu\phi}{\sqrt{|\gmetr|}} \Big) = X\big( \bgf'' (\phi_{\ub{u}})^2\big).\]
Now, the left hand side can be written as
\begin{multline*}
	[X, \Box_\mink] \phi + \boxed{\Box_\mink X\phi} + 2 [X, \partial_{\ub{u}}](\phi \bgf'' \partial_{\ub{u}}\phi) \\
	+ 2 \partial_{\ub{u}}\big( X(\phi \bgf'') \partial_{\ub{u}}\phi\big) + 2 \partial_{\ub{u}} \big( \phi\bgf'' [X,\partial_{\ub{u}}]\phi\big) + \boxed{2 \partial_{\ub{u}} ( \phi \bgf'' \partial_{\ub{u}} X\phi )} \\
	+ \frac{1}{2 |\gmetr|^2} X(|\gmetr|) \gmetrf(\D\phi, \D|\gmetr|) - \frac{1}{2|\gmetr|} \lieD_X(\gmetrf^{-1})(\D\phi, \D|\gmetr|) - \frac{1}{2|\gmetr|} \gmetrf(\D X\phi, \D|\gmetr|) \\
	- \frac{1}{2|\gmetr|} \gmetrf\big( \D\phi, \D\big(\lieD_X(\gmetrf^{-1})(\D\phi, \D\phi) \big)\big)
	- \boxed{\frac{1}{|\gmetr|} \gmetrf\big( \D\phi, \D\gmetrf(\D\phi, \D X\phi)\big)}.
\end{multline*}
Throughout we have used the Leibniz rule for Lie differentiation with respect to tensor contractions, as well as the fact that Lie derivatives commute with exterior differentiation. 
The boxed terms, we notice, are identical to the principal terms in $\Box_\gmetr \psi$ if we set $\psi = X\phi$. 
The formula \eqref{eq:geomeqphiX} follows by rearranging the terms. 

\subsubsection{Control of $TT\phi$ terms}
As the null structure that we require can all be recovered as discussed in Remark \ref{rmk:ncvestige}, for the control of the $TT\phi$ terms in terms of other $\algdo^{2,1}_{*}$ terms we do not need to be too precise with the weights. 
Starting from the equation 
\[ \Box_\mink \phi + 2 \bgf'' \partial_{\ub{u}}(\phi \phi_{\ub{u}}) - \frac{1}{2 |\gmetr|} \gmetrf( \D\phi, \D|\gmetr|) = \bgf''(\phi_{\ub{u}})^2 \]
we first observe
\[ \Box_{\mink} \phi = - \frac{\tau^2}{(y^0)^2} TT \phi - \underbrace{\frac{d}{y^0} T\phi + \frac{1}{(y^0)^2} \sum_{i = 1}^d L^i L^i \phi - y^i L^i T \phi - y^i T L^i \phi}_{= \algdo^{2,1}_2 \phi}.\]
Additionally, the quadratic terms 
\begin{multline*}
	2 \bgf'' \partial_{\ub{u}}(\phi \phi_{\ub{u}}) - \bgf'' (\phi_{\ub{u}})^2 = \pwWf_0 \swtf_2 (L^1 \phi + T \phi)^2 \\
	+\phi \pwWf_0 \swtf_2 (L^1 L^1 \phi + L^1 T \phi + T L^1 \phi + TT \phi + L^1\phi + T \phi ).
\end{multline*}
The cubic and higher order terms are captured schematically by
\begin{multline}
	\gmetrf(\D\phi, \D|\gmetr|) = \swtf_2 [ \tau^2 (\algdo_1^{1,1}\phi)^2 + (\algdo_0^{1,0}\phi)^2 + \phi \pwWf_0 (\algdo_1^{1,1}\phi) (L^1\phi + T\phi)] TT\phi \\
	+ [(\algdo_1^{1,1}\phi)^2 + \phi \pwWf_0 (\algdo_1^{1,1}\phi)^2   ] \algdo_2^{2,1}\phi \\
	+\swtf_1 (\algdo_1^{1,1}\phi)^3 + \pwWf_0 (\algdo_1^{1,1}\phi)^4 + \phi \pwWf_0 (\algdo_1^{1,1}\phi)^3 + \phi \pwWf_1 \swtf_1 (\algdo_1^{1,1}\phi)^3, 
\end{multline}
where we took care to isolate the terms with $TT\phi$ from other second derivatives. 

This means that we can re-write
\begin{multline}\label{eq:TTcomp}
	TT\phi = \frac{1}{\mathfrak{c}_{TT}} \bigl( \smtf \algdo_2^{2,1}\phi + \pwWf_0 \swtf_2 (L^1\phi +T\phi)^2  + \pwWf_0 (\algdo_1^{1,1}\phi)^4 \\
	+ \phi\pwWf_0 \swtf_1 (\algdo_1^{1,1}\phi) + \smtf (\algdo_1^{1,1}\phi)^3 + (1+\phi \pwWf_1) \swtf_1 (\algdo_1^{1,1}\phi)^3 \bigr) 
\end{multline}
where
\begin{multline}\label{eq:def:ctt}
	\mathfrak{c}_{TT} \eqdef \frac{\tau^2}{(y^0)^{2}} \bigl[ 1 + (\algdo_1^{1,1}\phi)^2 + \frac{1}{\tau^{2}}(\algdo_0^{1,0}\phi)^2\\
	+ \frac{1}{\tau^2} (\algdo_{0}^{0,0}\phi) \pwWf_0 (1 + (\algdo_1^{1,1}\phi)(L^1\phi + T\phi)) \bigr].
\end{multline}

\begin{rmk}\label{rmk:nogrowth:TT}
	Note that none of the $\smtf$ factors in \eqref{eq:TTcomp} include any $\algdo_0^{1,0}\phi$ dependence. 
\end{rmk}

\subsubsection{Verification of (\ref*{eq:geomeqL1phi})} \label{appen:HO1}

We note the following very rough estimate for the cubic terms
\[
	\mink^{-1}(\D\psi_1, \D\psi_2) = \algdo_1^{1,1}\psi_1 \algdo_1^{1,1}\psi_2
\]
and hence
\[ 
	\gmetrf(\D\psi_1, \D\psi_2) = \algdo_1^{1,1} \psi_1 \algdo_1^{1,1}\psi_2(1  + \phi \pwWf_0).
\]
We also have that
\begin{equation}
	\lieD_{L^1}(\gmetrf^{-1})(\D\psi_1, \D\psi_2) = \algdo_0^{1,0}\phi \pwWf_0 \algdo_1^{1,1}\psi_1 \algdo_1^{1,1}\psi_2.
\end{equation}
So all the higher-order, non-boxed terms in \eqref{eq:geomeqphiX} can be captured by the sum
\begin{multline}
	\mathrm{HO}_1 \eqdef \smtf (\algdo_1^{1,1}\phi)(\algdo_1^{1,1}L^1\phi) (\algdo_2^{2,2}\phi) + \smtf \pwWf_0 (\algdo_0^{1,0}\phi)(\algdo_1^{1,1}\phi)^2 (\algdo_2^{2,2}\phi)\\
	+ \smtf \pwWf_0 (\algdo_1^{1,1}\phi)^3(\algdo_1^{2,1}\phi) + \smtf\pwWf_0 (\algdo_1^{1,1}\phi)^5(\algdo_2^{2,2}\phi) + \smtf \pwWf_0 (\algdo_0^{1,0}\phi)(\algdo_1^{1,1}\phi)^4\\
	+ \smtf \phi \swtf_1\pwWf_1 (\algdo_1^{1,1}\phi)^2(\algdo_1^{2,1}\phi) +   \smtf \phi^2 \swtf_2\pwWf_2 (\algdo_1^{1,1}\phi)^4(\algdo_1^{2,1}\phi) \\
	+ \smtf \phi \swtf_1\pwWf_1 (\algdo_1^{1,1}\phi)^5(\algdo_2^{2,2}\phi) + \smtf \swtf_1 \pwWf_1 (\algdo_0^{1,0}\phi)(\algdo_1^{1,1}\phi)^3.
\end{multline}
\begin{rmk}\label{app:rmk:hoDN}
	The term $\smtf (\algdo_1^{1,1}\phi)(\algdo_1^{1,1}L^1\phi)(\algdo_2^{2,2}\phi)$ stands out in the expression of $\mathrm{HO}_1$: it is both the only cubic term (all other terms are at least quartic in the unknowns) and the only term that is not explicitly multiplied by a factor of $\pwWf_*$. In fact, this term is the only nonlinearity that would remain when $\bgf \equiv 0$, where the equations reduce to the small-data scenario studied by Lindblad \cite{Lindbl2004}, and the nonlinearity is of the \emph{double null} form $\mink^{-1}(\D\phi,\D(\mink^{-1}(\D\phi,\D\phi)))$. 

	We note that instead of writing $\algdo_1^{2,1}\phi$ we have chosen to write $\algdo_1^{1,1}L^1\phi$. This is deliberate in order to allow us to exploit certain improvements of decay for the $L^1\phi$ derivatives. 
\end{rmk}

We concentrate on the boxed, quadratic terms in \eqref{eq:geomeqphiX} next. 
For these terms we need the additional null structure as seen in \eqref{eq:semiLbd}, and we write, noting that $[L^1,\partial_{\ub{u}}] = -\partial_{\ub{u}}$, the following schematic decompositions for the quadratic terms:
\begin{align*}
	L^1(\bgf'' (\phi_{\ub{u}})^2) &= \pwWf_0 \swtf_2 (L^1\phi + T\phi)(L^1 L^1\phi + T L^1\phi) + \pwWf_0 \swtf_2 (L^1\phi + T\phi)^2,\\
	\partial_{\ub{u}}(\phi \bgf'' \phi_{\ub{u}}) &= \pwWf_0 \swtf_2 (L^1\phi + T\phi)^2 \\
		&+\pwWf_0 \phi \swtf_2 (L^1\phi + T\phi) + \pwWf_0 \phi \swtf_2 (L^1 L^1 \phi + T L^1 \phi + TT\phi),\\
	\partial_{\ub{u}}(L^1(\phi \bgf'') \phi_{\ub{u}}) &= \pwWf_0 (\phi + L^1\phi) \swtf_2 (L^1 \phi + T\phi)\\ 
								  & +\pwWf_0 (\phi + L^1\phi) \swtf_2 (L^1 L^1\phi + TL^1\phi + TT\phi) \\
								  &+ \pwWf_0 \swtf_2 (L^1\phi + T\phi)^2 
								  + \pwWf_0 \swtf_2 (L^1\phi + T\phi) (L^1 L^1\phi + T L^1\phi).
\end{align*}
So we can summarize the quadratic nonlinearities schematically as
\begin{multline}\label{eq:def:qn1}
	\mathrm{QN}_1 = \pwWf_0 \swtf_2 \cdot \Bigl[ (L^1\phi + T\phi)^2 + T\phi(L^1L^1\phi + TL^1\phi) \\
	+ \phi (L^1\phi + T\phi) + (\phi + L^1\phi) (L^1 L^1 \phi + T L^1\phi + T T\phi) \Bigr].
\end{multline}

\subsubsection{Verification of (\ref*{eq:geomeqLiphi})} \label{appen:HOi}

In the case where $X = L^i$ for $i \neq 1$, we have that
\begin{equation}
	\lieD_{L^i}(\gmetrf^{-1})(\D\psi_1, \D\psi_2) = \algdo_0^{1,0}\phi \pwWf_1 \algdo_1^{1,1}\psi_1 \algdo_1^{1,1}\psi_2.
\end{equation}
One can check that the higher-order, non-boxed terms in \eqref{eq:geomeqphiX} are now captured by 
\begin{multline} 
	\mathrm{HO}_i = \mathrm{HO}_1 + \smtf (\algdo_1^{1,1}\phi)(\algdo_1^{2,1}\phi)(\algdo_2^{2,2}\phi) \\
	+  \smtf \phi \pwWf_1 (\algdo_1^{1,1}\phi)^2 (\algdo_2^{2,2}\phi) 
		+ \smtf \phi \swtf_1 \pwWf_2 (\algdo_1^{1,1}\phi)^3 + \smtf \pwWf_1 (\algdo_1^{1,1}\phi)^4.
\end{multline}
The added terms are now the pure cubic term which now can include $L$ derivatives in all directions, and additional quartic terms which arises from $X$ hitting $\bgf''$ which generates a $\pwWf_1$ instead of $\pwWf_0$.  

The quadratic parts of the nonlinearity can also be expanded schematically. The computations are as follows:
\begin{align*}
	L^i(\bgf'' (\phi_{\ub{u}})^2) &= \pwWf_0 \swtf_2 (L^1 \phi + T\phi)(L^i L^1\phi + L^i T\phi) + \pwWf_1 \swtf_2 (L^1\phi + T\phi)^2,\\
	\frac{1}{y^0} (L^i - y^i T) (\phi \bgf'' \phi_{\ub{u}}) &= \pwWf_0 \algdo_1^{1,1}\phi \swtf_1 (L^1\phi + T\phi) + \pwWf_1 \swtf_2 \phi (L^1 \phi + T\phi) \\
								& + \pwWf_0 \phi \swtf_1 (\algdo_1^{1,1}L^1\phi + \algdo_1^{1,1} T\phi),\\
	\partial_{\ub{u}}(L^i(\phi \bgf'') \phi_{\ub{u}}) &= \pwWf_0 \swtf_2(L^1\phi + T\phi)(L^1 L^i\phi + T L^i\phi) + \pwWf_1 \swtf_2 (L^1\phi+T\phi)^2\\
	&+ \pwWf_0 \swtf_2 L^i\phi (L^1 L^1 \phi + T L^1\phi + TT\phi + L^1\phi + T\phi)\\
	&+ \pwWf_1 \swtf_2 \phi (L^1 L^1 \phi + T L^1\phi + T T\phi + L^1\phi + T\phi),\\
	\partial_{\ub{u}}(\phi \bgf'' \frac{1}{y^0}(L^i - y^i T)\phi) &= \pwWf_0 \swtf_1 (L^1\phi + T\phi) \algdo_1^{1,1}\phi + \pwWf_0 \swtf_1 \phi (L^1 \algdo_1^{1,1}\phi + T \algdo_1^{1,1}\phi).
\end{align*}
Thus we can collect the quadratic nonlinearities using the schematic expression
\begin{multline}\label{eq:def:qni}
	\mathrm{QN}_i = \pwWf_0 \swtf_1(\phi + L^1\phi + T \phi)(\algdo_1^{1,1} L^1\phi + \algdo_2^{2,2}\phi + \algdo_1^{1,1}\phi)\\
	+ \swtf_2 (\pwWf_0 L^i\phi + \pwWf_1 \phi)(L^1 L^1\phi + TL^1\phi + TT\phi + L^1\phi + T\phi) \\
	+ \pwWf_1 \swtf_2 (\phi + L^1\phi + T\phi)(L^1\phi + T\phi).
\end{multline}

\subsection{Computations supporting \ref*{sect:commutatorrels}}\label{sect:app:cr}

Observe that, expanding using the standard formula for the Laplace-Beltrami operator and \eqref{eq:InvMet}, we obtain
\begin{multline*}
	\Box_\gmetr \psi = \Box_m\psi + \frac12 \frac{1}{|\gmetr|} \gmetr^{-1}(\D|\gmetr|, \D\psi) + 2\partial_{\ub{u}}(\phi \bgf''\partial_{\ub{u}}\psi) \\
	- \frac{1}{\sqrt{|\gmetr|}} \gmetrf^{-1}(\D\phi, \D(\frac{1}{\sqrt{|\gmetr|}} \gmetrf^{-1}(\D\phi, \D\psi))) - \frac{1}{|\gmetr|} \bgf'' (\phi_{\ub{u}})^2 \gmetrf^{-1}(\D\phi,\D\psi).
\end{multline*}
This implies
\begin{multline*}
	[X, \Box_\gmetr]\psi = \frac12 \gmetr^{-1}(\D(X\ln|\gmetr|), \D\psi) + \frac12 \lieD_X\gmetr^{-1}(\D\ln|\gmetr|, \D\psi) \\
	+ \boxed{2 [X,\partial_{\ub{u}}](\phi \bgf'' \partial_{\ub{u}}\psi)} + \boxed{2 \partial_{\ub{u}}(X(\phi \bgf'') \partial_{\ub{u}}\psi)} + \boxed{2 \partial_{\ub{u}} ( \phi\bgf'' [X, \partial_{\ub{u}}]\psi)} \\
	+ \frac12 |\gmetr|^{-2/3} X(|\gmetr|) \gmetrf^{-1}(\D\phi, \D(\frac{1}{\sqrt{|\gmetr|}} \gmetrf^{-1}(\D\phi,\D\psi))) 
	- \frac{1}{\sqrt{|\gmetr|}} \lieD_X(\gmetrf^{-1})(\D\phi,  \D(\frac{1}{\sqrt{|\gmetr|}} \gmetrf^{-1}(\D\phi,\D\psi)))\\
	- \frac{1}{\sqrt{|\gmetr|}} \gmetrf^{-1}(\D(X\phi),        \D(\frac{1}{\sqrt{|\gmetr|}} \gmetrf^{-1}(\D\phi,\D\psi)))
	+ \frac12 \frac{1}{\sqrt{|\gmetr|}} \gmetrf^{-1}(\D\phi, \D( \frac{X|\gmetr|}{|\gmetr|^{3/2}} \gmetrf^{-1}(\D\phi,\D\psi))) \\
	- \frac{1}{\sqrt{|\gmetr|}} \gmetrf^{-1}(\D\phi, \D(\frac{1}{\sqrt{|\gmetr|}} \lieD_X(\gmetrf^{-1})(\D\phi, \D\psi)))
	- \frac{1}{\sqrt{|\gmetr|}} \gmetrf^{-1}(\D\phi, \D(\frac{1}{\sqrt{|\gmetr|}} \gmetrf^{-1}(\D(X\phi),\D\psi)))\\
	- X(\frac{1}{|\gmetr|} \bgf'' (\phi_{\ub{u}})^2) \gmetrf^{-1}(\D\phi, \D\psi) 
	- \frac{1}{|\gmetr|} \bgf'' (\phi_{\ub{u}})^2 \lieD_X(\gmetrf^{-1})(\D\phi, \D\psi) \\
	- \frac{1}{|\gmetr|} \bgf'' (\phi_{\ub{u}})^2 \gmetrf^{-1}(\D(X\phi), \D\psi).
\end{multline*}
Except for the three boxed terms, which are linear in both $\phi$ and $\psi$, all remaining terms are at least quadratic in $\phi$.

The terms that are quadratic and above in $\phi$ are generally harmless. 
We will use the following rough estimate for the quadratic form $\gmetr^{-1}$:
\[
	\gmetr^{-1}(\D\psi_1, \D\psi_2) =\algdo_1^{1,1}\psi_1 \algdo_1^{1,1}\psi_2 (1 + \phi \pwWf_0 + \smtf (\algdo_1^{1,1}\phi)^2).
\]
For $X = L^i$, where $i= 1, \ldots, 3$, the terms quadratic and above in $\phi$ can be 
schematically captured by the following collection of terms:
\begin{multline}
	(X\smtf)\cdot \bigl[ (\algdo_1^{1,1}\phi)(\algdo_2^{2,2}\phi)(\algdo_1^{1,1}\psi) + (\algdo_1^{1,1}\phi)^2 (\algdo_2^{2,2}\psi) \\
	+ \pwWf_0 (\algdo_1^{1,1}\phi)^3 (\algdo_1^{1,1}\psi) + \pwWf_1 \swtf_1\phi (\algdo_1^{1,1}\phi)^2 (\algdo_1^{1,1}\psi)\bigr]\\
	+ \smtf \cdot \bigl[ (\algdo_1^{2,1}\phi)(\algdo_2^{2,2}\phi)(\algdo_1^{1,1}\psi) + (\algdo_1^{1,1}\phi)(\algdo_2^{3,2}\phi)(\algdo_1^{1,1}\psi) \\
		+ (\algdo_1^{1,1}\phi)(\algdo_1^{2,1}\phi)(\algdo_2^{2,2}\psi) + \pwWf_0 (\algdo_1^{1,1}\phi)^2(\algdo_1^{2,1}\phi) (\algdo_1^{1,1}\psi) \\
		+ \pwWf_1 (\algdo_1^{1,1}\phi)^3 (\algdo_1^{1,1}\psi) + \pwWf_1 \swtf_1\phi (\algdo_1^{1,1}\phi_)(\algdo_1^{2,1}\phi)(\algdo_1^{1,1}\psi) \\
	+ \pwWf_1 \swtf_1 (\algdo_0^{1,0}\phi)(\algdo_1^{1,1}\phi)^2 (\algdo_1^{1,1}\psi) + \pwWf_2 \swtf_1 \phi (\algdo_1^{1,1}\phi)^2 (\algdo_1^{1,1}\phi)\bigr]
\end{multline}
where since $X = L^i$ we have that 
\begin{equation}
	X\smtf = \smtf \cdot \bigl( \algdo_0^{1,0}\phi \pwWf_0 + \phi \pwWf_1 + \algdo_0^{2,0}\phi + \algdo_1^{2,1}\phi \bigr).
\end{equation}
We note, as before, all terms involving $\pwWf_*$ weights are at least cubic in $\phi$. 

The terms that are linear in $\phi$ in the commutator can also be expanded. For higher level commutations we do not need to separate between $L^1$ and $L^i$ for $i \neq 1$. So we can just write $[X,\partial_{\ub{u}}] = \algdo_1^{1,1}$, which allows us to capture the relevant terms by
\begin{align*}
	[X, \partial_{\ub{u}}](\phi \bgf'' \psi_{\ub{u}}) &= \pwWf_0 \algdo_1^{1,1}\phi \swtf_1 (L^1\psi + T\psi) + \pwWf_1 \swtf_2\phi (L^1\psi + T\psi) \\
							   &+ \pwWf_0 \phi \swtf_1 \algdo_1^{1,1}(L^1 \psi + T\psi),\\
	\partial_{\ub{u}} (X(\phi\bgf'')\psi_{\ub{u}}) &= (\pwWf_1 \phi + \pwWf_0 \algdo_0^{1,0}\phi) \swtf_2 (L^1 \psi + T\psi + L^1 L^1\psi + T L^1\psi + TT\psi) \\
						       &+ \pwWf_0 \algdo_1^{1,1}\phi \swtf_1 (L^1\psi + T\psi) + \pwWf_0 \swtf_2 \algdo_0^{1,0}(L^1\phi + T\phi) (L^1\psi + T\psi) \\
						       &+ \pwWf_1 \swtf_2 (L^1 \phi + T\phi) (L^1\psi + T\psi) + \pwWf_1 \swtf_2\phi (L^1\psi + T\psi) ,\\
	\partial_{\ub{u}}(\phi \bgf'' [X,\partial_{\ub{u}}]\psi) &= \pwWf_0 \swtf_1 (L^1\phi + T\phi) \algdo_1^{1,1}\psi + \pwWf_0 \swtf_1 \phi (L^1 \algdo_1^{1,1}\psi + T\algdo_1^{1,1}\psi). 
\end{align*}
Using the commutation relations of Proposition \ref{prop:diffalg} we can summarize these terms by 
\begin{multline}
	\pwWf_0 \swtf_1 (\algdo_1^{1,1}\phi)(L^1\psi + T\psi) + \pwWf_0 \swtf_1 (\phi + L^1\phi + T\phi) (\algdo_1^{1,1}\psi)\\
	+ \pwWf_0 \swtf_1 (\algdo_0^{1,0}\phi) (\algdo_1^{1,1} L^1\psi + \algdo_1^{1,1} T\psi) + \pwWf_0 \swtf_2 (\algdo_0^{1,0} L^1 \phi + \algdo_0^{1,0} T\phi)(L^1\psi + T\psi) \\
	+\pwWf_1 \swtf_2 (\phi + L^1\phi + T\phi) (L^1\psi + T\psi) + \pwWf_1 \swtf_2 \phi (L^1 L^1\psi + T L^1 \psi + T T\psi) 
\end{multline}
which have similar structure to the terms appearing in $\mathrm{QN}_1$ and $\mathrm{QN}_i$ above. 

\section{List of symbols and notations}
\begin{eqlist}
	\item[$\mink$] Minkowski metric with signature $(-1,1,\ldots,1)$.
	\item[$\gmetrf$] Linear part of the dynamical metric; see \eqref{eq:dynmet1} in \S\ref{sect:perturb}.
	\item[$\gmetr$] Dynamical metric; see \eqref{eq:dynmet} in \S\ref{sect:perturb}.
	\item[$\bgf$] Plane-wave background; see \eqref{eq:pwbkgd} in \S\ref{sect:pwbkgd}.
	\longitem[$u, \ub{u}, \hat{x}$] Null coordinates adapted to plane-wave background; see \eqref{eq:changevar} and the discussion after \eqref{eq:pwbkgd2} in \S\ref{sect:pwbkgd}.
	\item[$y^i$] Rectangular coordinates adapted to plane-wave background; see \eqref{eq:ycoord} in \S\ref{sect:wvfalg}.
	\item[$\mathcal{I}^+$] Forward light cone; see start of \S\ref{sect:globsob}.
	\longitem[$\tau,\rho, \Sigma_\tau$] Hyperboloidal foliation and related parameters; see \eqref{eq:def:tau} and Notation \ref{not:def:tau} in \S\ref{sect:globsob}.
	\item[$T, L^i$] Vector fields; see \eqref{eq:def:T}, \eqref{eq:def:L1}, and \eqref{eq:def:Li} in \S\ref{sect:globsob}. 
	\item[$\seten$] Stress-energy tensors; see \eqref{eq:def:seten} in \S\ref{sect:globsob}.
	\item[$\Energy$] Background energy integrals; see \eqref{eq:def:eng1} and \eqref{eq:def:eng2} in \S\ref{sect:basest}. 
	\item[$\sweight_*$] Weight functions; see Definition \ref{defn:sweight} in \S\ref{sect:wvfalg}.
	\item[$\swtf_*$] Elements of $\sweight_*$; see Notation \ref{notn:PfWf} in \S\ref{sect:perturbedsystem}.
	\item[$\pwWf_*$] Plane-wave like weights; see Notation \ref{notn:PfWf} in \S\ref{sect:perturbedsystem}.
	\item[$\algcomm_*$] Weighted commutator algebra; see discussion surrounding Proposition \ref{prop:commalg} in \S\ref{sect:wvfalg}. 
	\item[$\algdiff_*^{*,*}$] Weighted differential operators; see discussion surrounding \eqref{eq:def:algdiff:term} in \S\ref{sect:wvfalg}, as well as Definition \ref{defn:algdiff:ind}. 
	\item[$\algdo_*^{*,*}$] Elements of $\algdiff_*^{*,*}$; see Notation \ref{notn:algdo} in \S\ref{sect:perturbedsystem}.
	\item[$\smtf$] Smooth functions representing bounded terms; see Notation \ref{notn:smtf} in \S\ref{sect:perturbedsystem}. 
\end{eqlist}

\bibliographystyle{amsalpha}
\bibliography{vmcref.bib}

\end{document}